\documentclass{amsart}

\usepackage{paralist}
\usepackage{color}
\usepackage{mathrsfs}
\usepackage{amssymb}
\usepackage{amsmath}
\usepackage{amsfonts}
\usepackage{stmaryrd}
\usepackage{graphics}
\usepackage[shortlabels]{enumitem}
\setlist{
  font=\normalfont
}
\usepackage{mathdots}
\usepackage[final]{showlabels}
\usepackage{float}
\usepackage{changes}
\usepackage{phoenician}
\reversemarginpar
\usepackage{tikzit}

\tikzstyle{label}=[fill=black, draw=black, shape=circle, scale=0.3]
\tikzstyle{green label}=[fill={rgb,255: red,0; green,128; blue,128}, draw=none, shape=circle, scale=0.3]
\tikzstyle{red label}=[fill={rgb,255: red,191; green,0; blue,64}, draw=none, shape=circle, scale=0.3]
\tikzstyle{SIGN}=[fill=yellow, draw=black, shape=rectangle]

\tikzstyle{green line}=[-, draw={rgb,255: red,0; green,128; blue,128}]
\tikzstyle{punteada}=[-, dashed]
\tikzstyle{red line}=[-, draw={rgb,255: red,191; green,0; blue,64}]
\tikzstyle{level}=[-, draw=blue]

\usepackage[backend=biber]{biblatex}
\bibliography{higherWalksRefs.bib}

\usepackage{calc}

\makeatletter
\def\textSq#1{%
\begingroup%
\setlength{\fboxsep}{0.3ex}%
\setbox1=\hbox{#1}%
\setlength{\@tempdima}{\maxof{\wd1}{\ht1+\dp1}}%
\setlength{\@tempdimb}{(\@tempdima-\ht1+\dp1)/2}%
\raise-\@tempdimb\hbox{\fbox{\vbox to \@tempdima{%
  \vfil\hbox to \@tempdima{\hfil\copy1\hfil}\vfil}}}%
\endgroup%
}
\def\Sq#1{\textSq{\ensuremath{#1}}}%
\makeatother

\newtheorem*{thma}{Theorem~A}
\newtheorem*{thmb}{Theorem~B}
\newtheorem*{thmc}{Theorem~C}

\newtheorem{theorem}{Theorem}
\numberwithin{theorem}{section}
\newtheorem{lemma}[theorem]{Lemma}
\newtheorem{proposition}[theorem]{Proposition}
\newtheorem{corollary}[theorem]{Corollary}
\newtheorem{claim}{Claim}[theorem]

\newtheorem{fact}[theorem]{Fact}

\theoremstyle{definition}
\newtheorem{definition}[theorem]{Definition}
\newtheorem{remark}[theorem]{Remark}
\newtheorem{notation}[theorem]{Notation}

\usepackage[pdftex]{hyperref}
\hypersetup{
    colorlinks=true, %
    linktoc=all,     %
    linkcolor=blue,  %
    allcolors=blue,
}

\usepackage[framemethod=tikz]{mdframed}

\newcommand{\seq}[1]{{\langle{#1}\rangle}}
\newcommand{\restr}{\mathord{\upharpoonright}}
\renewcommand{\subset}{\subseteq}
\renewcommand{\epsilon}{\varepsilon}
\renewcommand{\L}{\mathrm{L}}

\newcommand{\cf}{\mathrm{cf}}
\newcommand{\dom}{\mathrm{dom}}
\newcommand{\bb}{\mathbb}

\newcommand{\otp}{\mathrm{otp}}

\newcommand{\nacc}{\mathrm{nacc}}
\newcommand{\acc}{\mathrm{acc}}

\newcommand{\ssup}{\mathrm{ssup}}

\newcommand{\mc}{\mathcal}

\newcommand{\sub}{\mathrm{sub}}
\newcommand{\ra}{\rightarrow}

\newcommand{\Coll}{\mathrm{Coll}}

\newcommand{\Q}{\bb{Q}}
\newcommand{\Z}{\mathbb{Z}}
\renewcommand{\P}{\bb{P}}

\newcommand{\w}{\omega}
\newcommand{\C}{\mathcal{C}}
\newcommand{\Tr}{\mathrm{Tr}}

\newcommand{\bS}{\mathtt{b}S}

\newcommand{\ZFC}{\sf ZFC}

\newcommand{\resh}{{\phncfamily\textphnc{r}}}
\newcommand{\sgn}{\mathrm{sgn}}

\title{Higher walks and squares}
\author{Chris Lambie-Hanson}
\address[Lambie-Hanson]{
Institute of Mathematics, 
Czech Academy of Sciences, 
{\v Z}itn{\'a} 25, Prague 1, 
115 67, Czech Republic
}
\email{lambiehanson@math.cas.cz}
\urladdr{https://clambiehanson.github.io}

\author{Pedro Marun}
\address[Pedro Marun]{
Institute of Mathematics, 
Czech Academy of Sciences, 
{\v Z}itn{\'a} 25, Prague 1, 
115 67, Czech Republic
}
\email{marun@math.cas.cz}
\urladdr{https://pedromarun.github.io}
\keywords{walks on ordinals, square principles, coherence, cohomology}
\subjclass[2020]{03E05, 03E35, 03E55, 03E10}
\thanks{We thank Jeffrey Bergfalk for a number of valuable conversations. Both authors were supported by the Czech Academy of Sciences (RVO 67985840) and the GA\v{C}R project 23-04683S.}

\begin{document}
\begin{abstract}
We continue the development of the theory of higher dimensional walks on ordinals began recently by Bergfalk. In particular we identify natural coherence conditions on higher dimensional $C$-sequences that entail coherence of the resultant higher rho-functions. We also introduce various higher square principles by adding non-triviality conditions to these coherent higher $C$-sequences and investigate basic properties of said square principles. For example, in analogy with the classical case, we prove that these higher square principles abound in the constructible universe but can be forced to fail, modulo large cardinals. Finally, we prove that certain higher rho-functions obtained by walking along higher square sequences exhibit non-triviality in addition to coherence. In particular, it follows that higher square principles on a cardinal $\lambda$ entail certain non-vanishing \v{C}ech cohomology groups for $\lambda$ considered with the order topology.
\end{abstract}
\maketitle

\section{Introduction}

One of the most important tools in the contemporary study of combinatorial set theory 
is Todorcevic's method of \emph{walks on ordinals}. This method, introduced in 
\cite{todorcevic_countable} (see \cite{todorcevic_walks_book} for a book-length treatment), 
has proven transformatively useful in a variety of contexts, and is particularly efficacious 
at the level of $\omega_1$, the first uncountable cardinal.

We will define the walks machinery in more detail below, but for now let us give a general overview. 
To begin, one fixes an uncountable cardinal $\lambda$ and a $C$-sequence 
$\mc{C} = \langle C_\alpha \mid \alpha < \lambda \rangle$.\footnote{See Sections \ref{section: prelim} and \ref{section: squares} below for any undefined notions in this introduction.} 
This then provides a uniform 
method of ``walking'' down from an ordinal $\beta < \lambda$ to any ordinal $\alpha \leq \beta$. 
Namely, we define an \emph{upper trace function} $\Tr : [\lambda]^2 \ra [\lambda]^{<\omega}$ 
recursively by setting, for each $\alpha \leq \beta < 
\lambda$,
\[
  \Tr(\alpha,\beta) = \{\beta\} \cup \Tr(\alpha, \min(C_\beta \setminus \alpha)),
\]
with a boundary condition specifying that $\Tr(\alpha,\alpha) = \{\alpha\}$ for all 
$\alpha < \lambda$. 

This trace function, and its interaction with the $C$-sequence $\mc{C}$, can then be used 
to uncover a variety of interesting and nontrivial combinatorial behavior at the cardinal 
$\lambda$. In this paper, we will particularly be interested in the existence of families of 
functions that are simultaneously \emph{coherent} and \emph{nontrivial}. As the definitions will 
make clear, coherence is a local form of triviality, so nontrivial coherent families will be 
witnesses to \emph{incompactness phenomena}, closely linked to cohomological considerations. 
In fact, (equivalence classes of) such families will comprise the nonzero elements of first 
\v{C}ech cohomology groups of the ordinal $\lambda$, taken with the order topology 
(cf.\ \cite{cohomology_of_the_ordinals}).

To illustrate this phenomenon, we consider here one of the simplest of the combinatorial objects 
derived from the walks apparatus: the \emph{number of steps function} 
$\rho_2 : [\lambda]^2 \ra \omega$ defined by setting $\rho_2(\alpha,\beta) = 
|\Tr(\alpha,\beta)| - 1$ for all $\alpha < \beta < \lambda$. From this one obtains the family 
of functions 
\[
  \Phi(\rho_2) = \langle \rho(\cdot, \beta) : \beta \ra \omega \mid \beta < \lambda \rangle.
\]
It turns out (cf. \cite[\S 6.3 and \S 7.1]{todorcevic_walks_book}) that, if $\mc{C}$ is a 
$\square(\lambda)$-sequence, then $\Phi(\rho_2)$ is coherent and nontrivial \emph{modulo 
locally semi-constant functions}:
\begin{itemize}
  \item (coherence) for all $\alpha < \beta < \lambda$, the function
  \[
    \rho_2(\cdot,\alpha) - \rho_2(\cdot,\beta) \restriction \alpha : \alpha \ra \bb{Z}
  \]
  is locally semi-constant;
  \item (nontriviality) there does not exist a function $\psi:\lambda \ra \bb{Z}$ such that, 
  for every $\alpha < \lambda$, the function
  \[
    \rho_2(\cdot, \alpha) - \psi \restriction \alpha
  \]
  is locally semi-constant.
\end{itemize}
The aforementioned efficacy of the walks machinery at $\omega_1$ is largely due to the fact 
that $\square(\omega_1)$ trivially holds in $\ZFC$. Therefore, $\rho_2$ provides a 
canonical $\ZFC$ instance of incompactness at $\omega_1$. To recover this behavior at higher 
cardinals $\lambda$, one must make assumptions that go beyond $\ZFC$ (namely, one must 
assume some $\square(\lambda)$-like property).

The classical walks on ordinals machinery is inherently \emph{one-dimensional} in nature. For 
instance, the $C$-sequences that serve as inputs and the coherent families of functions that 
are among its most prominent outputs are one-dimensional families of objects, and walking 
along a $C$-sequence on a cardinal $\lambda$ yields insight into the \emph{first} cohomology 
group of $\lambda$. There has been a growing recognition of late that, just as the cardinal 
$\omega_1$ provably exhibits nontrivial one-dimensional combinatorial phenomena, and just as 
these phenomena can be probed via walks on ordinals, so it is the case that the cardinal 
$\omega_n$ for $1 < n < \omega$ provably exhibits nontrivial $n$-dimensional combinatorial 
phenomena, and these phenomena should be legible via a higher-dimensional analogue of the walks 
on ordinals machinery. The beginnings of this recognition were visible already in work of 
Goblot \cite{goblot}, Mitchell \cite{mitchell}, and Osofsky 
\cite{osofsky_continuum, osofsky_cardinality} from around 1970 and has been made explicit in 
recent years, most notably in two works of Bergfalk beginning to develop the theory of 
higher-dimensional walks on ordinals  \cite{bergfalk_alephs, bergfalk2024introductionhigherwalks}.

In this paper, we build upon \cite{bergfalk2024introductionhigherwalks} and continue to develop 
the theory of higher-dimensional walks on ordinals. We focus in particular on Task 5 of that 
paper, and on applying the machinery of $n$-dimensional walks to cardinals greater than $\omega_n$. 
In the classical case of $n = 1$, effectively employing the walks machinery to cardinals
$\lambda > \omega_1$ typically requires making $\square(\lambda)$-type assumptions. Accordingly, 
we begin our story with an exploration of possible higher-dimensional square principles, 
asserting the existence of higher-dimensional $C$-sequences that simultaneously satisfy 
certain \emph{coherence} and \emph{nontriviality} conditions. The coherence condition we 
isolate is motivated by our eventual proof that $n$-dimensional walks along coherent 
$n$-$C$-sequences generate analogues of the function $\rho_2$ yielding coherent families of 
functions. We then isolate three natural nontriviality conditions, in increasing order of 
strength: weak nontriviality, nontriviality, and strong nontriviality. Then, given an 
$n \geq 1$ and a regular uncountable cardinal $\lambda$, we let 
$\Sq{n}^{w}(\lambda)$ (\emph{resp.}\ $\Sq{n}(\lambda)$, $\Sq{n}^{s}(\lambda)$) assert the existence of 
a coherent $n$-$C$-sequence on $\lambda$ that is weakly nontrivial (\emph{resp.}\ nontrivial, 
strongly nontrivial). In the case $n = 1$, all three principles coincide with the classical 
$\square(\lambda)$.

In Section \ref{section: squares}, we provide the formal definitions and begin to develop the 
theory of these higher-dimensional square principles, proceeding largely by analogy with the 
classical one-dimensional setting. For example, we establish all of the following results: if $1 \leq n < \omega$, then $\Sq{n}^s(\omega_n)$ holds, 
$\Sq{n}^s(\lambda)$ fails for all $\lambda < \omega_n$, and, if 
$\mathrm{V} = \mathrm{L}$, then $\Sq{n}^s(\lambda)$ holds for every 
regular uncountable cardinal $\lambda \geq \omega_n$ that is not Mahlo.
In the other direction, we establish a consistency result showing that, consistently, 
$\omega_n$ is the unique regular cardinal $\lambda$ satisfying $\Sq{n}^s(\lambda)$ (or even 
$\Sq{n}^w(\lambda)$:

\begin{thma}
  If there are infinitely many supercompact cardinals, then there is a forcing extension in 
  which, for all $1 \leq n < \omega$ and all regular cardinals $\lambda > \omega_n$, 
  $\Sq{n}^w(\lambda)$ fails.
\end{thma}

Beginning in Section \ref{Walks}, we shift our attention to higher-dimensional walks. We begin 
by recalling the relevant definitions from \cite{bergfalk2024introductionhigherwalks} and 
introducing some fundamental notation before turning to the primary subject of the remainder of 
the paper: higher-dimensional generalizations of $\rho_2$. In 
\cite{bergfalk2024introductionhigherwalks}, Bergfalk identifies a natural function 
$\rho_2^n : [\lambda]^{n+1} \ra \Z$ derived from performing $n$-dimensional walks on an 
$n$-$C$-sequence over a cardinal $\lambda$. If $n = 1$, then $\rho_2^n$ is essentially the 
classical function $\rho_2$, and Bergfalk proves that, if $\lambda = \omega_n$ the 
$n$-$C$-sequence used for walking is ``order-minimal'', then the family
\[
  \Phi(\rho_2^n) = \langle \rho(\cdot, \vec{\gamma}):\gamma_0 \ra \Z \mid \vec{\gamma} \in 
  [\lambda]^n \rangle
\]
is coherent modulo locally semi-constant functions. In addition, he asks for a more general 
condition on the $n$-$C$-sequence being employed that can hold at cardinals greater than 
$\omega_n$ and implies that $\Phi(\rho_2^n)$ is coherent. We prove that the notion of 
coherence isolated in Section \ref{section: squares} suffices for this. Looking forward to 
later sections on nontriviality, we choose not to focus on $\rho_2^n$ itself but rather on an 
\emph{enrichment} of $\rho_2^n$ of the form
\[
  \resh_n : [\lambda]^{n+1} \ra \bigoplus_{[\lambda]^{n-1}} \Z.
\]
If $n = 1$, then $\resh_n$ is again essentially the classical function $\rho_2$. For all 
positive $n$, the function $\resh_n$ will be seen to ``project'' to $\rho_2^n$ in such a way 
that the coherence of $\Phi(\resh_n)$ immediately yields the coherence of $\Phi(\rho_2^n)$. 
The following therefore answers Bergfalk's question from \cite{bergfalk2024introductionhigherwalks}:

\begin{thmb}
  Suppose that $1 \leq n < \omega$, $\lambda$ is a cardinal, and the function $\resh_n$ is 
  derived from performing $n$-dimensional walks along a coherent $n$-$C$-sequence over $\lambda$. 
  Then $\Phi(\resh_n)$ is coherent modulo locally semi-constant functions.
\end{thmb}

In Section \ref{section: nontriviality}, 
we address nontriviality. As will be discussed in further detail 
at the end of Section \ref{section: nontriviality}, it seems unlikely that, for $n > 1$, a nontrivial coherent 
$n$-dimensional family of functions taking values in $\bb{Z}$ can be straightforwardly and 
uniformly derived from performing $n$-dimensional walks along an $n$-$C$-sequence over a cardinal 
$\lambda$. Indeed, it remains a major open whether, for $n > 1$, there provably exists a 
nontrivial coherent $n$-family of functions of the form 
\[
  \Phi = \langle \varphi_{\vec{\gamma}} : \gamma_0 \ra \bb{Z} \mid \vec{\gamma} \in [\omega_n]^n 
  \rangle.
\]
This is the primary motivation for our introduction of the $\resh_n$ function, which, for 
$n > 1$, takes values in the larger group $\bigoplus_{[\lambda]^{n-1}} \Z$ rather than 
$\Z$:

\begin{thmc}
  Suppose that $1 \leq n < \omega$, $\lambda$ is a regular uncountable cardinal, and the function 
  $\resh_n$ is derived from performing $n$-dimensional walks along a $\Sq{n}^s(\lambda)$-sequence. 
  Then $\Phi(\resh_n)$ is nontrivial modulo locally semi-constant functions.
\end{thmc}

\subsection{Notation and conventions} If $A$ is a set of ordinals, then we let 
$\ssup(A) = \sup\{\alpha + 1 \mid \alpha \in A\}$, $\acc^+(A) = \{\beta < \ssup(A) \mid 
\sup(A \cap \beta) = \beta > 0\}$, $\acc(A) = A \cap \acc^+(A)$, and 
$\nacc(A) = A \setminus \acc(A)$. If $\lambda$ is an ordinal and $D \subseteq \lambda$, we say 
that $D$ is \emph{club} in $\lambda$ if $\ssup(D) = \lambda$ and $\acc^+(D) \subseteq D$. Note 
that this is slightly nonstandard usage, as it is possible, for instance, for a set $D$ to be club 
in a successor ordinal. If $\gamma$ is an infinite ordinal and $\mu$ is a regular cardinal, then $S^\gamma_{\mu}:=\{\alpha<\gamma:\cf(\alpha)=\mu\}$. Symbols like $S^\gamma_{{R}\mu}$ for $R\in \{<,>,\le,\ge\}$ receive the obvious meaning.

If $m < n < \omega$ and $\vec{\gamma}$ is a sequence of ordinals of length $n$, 
then we will always denote the $m^{\mathrm{th}}$ element of $\vec{\gamma}$ by 
$\gamma_m$, i.e., $\vec{\gamma} = \langle \gamma_0, \gamma_1, \ldots, \gamma_{n-1} \rangle$. 
In such situations, we let $\vec{\gamma}^m$ denote the sequence obtained 
by removing $\gamma_m$ from $\vec{\gamma}$; i.e.,
\[
  \vec{\gamma}^m = \langle \gamma_0, \ldots, \gamma_{m-1}, \gamma_{m+1}, \ldots, \gamma_{n-1} \rangle.
\] 
If $A$ is a set of ordinals and $n < \omega$, then 
$[A]^n = \{x \subseteq A : |x| = n\}$. We will frequently identify $[A]^n$ 
with the set of all strictly increasing sequences $\vec{\gamma}$ of length $n$ 
consisting of elements of $A$. Similarly, we will let $A^{[n]}$ denote the 
set of all \emph{weakly} increasing sequences of length $n$ from $A$, i.e., 
the set of all $\vec{\gamma} \in A^n$ such that $\gamma_m \leq \gamma_{m+1}$ for 
all $m < n-1$. We will sometimes slightly abuse notation and write, e.g., 
$(\vec{\alpha},\vec{\beta})$ instead of $\vec{\alpha}^\frown \vec{\beta}$. 
For instance, if $n$ is a positive integer, $\lambda$ is an ordinal, 
$\vec{\gamma} \in \lambda^{[n]}$, and $\alpha \leq \gamma_0$, then we will think 
of $(\alpha,\vec{\gamma})$ as an element of $\lambda^{[n+1]}$, namely 
$\langle \alpha \rangle ^\frown \vec{\gamma}$. We let 
$\lambda \otimes [\lambda]^n$ denote the set of all $(\alpha,\vec{\gamma})$ such that 
$\vec{\gamma} \in [\lambda]^n$ and $\alpha \leq \gamma_0$. As in the previous sentence, we 
will typically think of elements of $\lambda \otimes [\lambda]^n$ as $(n+1)$-tuples.

If $\varphi$ and $\psi$ are functions mapping into the same abelian group, 
then we will slightly 
abuse notation and write $\varphi + \psi$ to denote the function 
$\rho$ such that $\dom(\rho) = \dom(\varphi) \cap \dom(\psi)$ and 
$\rho(x) = \varphi(x) + \psi(x)$ for all $x \in \dom(\rho)$. Analogous conventions 
apply to similar sums of more than two functions.

If $I$ is an index set, then $\bigoplus_I \bb{Z}$ denotes the free abelian group on $I$; 
formally, its elements are finitely supported functions from $I$ to $\bb{Z}$. Given 
$e \in I$, we let $\lfloor e \rfloor$ denote the basis element of $\bigoplus_I \bb{Z}$ 
associated with $e$; formally, this is the function taking value $1$ on $e$ and $0$ on all 
elements of $I \setminus \{e\}$. We will sometimes slightly abuse notation and write, say, 
$\lfloor ((-1)^k,e) \rfloor$ instead of $(-1)^i \lfloor e \rfloor$.

\section{Preliminaries on coherence and triviality} \label{section: prelim}

In this brief section, we review some background information about nontrivial coherent families 
of functions.

\begin{definition}
  Suppose that $D$ is a set of ordinals that is closed in its supremum. 
  Given an abelian group $H$, and an ordinal $\gamma \in D$, we say that a function 
  $\varphi:D \cap \gamma \ra H$ is \emph{locally semi-constant} if, for every 
  $\alpha \in \acc(D) \cap (\gamma + 1)$, there is $\eta < \alpha$ such that 
  $\varphi \restriction D \cap (\eta,\alpha)$ is constant.
\end{definition}

\begin{definition} \label{def: coh_triv_def}
  Suppose that $n$ is a positive integer, $D$ is a set of ordinals closed in its 
  supremum, $H$ is an abelian group. We say that a family of functions
  \[
    \Phi = \langle \varphi_{\vec{\gamma}}:\gamma_0 \ra H \mid \vec{\gamma} \in 
    [D]^n \rangle
  \]
  is
  \begin{enumerate}
    \item \emph{coherent modulo locally semi-constant functions} if, for all 
    $\vec{\delta} \in [D]^{n+1}$, the function
    \[
      \sum_{i \leq n} (-1)^i \varphi_{\vec{\delta}^i} : \delta_0 \ra H
    \]
    is locally semi-constant;
    \item \emph{trivial modulo locally semi-constant functions} if
    \begin{enumerate}
      \item $n = 1$ and there exists a function $\psi : D \ra H$ such that, 
      for all $\gamma \in D$, $\varphi_\gamma - \psi \restriction (D \cap \gamma)$ 
      is locally semi-constant;
      \item $n > 1$ and there exists a family of functions
      \[
        \Psi = \langle \psi_{\vec{\beta}}: \beta_0 \ra H \mid 
        \vec{\beta} \in [D]^{n-1} \rangle
      \]
      such that, for all $\vec{\gamma} \in [D]^n$, the function
      \[
        \varphi_{\vec{\gamma}} - \sum_{i < n} (-1)^i \psi_{\vec{\gamma}^i}: 
        \gamma_0 \ra H
      \]
      is locally semi-constant.
    \end{enumerate}
  \end{enumerate}
\end{definition}

We will often refer to a family of the form 
\[
    \Phi = \langle \varphi_{\vec{\gamma}}:\gamma_0 \ra H \mid \vec{\gamma} \in 
    [D]^n \rangle
\]
as an \emph{$n$-family} on $D$, taking values in $H$ (or simply an $n$-family, if 
$D$ and $H$ are clear from context).
It is readily verified that, under the assumptions of Definition 
\ref{def: coh_triv_def}, if a family $\Phi = \langle \varphi_{\vec{\gamma}} 
\mid \vec{\gamma} \in [D]^n \rangle$ is trivial, then it is coherent. 
The general question motivating much of this paper concerns determining 
the situations in which there exists such a family that is coherent but 
nontrivial.

By varying the modulus, one obtains other notions of coherence and triviality. 
One of the most prominent in the literature is coherence and triviality 
\emph{modulo finite}, in which one replaces the requirement that the relevant 
functions in Definition \ref{def: coh_triv_def} be locally semi-constant by 
the requirement that they be finitely supported. In this paper, we will be 
almost exclusively be working with coherence and triviality modulo almost 
locally constant functions, so if there is no risk of confusion we will 
write, simply ``coherent" instead of ``coherent modulo locally semi-constant 
functions". Moreover, from the point of view of the existence of coherent 
nontrivial families, it does not matter whether the modulus is ``locally 
semi-constant" or ``finite", as evidenced by the following fact, 
combining Theorem 3.2 and Lemma 7.4 from 
\cite{bergfalk2024introductionhigherwalks}:\footnote{The cited results from 
\cite{bergfalk2024introductionhigherwalks} are stated in the case in 
which the club $D$ is itself an ordinal, but the apparently more 
general statement given here is easily seen to be equivalent to this special 
case.}

\begin{fact}
  Suppose that $n$ is a positive integer, $D$ is a set of ordinals closed 
  in its supremum, and 
  $H$ is an abelian group. Then the following are equivalent:
  \begin{enumerate}
    \item there exists a family of functions $\langle \varphi_{\vec{\gamma}} 
    : \gamma_0 \ra H \mid \vec{\gamma} \in [D]^n \rangle$ that is coherent 
    and nontrivial modulo locally semi-constant functions;
    \item there exists a family of functions $\langle \varphi_{\vec{\gamma}} 
    : \gamma_0 \ra H \mid \vec{\gamma} \in [D]^n \rangle$ that is coherent 
    and nontrivial modulo finite.
  \end{enumerate}   
\end{fact}

Let us mention at the end of this section that part of the motivation for the 
study of nontrivial coherent families of functions comes from cohomological considerations. 
Given an ordinal $\lambda$, consider $\lambda$ as a topological space with the order 
topology. Given an abelian group $H$, let $\mc{F}_H$ denote the presheaf on $\lambda$ 
defined by setting $\mc{F}_H(U) = \bigoplus_U H$ for all open $U \subseteq \lambda$, and 
let $\mc{A}_H$ denote the presheaf defined by letting $\mc{A}_H(U)$ be the set of all 
locally semi-constant functions from $U$ to $H$ for all open $U \subseteq \lambda$. Then, as 
noted in \cite[Theorem 3.2]{bergfalk2024introductionhigherwalks}, $n$-families 
\[
  \Phi = \langle \varphi_{\vec{\gamma}}:\gamma_0 \ra H \mid \vec{\gamma} \in [\lambda]^n \rangle
\]
that are coherent and nontrivial modulo finite represent the nonzero cohomology classes in the 
\v{C}ech cohomology group $\check{\mathrm{H}}(\lambda, \mc{F}_H)$, while families of the above 
form that are coherent and nontrivial modulo locally semi-constant functions represent the nonzero 
cohomology classes of $\check{\mathrm{H}}(\lambda, \mc{A}_H)$. In particular, when we prove below 
that a certain cardinal $\lambda$ carries an $n$-family that is nontrivial and coherent modulo 
locally semi-constant functions, this will establish that the \v{C}ech cohomology group 
$\check{\mathrm{H}}(\lambda, \mc{A}_H)$ is nonzero; by \cite[Theorem 3.2 and 
Lemma 7.4]{bergfalk2024introductionhigherwalks}, this is equivalent to the assertion that 
$\check{\mathrm{H}}(\lambda, \mc{F}_H)$ is nonzero.

\section{Higher squares} \label{section: squares}

In this section we begin our investigation into higher-dimensional square principles. 
We first recall the notion of an \emph{$n$-$C$-sequence} from \cite{bergfalk2024introductionhigherwalks} 
and then specify what we mean by \emph{coherence} of an $n$-$C$-sequence. 
We note that a different notion of coherence is considered in 
\cite[Theorem 7.2(3)]{bergfalk2024introductionhigherwalks}. We found this notion to be too 
strong; for instance, in analogy with the $1$-dimensional situation at $\omega_1$, one would 
like it to be the case that, for all $n$, an \emph{order-type-minimal} $n$-$C$-sequence 
on $\omega_n$ is trivially coherent. This does not seem to be the case for the notion of 
coherence considered in \cite{bergfalk2024introductionhigherwalks} but is easily seen to hold 
for our weaker notion of coherence, as in such situations the set 
$X(\mc{C})$ introduced below will be empty.

\begin{definition} \label{def: n_c_sequence}
  Suppose that $n$ is a positive integer, $\delta$ is an infinite ordinal, 
  and $D$ is a club in $\delta$. An $n$-$C$-sequence on $D$ is a system 
  $\mc{C} = \langle C_{\vec{\gamma}} \mid \vec{\gamma} \in I(\mc{C}) \rangle$ 
  such that $I(\mc{C}) \subseteq [D]^{\leq n}$ and the following hold:
  \begin{enumerate}
    \item $\emptyset \in I(\mc{C})$ and $C_\emptyset = D$;
    \item given $\langle \beta \rangle^\frown \vec{\gamma} \in [D]^{\leq n}$, 
    $\langle \beta \rangle^\frown \vec{\gamma} \in I(\mc{C})$ if and only if 
    $\vec{\gamma} \in I(\mc{C})$ and $\beta \in C_{\vec{\gamma}}$. In 
    that case, $C_{\beta\vec{\gamma}}$ is club in $\beta \cap C_{\vec{\gamma}}$ 
    (and $C_{\beta\vec{\gamma}} = \emptyset$ if $\beta \cap C_{\vec{\gamma}} = 
    \emptyset$).
  \end{enumerate}
  We say that an $n$-$C$-sequence $\mc{C}$ on $D$ is \emph{order-type-minimal} 
  if $\otp(C_{\vec{\gamma}}) = \cf(\gamma_0 \cap C_{\vec{\gamma}^0})$ for all 
  nonempty $\vec{\gamma} \in I(\mc{C})$.  
  
  Given an $n$-$C$-sequence $\mc{C}$ on $D$, let $X(\mc{C})$ denote the 
  set of $\alpha \in D$ for which there exists $\vec{\gamma} \in 
  I(\mc{C}) \cap [D]^n$ such that $\alpha \in \acc(C_{\vec{\gamma}})$. 
  We say that $\mc{C}$ is \emph{$n$-coherent}, or simply \emph{coherent} if 
  the value of $n$ is clear from context, if:
  \begin{enumerate}[start=3]
    \item \label{coherence_condition} For all $\alpha \in X(\mc{C})$, the following hold:
    \begin{itemize}
      \item for all $\vec{\gamma} \in I(\mc{C}) \cap [D]^n$ such that 
      $\alpha \in \acc(C_{\vec{\gamma}})$, we have 
      $C_{\vec{\gamma}} \cap \alpha = C_\alpha$;
      \item for all $\vec{\beta} \in I(\mc{C})$ such that $\beta_0 = \alpha$ 
      and $\sup(C_{\vec{\beta}}) = \alpha$, we have $C_{\vec{\beta}} = C_\alpha$.
    \end{itemize}
  \end{enumerate}    
\end{definition}

\begin{remark} \label{remark: minimal_at_successors}
  We will typically assume without comment that all $n$-$C$-sequences under consideration 
  are \emph{minimal at successors}, i.e., if $\vec{\gamma} \in I(\mc{C}) \cap [D]^{<n}$ and 
  $\beta \in \nacc(C_{\vec{\gamma}})$, then $C_{\beta\vec{\gamma}} = \{\max(C_{\vec{\gamma}} 
  \cap \beta)\}$ (or $C_{\beta\vec{\gamma}} = \emptyset$ if $C_{\vec{\gamma}} \cap 
  \beta = \emptyset$). In this case, we will let $I^+(\mc{C})$ denote the set of all 
  $\vec{\gamma} \in I(\mc{C})$ such that $C_{\vec{\gamma}}$ has more than one element, i.e., 
  the set of $\vec{\gamma} = \langle \gamma_0, \ldots, \gamma_{m-1} \rangle$ in 
  $I(\mc{C})$ such that $\gamma_{m-1} \in \acc(D)$ and, for all $k < m-1$, we have 
  $\gamma_k \in \acc(C_{\gamma_{k+1}\ldots\gamma_{m-1}})$.
  Thus, to specify an $n$-$C$-sequence $\mc{C}$ on a club $D$, it will 
  suffice to explicitly define clubs of the form $C_{\vec{\gamma}}$ such that 
  $\vec{\gamma} \in I^+(\mc{C})$.
\end{remark}

Suppose that $n$ is a positive integer, $\lambda$ is an infinite ordinal, 
$D$ is a club in $\lambda$, and $\mc{C}$ is an $(n+1)$-$C$-sequence on 
$D$. Fix an ordinal $\delta \in D$. Let $I(\mc{C}^\delta) = 
\{\vec{\gamma} \in [C_\delta]^{\leq n} : \vec{\gamma}^\frown \langle 
\delta \rangle \in I(\mc{C})\}$ and, for all $\vec{\gamma} \in 
I(\mc{C}^\delta)$, set $C^\delta_{\vec{\gamma}} = C_{\vec{\gamma}\delta}$. 

\begin{proposition}
  The sequence $\mc{C}^\delta = \langle C^\delta_{\vec{\gamma}} \mid 
  \vec{\gamma} \in I(\mc{C}^\delta) \rangle$ is an $n$-$C$-sequence on 
  $C_\delta$. Moreover, if $\mc{C}$ is $(n+1)$-coherent, then $\mc{C}^\delta$ is 
  $n$-coherent.
\end{proposition}

\begin{proposition} \label{prop: stat_x}
  Suppose that $n$ is a positive integer, $\lambda$ is a limit ordinal, $D$ is 
  a club in $\lambda$, and $\mc{C}$ is a coherent $n$-$C$-sequence on $D$. 
  Then, for every $\delta \in \acc(D) \cap S^{\lambda}_{\geq \aleph_n}$, 
  $X(\mc{C}) \cap \delta$ is stationary in $\delta$.
\end{proposition}

\begin{proof}
  The proof is by induction on $n$. If $n = 1$, then, for every 
  $\delta \in \acc(D) \cap S^{\lambda}_{\geq \aleph_n}$, we have 
  $\acc(C_\delta) \subseteq X(\mc{C})$, and the conclusion follows. 
  Thus, suppose that $n > 1$, and fix $\delta \in \acc(D) \cap 
  S^{\lambda}_{\geq \aleph_n}$. It follows immediately from the definitions 
  that $X(\mc{C}^\delta) \subseteq X(\mc{C})$. Moreover, by the induction 
  hypothesis, we know that, for every $\gamma \in \acc(C_\delta) \cap 
  S^\delta_{{\geq} \aleph_{n-1}}$, the set $X(\mc{C}^\delta) \cap \gamma$ is 
  stationary in $\gamma$. Since $\cf(\delta) \geq \aleph_n$, it follows that 
  $X(\mc{C}^\delta)$ is stationary in $\delta$, and hence 
  $X(\mc{C}) \cap \delta$ is stationary in $\delta$ as well.
\end{proof}

We now define various natural notions of nontriviality for coherent $n$-$C$-sequences.

\begin{definition} \label{def: trivial_c_sequence}
  Suppose that $n$ is a positive integer, $\lambda$ is an ordinal, 
  $D$ is a club in $\lambda$, and $\mc{C}$ is a coherent $n$-$C$-sequence on $D$.
  \begin{enumerate}
    \item $\mc{C}$ is \emph{weakly nontrivial} if, for every club 
    $D' \subseteq D$ in $\lambda$, there exists $\alpha \in \acc(D')$ such that 
    $D' \cap \alpha \neq C_\alpha$;  
    \item $\mc{C}$ is \emph{nontrivial} if it cannot be extended to a 
    coherent $n$-$C$-sequence on $D \cup \{\lambda\}$, i.e., there does not exist an 
    $n$-$C$-sequence $\mc{C'}$ on $D \cup \{\lambda\}$ such that $I(\mc{C}') \cap 
    [\lambda]^{\leq n} = I(\mc{C})$ and, for all nonempty $\vec{\gamma} \in I(\mc{C})$, we have 
    $C_{\vec{\gamma}} = C'_{\vec{\gamma}}$;
    \item $\mc{C}$ is \emph{strongly nontrivial} if $\otp(D)$ is a regular uncountable 
    cardinal and either
    \begin{itemize}
      \item $n = 1$ and $\mc{C}$ is nontrivial; or 
      \item $n > 1$ and the set $\{\delta \in D \mid \mc{C}^\delta \emph{is strongly nontrivial}\}$ 
      is stationary in $\lambda$.
    \end{itemize}
  \end{enumerate}
  We let $\Sq{n}(D)$ (resp.\ $\Sq{n}^{w}(D)$, resp.\ $\Sq{n}^s(D)$) denote the assertion that there exists 
  a coherent $n$-$C$-sequence on $D$ that is nontrivial (resp.\ weakly nontrivial, resp.\ 
  strongly nontrivial). A witness to $\Sq{n}(D)$ is called a $\Sq{n}(D)$-sequence 
  (and similarly for $\Sq{n}^w(D)$ and $\Sq{n}^s(D)$).
\end{definition}

In Proposition \ref{prop: small_cof} and Fact \ref{fact: cof_omega_n} below, 
we establish analogues of the trivial observations about classical square principles 
at $\aleph_0$ and $\aleph_1$, namely that $\square(\aleph_0)$ fails and 
$\square(\aleph_1)$ holds in $\ZFC$.

\begin{proposition} \label{prop: small_cof}
  Suppose that $n$ is a positive integer, $\lambda$ is an ordinal with $\cf(\lambda) < \aleph_n$, 
  $D$ is club in $\lambda$. Then $\Sq{n}^s(D)$ fails.
\end{proposition}

\begin{proof}
  The proof is by induction on $n$. If $n = 1$, then $\cf(\lambda) < \aleph_1$, and hence $\otp(D)$ 
  cannot be a regular uncountable cardinal, so $\Sq{n}^s(D)$ fails. Suppose that $n > 1$ and 
  we have established the proposition for $n-1$. Let $\mc{C}$ be a coherent $n$-$C$-sequence 
  on $D$; we will show that it is not strongly nontrivial. Let $D' \subseteq D$ be club in $\lambda$ 
  such that $\cf(\delta) < \aleph_{n-1}$ for all $\delta \in D'$. Then, for all $\delta \in D'$, 
  the inductive hypothesis implies that $\mc{C}^\delta$ is not strongly nontrivial; it follows 
  that $\mc{C}$ itself is not strongly nontrivial.
\end{proof}


It is straightforward to prove by induction on $n \geq 1$ that, if $\mc{C}$ is a 
coherent order-minimal $n$-$C$-sequence, then it is a $\Sq{n}^s(D)$-sequence. 
Moreover, if $D$ is a club in an ordinal $\lambda$ and $\otp(D) = \omega_n$, 
then there exists a coherent order-minimal $n$-$C$-sequence on $D$. We therefore 
obtain the following fact.

\begin{fact} \label{fact: cof_omega_n}
  Suppose that $n$ is a positive integer, $\lambda$ is an ordinal, and 
  $D$ is club in $\lambda$ with $\otp(D) = \aleph_n$. Then $\Sq{n}^s(D)$ holds.
\end{fact}

We now show that, in the constructible universe there is a preponderance of higher square 
sequences.

\begin{lemma} \label{lemma: square_stepping_up}
  Suppose that $m$ is a positive integer, $n = m+1$, $\kappa < \lambda$ are regular 
  uncountable cardinals, 
  $\Sq{m}^s(\kappa)$ holds, and there is a $\square(\lambda)$-sequence $\mc{D}$ such that the 
  set
  \[
    S = \{\gamma \in S^\lambda_\kappa \mid \otp(D_\gamma) = \kappa\}
  \]
  is stationary. Then $\Sq{n}^s(\lambda)$ holds.
\end{lemma}

\begin{proof}
  We begin by slightly modifying $\mc{D}$. Let $S_0 = \{\gamma \in \acc(\lambda) \mid 
  \otp(D_\gamma) \leq \kappa\}$ and $S_1 = \acc(\lambda) \setminus S_0$. Now define a 
  sequence $\mc{D}' = \langle D'_\gamma \mid \gamma \in \acc(\lambda) \rangle$ by setting
  \[
    D'_\gamma = 
    \begin{cases}
      D_\gamma & \text{if } \gamma \in S_0 \\ 
      D_\gamma \setminus D_\gamma(\kappa) & \text{if } \gamma \in S_1,
    \end{cases}
  \]
  where $D_\gamma(\kappa)$ denotes the unique $\alpha \in D_\gamma$ such that 
  $\otp(D_\gamma \cap \alpha) = \kappa$. Note that $S \subseteq S_0$ and, for all 
  $\gamma \in S_1$, we have $\acc(D'_\gamma) \cap S_0 = \emptyset$. Fix a 
  $\Sq{m}^s(\kappa)$-sequence $\mc{E} = \langle E_{\vec{\eta}} \mid \vec{\eta} \in I(\mc{E}) 
  \rangle$.
  
  We will now define a $\Sq{n}^s(\lambda)$-sequence $\mc{C} = \langle C_{\vec{\gamma}} 
  \mid \vec{\gamma} \in I(\mc{C}) \rangle$, recalling from Remark \ref{remark: minimal_at_successors} 
  our assumption of minimality at successors. The construction will be essentially disjoint 
  on the sets $S_0$ and $S_1$. In particular, we will have $I^+(\mc{C}) \subseteq 
  [S_0]^{\leq n} \cup [S_1]^{\leq n}$. Let us first deal with $S_1$. Given $\vec{\gamma} 
  = \langle \gamma_0, \ldots, \gamma_{k-1} \rangle \in [S_1]^{\leq n}$, we put 
  $\vec{\gamma} \in I^+(\mc{C})$ if and only if, for all $j < k-1$, we have 
  $\gamma_j \in \acc(D'_{\gamma_{k-1}})$; for all such $\vec{\gamma}$, set 
  $C_{\vec{\gamma}} = D'_{\gamma_0}$.
  
  We next deal with $S_0$. The idea is to copy $\mc{E}$ along $D_\gamma$ for each 
  $\gamma \in S$. First, for each $\gamma \in S_0$, set 
  $\eta_\gamma = \otp(D_\gamma)$. For each $\vec{\gamma} = \langle \gamma_0, \ldots, 
  \gamma_{k-1}\rangle \in [S_0]^{\leq n}$, we set 
  $\vec{\gamma} \in I^+(\mc{C})$ if and only if, for all $j < k-1$, the following hold:
  \begin{itemize}
    \item $\gamma_j \in \acc(D_{\gamma_{k-1}})$; and
    \item one of the following two holds:
    \begin{itemize}
      \item $\gamma_{k-1} \in S$, $\langle \eta_{\gamma_{j+1}}, \ldots, 
      \eta_{\gamma_{k-2}} \rangle \in I(\mc{E})$, and 
      $\eta_{\gamma_j} \in \acc(E_{\eta_{\gamma_{j+1}}\ldots\eta_{\gamma_{k-2}}})$;
      \item $\gamma_{k-1} \notin S$, $\langle \eta_{\gamma_{j+1}}, \ldots, 
      \eta_{\gamma_{k-1}} \rangle \in I(\mc{E})$, and $\eta_{\gamma_j} \in \acc(E_{\eta_{\gamma_{j+1}}\ldots\eta_{\gamma_{k-1}}})$.
    \end{itemize}
  \end{itemize}
  For such $\vec{\gamma}$, we define $C_{\vec{\gamma}}$ according to the following cases:
  \begin{itemize}
    \item if $\gamma \in S$, then set $C_\gamma = D_\gamma$;
    \item if $\vec{\gamma} = \langle \gamma_0, \ldots, \gamma_{k-1} \rangle$ and 
    $\gamma_{k-1} \in S$, then set
    \[
      C_{\vec{\gamma}} = \{D_{\gamma_0}(\xi) \mid \xi \in E_{\eta_{\gamma_0}
      \ldots\eta_{\gamma_{k-2}}}\};
    \]
    \item if $\vec{\gamma} = \langle \gamma_0, \ldots, \gamma_{k-1} \rangle$,
    $\gamma_{k-1} \notin S$, and $k < n$, then set
    \[
      C_{\vec{\gamma}} = \{D_{\gamma_0}(\xi) \mid \xi \in E_{\eta_{\gamma_0}
      \ldots\eta_{\gamma_{k-1}}}\};
    \]
    \item if $\vec{\gamma} = \langle \gamma_0, \ldots, \gamma_{k-1} \rangle$,
    $\gamma_{k-1} \notin S$, and $k = n$, then set
    \[
      C_{\vec{\gamma}} = C_{\vec{\gamma}^0} \cap \gamma_0.
    \]
  \end{itemize}
  It is now straightforward if tedious to use the coherence of $\mc{D}$ and $\mc{E}$ to verify 
  that $\mc{C}$ is a coherent $n$-$C$-sequence on $\lambda$. Moreover, by construction, for 
  each $\gamma \in S$, $\mc{C}^\gamma$ is isomorphic to $\mc{E}$ via the order-preserving bijection
  $\pi_\gamma : D_\gamma \ra \kappa$. Since $\mc{E}$ is strongly nontrivial, this implies that 
  $\mc{C}^\gamma$ is strongly nontrivial for all $\gamma \in S$. Since $S$ is stationary, it 
  follows that $\mc{C}$ is a $\Sq{n}^s(\lambda)$-sequence, as desired.
\end{proof}

\begin{corollary}
  Let $n$ be a positive integer. If $\lambda \geq \omega_n$ is a regular cardinal that is not Mahlo 
  in $\mathrm{L}$, then $\Sq{n}^s(\lambda)$ holds.
\end{corollary}

\begin{proof}
  If $n = 1$, then \cite[Theorem 7.1.5]{todorcevic_walks_book} implies that $\square(\lambda)$ 
  holds for every regular uncountable cardinal $\lambda$ that is not weakly compact in $\mathrm{L}$. 
  We can thus assume that $n = m+1$ for some positive integer $m$. By assumption, we have 
  $\lambda > \omega_m$. If $\lambda$ is not Mahlo in $\mathrm{L}$, then 
  \cite[Theorem 7.3.1]{todorcevic_walks_book} implies that there is a special 
  $\square(\lambda)$-sequence. The exact definition of \emph{special $\square(\lambda)$-sequence} 
  is not important here; what is relevant is that, by \cite[Lemma 7.2.12]{todorcevic_walks_book}, 
  there exists a $\square(\lambda)$-sequence $\mc{D}$ such that the set
  \[
    S = \{\gamma \in S^\lambda_{\omega_m} \mid \otp(D_\gamma) = \omega_m\}
  \]
  is stationary. Then Lemma \ref{lemma: square_stepping_up} together with 
  Fact \ref{fact: cof_omega_n} implies that $\Sq{n}^s(\lambda)$ holds.
\end{proof}

\begin{proposition}
  Suppose that $n$ is a positive integer, $\lambda$ is an ordinal with 
  $\cf(\lambda) \geq \aleph_n$, $D$ is club in $\lambda$, 
  and $\mc{C}$ is a coherent $n$-$C$-sequence on $D$. If $\mc{C}$ is strongly nontrivial, then 
  it is nontrivial, and if $\mc{C}$ is nontrivial, then it is weakly nontrivial.
\end{proposition}

\begin{proof}
  If $n = 1$, then all 
  three notions of nontriviality coincide, with strong nontriviality having 
  the additional requirement that $\otp(D)$ is regular, 
  so there is nothing to prove. Thus, assume 
  that $n > 1$. First, we leave to the reader the easy verification that, 
  if $\lambda$ is a successor ordinal or a limit ordinal of countable cofinality, then 
  $\mc{C}$ cannot satisfy any of the above varieties of nontriviality; we may 
  thus assume that $\cf(\lambda) > \omega$.
  
  Assume first that $\mc{C}$ is not weakly nontrivial, and fix 
  a club $D' \subseteq D$ such that $D' \cap \alpha = C_\alpha$ for all 
  $\alpha \in \acc(D')$. 
  
  \begin{claim} \label{claim: x_s}
    $\acc(D') \subseteq X(\mc{C})$.  
  \end{claim}
  
  \begin{proof}
    Fix $\alpha \in \acc(D')$. Fix $\gamma \in \acc(D') \cap S^\lambda_{\aleph_{n-1}}$ 
    with $\gamma > \alpha$, and fix $\delta \in \acc(D') \setminus (\gamma + 1)$. 
    By assumption, we have $D' \cap \delta = C_\delta$; in particular, 
    $\gamma \in \acc(C_\delta)$. By Proposition 
    \ref{prop: stat_x}, the set $X(\mc{C}^\delta) \cap \gamma$ is stationary 
    in $\gamma$. Since $X(\mc{C}^\delta) \subseteq X(\mc{C})$ and $\gamma \in 
    \acc(D')$, we can find $\beta \in \acc(D') \cap X(\mc{C})$ with 
    $\alpha < \beta < \gamma$. Now find $\vec{\varepsilon} \in I(\mc{C}) \cap 
    [D]^n$ with $\beta \in \acc({C_{\vec{\varepsilon}}})$. Then we have 
    $C_{\vec{\varepsilon}} \cap \beta = C_\beta = D' \cap \beta$. It follows 
    that $\alpha \in \acc(C_{\vec{\varepsilon}})$, and hence $\alpha \in 
    X(\mc{C})$.
  \end{proof}
  
  We define an extension $\mc{C'}$ of $\mc{C}$ to $D \cup \{\lambda\}$ 
  as follows. Let $J$ be the set of all $\vec{\gamma} \in [D']^{<n}$ such that, for 
  all $ i < |\vec{\gamma}|$, we have $\gamma_i \in \acc(D')$. Set 
  $I^+(\mc{C}') = I^+(\mc{C}) \cup \{\vec{\gamma}^\frown\langle \lambda \rangle \mid 
  \vec{\gamma} \in J\}$. By Remark \ref{remark: minimal_at_successors}, it suffices to specify $\C'\restr I^+(\C')$. Set $C'_\lambda = D'$ and, for nonempty 
  $\vec{\gamma} \in J$, set $C'_{\vec\gamma\lambda}=D'\cap \gamma_0$. 

  By Claim \ref{claim: x_s} and the fact that $D' \cap \alpha = C_\alpha$ 
  for all $\alpha \in \acc(D')$, it follows that $\mc{C'}$ is in fact 
  $n$-coherent and hence witnesses that $\mc{C}$ is not nontrivial.
  
  We will now show that strong nontriviality implies nontriviality. We will prove the contrapositive by induction on $n$, with the case $n=1$ having already been dealt with in the first paragraph. Suppose next that $\mc{C}$ is not nontrivial, and let $\mc{C}'$ be an extension of $\mc{C}$ to $D \cup \{\lambda\}$. We will show that, 
  for all $\delta \in \acc(D'_\lambda)$, $\mc{C}^\delta$ fails to be 
  strongly nontrivial. If $\cf(\delta) < \aleph_{n-1}$, then this follows 
  from Proposition \ref{prop: small_cof}. Thus, fix $\delta \in 
  \acc(D'_\lambda) \cap S^\lambda_{\geq \aleph_{n-1}}$. We will show that 
  $\mc{C}^\delta$ fails to be nontrivial; by the inductive hypothesis, this 
  will imply that it is not strongly nontrivial.
  
  We will define an extension $\mc{E}$ of $\mc{C}^\delta$ to 
  $C_\delta \cup \{\delta\}$ recursively as follows. We will assume 
  that $n > 2$. The case in which $n = 2$ is similar and simpler. Note that, 
  to define $\mc{E}$, it suffices to specify those $\vec{\gamma} 
  \in [C_\delta]^{{<}n-1}$ for which $\vec{\gamma}^\frown \langle \delta 
  \rangle \in I(\mc{E})$, and to define $E_{\vec{\gamma}\delta}$ for such 
  $\vec{\gamma}$. We will arrange so that, for each such $\vec{\gamma}$, 
  we have $\vec{\gamma}^\frown \langle \delta, \lambda \rangle \in 
  \mc{C}'$ and $E_{\vec{\gamma}\delta} \subseteq C'_{\vec{\gamma}\delta\lambda}$.
  Moreover, we will arrange so that, for all such $\vec{\gamma}$, 
  we have $E_{\vec{\gamma}\delta} = C'_{\vec{\gamma}\delta\lambda} \cap 
  C_\delta$ \textbf{unless} this set would be bounded in $E_{\vec{\gamma}^0\delta} 
  \cap \gamma_0$, in which case $\acc(E_{\vec{\gamma}\delta}) = \emptyset$.

  First put $\langle \delta \rangle \in I(\mc{E})$ and 
  set $E_\delta = C_\delta \cap C'_{\delta\lambda}$. Now suppose that 
  $\vec{\gamma} \in I(\mc{E}) \cap [C_\delta]^{{<}n-2}$ and 
  $\beta \in E_{\vec{\gamma}\delta}$. If $\beta \in \nacc(E_{\vec{\gamma}\delta})$, 
  then simply set $E_{\beta\vec{\gamma}\delta} = \{\max(E_{\vec{\gamma}\delta}) 
  \cap \beta)\}$ (or $\emptyset$ if $\beta = \min(E_{\vec{\gamma}\delta})$).
  
  If $\beta \in \acc(E_{\vec{\gamma}\delta})$, then by arrangement we have 
  $E_{\vec{\gamma}\delta} = C'_{\vec{\gamma}\delta\lambda} \cap C_\delta$. 
  Let $E_{\beta\vec{\gamma}} = C'_{\beta\vec{\gamma}\delta\lambda} \cap 
  C_\delta$ if this set is unbounded in $\beta$. Otherwise, note that we 
  must have $\cf(\beta) = \omega$. In this case, let 
  $E_{\beta\vec{\gamma}}$ be an arbitrary $\omega$-sequence that is a subset 
  of $\beta \cap E_{\vec{\gamma}\delta}$ and is cofinal in $\beta$. Note that, 
  in this latter case, we must have $\beta \notin X(\mc{C'})$, since otherwise 
  we would have $C'_{\beta\vec{\gamma}\delta\lambda} = C_\beta = C_{\beta\delta} 
  \subseteq C_\delta$.
  
  This completes the description of $\mc{E}$. We leave to the reader the 
  straightforward but somewhat tedious verification that it is coherent and 
  thus witnesses that $\mc{C}^\delta$ is not nontrivial.
\end{proof}

\begin{definition} \label{def: partial_square}
  Suppose that $\lambda$ is an ordinal of uncountable cofinality and $S \subseteq \acc(\lambda)$ 
  is stationary. Then $\square(S)$ is the assertion that there exists a sequence 
  $\mc{C} = \langle C_\alpha \mid \alpha \in \Gamma \rangle$ such that
  \begin{enumerate}
    \item $S \subseteq \Gamma \subseteq \acc(\lambda)$;
    \item for all $\alpha \in \Gamma$, $C_\alpha$ is club in $\alpha$;
    \item for all $\beta \in \Gamma$ and all $\alpha \in \acc(C_\beta)$, we have 
    $\alpha \in \Gamma$ and $C_\alpha = C_\beta \cap \alpha$;
    \item for every club $D \subseteq \lambda$, there exists $\alpha \in 
    \Gamma \cap \acc(D)$ such that $D \cap \alpha \neq C_\alpha$.
  \end{enumerate}
\end{definition}

\begin{proposition} \label{prop: product_thread}
  Suppose that $\mu$ is an infinite regular cardinal, $\lambda$ is an ordinal of cofinality 
  greater than $\mu$, and $S \subseteq \lambda$ is $({\geq} \mu)$-club. Suppose that 
  $\mc{C}$ is a $\square(S)$-sequence and $\P$ is a forcing poset such that 
  \[
    \Vdash_{\P \times \P} ``\cf(\lambda) > \mu".
  \]
  Then $\mc{C}$ remains a $\square(S)$-sequence in $V^{\P}$.
\end{proposition}

\begin{proof}
  $S$ remains stationary in $V^{\P}$, and items (1)--(3) of Definition \ref{def: partial_square} 
  are clearly upward absolute to $V^{\P}$, so it suffices to check that item (4) continues to 
  hold in $V^{\P}$. Suppose for the sake of contradiction that $p^* \in \P$ and 
  $\dot{D}$ is a $\P$-name forced by $p^*$
  to be a club in $\lambda$ with the property that, for all $\alpha \in \acc(\dot{D}) \cap 
  \Gamma$, we have $\dot{D} \cap \alpha = C_\alpha$.
  
  Since $\Vdash_{\P} ``\dot{D} \notin V"$, we can find $p, p' \leq p^*$ and an $\eta < \lambda$ 
  such that $p \Vdash ``\eta \in \dot{D}"$ and $p' \Vdash ``\eta \notin \dot{D}"$. 
  Now let $G \times G'$ be $\P \times \P$-generic over $V$ with $(p,p') \in G \times G'$. 
  Let $D$ and $D'$ be the interpretations of $\dot{D}$ in $V[G]$ and $V[G']$, respectively. 
  Since $\cf(\lambda) > \mu$ and $S$ is $(\geq \mu)$-club in $V[G \times G']$, we can find 
  $\alpha \in (\acc(D) \cap \acc(D') \cap S) \setminus (\eta + 1)$. Since $p,p' \leq p^*$, 
  it follows that $D \cap \alpha = D' \cap \alpha = C_\alpha$. However, since $p \in G$, 
  we have $\eta \in D$, and since $p' \in G'$, we have $\eta \notin D'$, which yields the 
  desired contradiction.
\end{proof}

\begin{lemma} \label{lemma: 1_collapse}
  Suppose that $\mu < \kappa$ are infinite regular cardinals, with $\kappa$ 
  supercompact. Let $\P = \mathrm{Coll}(\mu^+, {<}\kappa)$, and let $\dot{\Q}$ be 
  a $\P$-name for a $\kappa$-directed closed forcing. Then the following statement holds in 
  $V^{\P \ast \dot{\Q}}$: for every ordinal $\lambda \in \mathrm{cof}({\geq}\kappa)$ and every 
  $({\geq }\mu)$-club $S \subseteq \lambda$, $\square(S)$ fails.
\end{lemma}

\begin{proof}
  In $V$, fix an ordinal $\lambda \in \mathrm{cof}({\geq}\kappa)$ and $\P \ast \dot{\Q}$-names
  $\dot{S}$ and $\dot{\mc{C}}$ for a $({\geq} \mu)$-club in $\lambda$ and a sequence 
  satisfying items (1)--(3) of Definition \ref{def: partial_square}. We will show that 
  $\dot{S}$ is not a $\square(\dot{S})$-sequence in $V^{\P \ast \dot{\Q}}$.
  
  Fix a cardinal $\chi$ such that $\chi >> \max{|\dot{\Q}|, \lambda}$, and let $j:V \ra M$ witness that 
  $\kappa$ is $\chi$-supercompact. Let $G \ast H$ be $\P \ast \dot{\Q}$-generic over $V$, and let 
  $S$ and $\mc{C} = \langle C_\alpha \mid \alpha \in \Gamma \rangle$ be the interpretations of 
  $\dot{S}$ and $\dot{\mc{C}}$, respectively, in $V[G \ast H]$.
  
  Note that $j(\P) = \mathrm{Coll}(\mu^+, {<}j(\kappa))$. By 
  \cite[Lemma 3]{magidor_reflecting}, we can write 
  $j(\P) \cong \P \ast \dot{\Q} \ast \dot{\bb{R}}$, where $\dot{\bb{R}}$ is a 
  $\P \ast \dot{\Q}$-name for a $\mu^+$-closed forcing. Let $I$ be $\bb{R}$-generic 
  over $V[G \ast H]$. Then, in $V[G \ast H \ast I]$, we can lift $j$ to 
  $j:V[G] \ra M[G \ast H \ast I]$. In $V[G \ast H \ast I]$, consider the set 
  $j``H \subseteq j(\Q)$. Note that $j``H \in M[G \ast H \ast I]$ and, in 
  $M[G \ast H \ast I]$, $j(\Q)$ is $j(\kappa)$-directed closed. Since 
  $|j``H| < \chi < j(\kappa)$, we can find $q^+ \in j(\Q)$ such that 
  $q^+$ is a lower bound for $j``H$. Let $H^+$ be $j(\Q)$-generic over 
  $V[G \ast H \ast I]$ with $q^+ \in H^+$. Then, in 
  $V[G \ast H \ast I \ast H^+]$, we can lift $j$ to $j:V[G \ast H] \ra 
  M[G \ast H \ast I \ast H^+]$.
  
  Let $j(\mc{C}) = \mc{C'} = \langle C'_\alpha \mid \alpha \in j(\Gamma) \rangle$. 
  Let $\delta = \sup(j``\lambda) < j(\lambda)$. Note that $\sup(j(S) \cap \delta) = 
  \delta$ and $\cf(\delta) = \cf(\lambda) \geq \mu$ in $M[G \ast H \ast I \ast H^+]$, and 
  thus $\delta \in j(S)$. Let $D_0 = \{\alpha < \lambda \mid j(\alpha) \in 
  C'_\delta\}$; note that $D_0$ is $({<}\kappa)$-club in $\lambda$. The following claim will 
  imply that $D_0$ is actually club in $\lambda$.
  
  \begin{claim}
    For all $\alpha\in\acc(D_0)$, $D_0 \cap \alpha = C_\alpha$.
  \end{claim}
  
  \begin{proof}
    Let $\beta = \sup(j``\alpha)$. Note that $j``C_\alpha \subseteq C'_{j(\alpha)}$; 
    it follows that $\beta \in \acc(C'_{j(\alpha)})$, and hence $C'_\beta = C'_{j(\alpha)} 
    \cap \beta$. It follows from the definition of $D_0$ that $\beta \in \acc(C'_\delta)$, 
    and hence, for all $\eta \in D_0 \cap \alpha$, we have $j(\eta) \in C'_\beta \subseteq 
    C'_{j(\alpha)}$. By elementarity, it follows that $D_0 \cap \alpha \subseteq C_\alpha$.
    
    For the other inclusion, fix $\eta \in C_\alpha$. Then $j(\eta) \in C'_{j(\alpha)} 
    \cap \beta = C'_\beta$. Since $C'_\beta = C'_\delta \cap \beta$, it follows from the 
    definition of $D_0$ that $\eta \in D_0$.
  \end{proof}
  
  By the above claim, in $V[G \ast H \ast I \ast H^+]$, $D_0$ witnesses that $\mc{C}$ is 
  not a $\square(S)$-sequence. However, $V[G \ast H \ast I \ast H^+]$ is an extension 
  of $V[G \ast H]$ by $\mu^+$-closed forcing; therefore, Proposition \ref{prop: product_thread} 
  implies that $\mc{C}$ is not a $\square(S)$-sequence in $V[G \ast H]$, either.
\end{proof}

\begin{theorem}
  Suppose that $n$ is a positive integer, $\mu$ is an infinite regular cardinal, 
  let $\kappa_0 = \mu^+$, and let 
  $\langle \kappa_i \mid 1 \leq i < n \rangle$ be an increasing sequence of supercompact cardinals 
  with $\kappa_1 > \kappa_0$. Let $\P = \Coll(\kappa_0, {<}\kappa_1) \ast 
  \dot{\Coll}(\kappa_1, {<}\kappa_2) \ast \cdots \ast \dot{\Coll}(\kappa_{n-1}, {<}\kappa_n)$, and 
  let $\dot{\Q}$ be a $\P$-name for a $\kappa_n$-directed closed forcing. 
  Then, in $V^{\P \ast \dot{\Q}}$, for every ordinal $\lambda \in 
  \mathrm{cof}({\geq}\kappa_n)$ and every club $D \subseteq \lambda$, 
  $\Sq{n}^w(D)$ fails.
\end{theorem}

\begin{proof}
  The proof is by induction on $n$. If $n = 1$, then the conclusion follows immediately from 
  Lemma \ref{lemma: 1_collapse}, so assume that $n > 1$. Let $G \ast H$ be $\P \ast \dot{\Q}$-generic 
  over $V$. In $V[G \ast H]$, fix an ordinal $\lambda \in \mathrm{cof}({\geq} \kappa_n)$, a 
  club $D \subseteq \lambda$, and a coherent $n$-$C$-sequence $\mc{C}$ on $D$. 
  For each $\delta \in \acc(D) \cap S^\lambda_{{\geq}\kappa_{n-1}}$, consider the 
  coherent $(n-1)$-$C$-sequence $\mc{C}^\delta$ on $C_\delta$. By the inductive hypothesis, $\mc{C}^\delta$ is not weakly nontrivial; we can therefore fix a 
  club $E_\delta \subseteq \delta$ 
  such that, for all $\alpha \in \acc(E_\delta)$, we have $E_\delta \cap \alpha = C_{\alpha\delta}$. 
  
  \begin{claim} \label{claim: x_claim}
    For all $\delta \in \acc(D) \cap S^\lambda_{{\geq}\kappa_{n-1}}$ and all 
    $\alpha \in \acc(E_\delta)$, we have $\alpha \in X(\mc{C})$.
  \end{claim}
  
  \begin{proof}
    Fix $\delta$ and $\alpha$ as in the statement of the claim. By Proposition \ref{prop: stat_x}, 
    $X(\mc{C}) \cap \delta$ is stationary in $\delta$, so we can fix $\beta \in (X(\mc{C}) \cap 
    \acc(E_\delta)) \setminus (\alpha + 1)$. We can then find $\vec{\gamma} \in I(\mc{C}) \cap 
    [D]^n$ such that $\beta \in \acc(C_{\vec{\gamma}})$. By coherence and the choice of 
    $E_\delta$, we have
    \[
      C_{\vec{\gamma}} \cap \beta = C_\beta = C_{\beta\delta} = E_\delta \cap \beta.
    \]
    In particular, since $\alpha \in \acc(E_\delta)$, we also have $\alpha \in \acc(C_{\vec{\gamma}})$, 
    and hence $\alpha \in X(\mc{C})$.
  \end{proof}
  
  Let $S = \acc(D) \cap S^\lambda_{{\geq}\kappa_{n-1}}$ and $\Gamma = S \cup\bigcup \{\acc(E_\delta) 
  \mid \delta \in S\}$. For $\alpha \in \Gamma \setminus S$, let $E_\alpha = C_\alpha$. 
  By a straightforward if somewhat tedious case analysis, it follows that the sequence 
  $\mc{E} = \langle E_\alpha \mid \alpha \in \Gamma \rangle$ satisfies items (1)--(3) 
  of Definition \ref{def: partial_square}. By Lemma \ref{lemma: 1_collapse}, 
  $\square(S)$ fails in $V[G \ast H]$, so there exists a club $E^* \subseteq \lambda$ 
  such that, for all $\alpha \in \Gamma \cap \acc(E^*)$, we have $E^* \cap \alpha = 
  E_\alpha$. Since $\Gamma$ is stationary in $\lambda$ and $\acc(E_\alpha) \subseteq 
  \Gamma$ for all $\alpha \in \Gamma$, it follows that $\acc(E^*) \subseteq \Gamma$.
  
  \begin{claim}
    For all $\alpha \in \acc(E^*)$, we have $E^* \cap \alpha = C_\alpha$.
  \end{claim}
  
  \begin{proof}
    Fix $\alpha \in \acc(E^*)$. If $\alpha \in \Gamma \setminus S$, then we have 
    $C_\alpha = E_\alpha = E^* \cap \alpha$, so assume that $\alpha \in S$. Fix 
    $\beta \in (\acc(E^*) \cap S^\lambda_{\aleph_0}) \setminus (\alpha + 1)$. 
    Then we have $E^* \cap \beta = E_\beta = C_\beta$. Moroever, by the fact that 
    $\beta \in \Gamma \setminus S$, there exists $\delta \in S$ such that $\beta \in \acc(E_\delta)$, 
    so, by Claim \ref{claim: x_claim}, we have $\beta \in X(\mc{C})$. Fix 
    $\vec{\gamma} \in I(\mc{C}) \cap [D]^n$ such that $\beta \in \acc(C_{\vec{\gamma}})$. 
    It follows that 
    \[
      E_\alpha = E^* \cap \alpha = C_\beta \cap \alpha = C_{\vec{\gamma}} \cap \alpha.
    \]
    In particular, $\alpha \in X(\mc{C})$, and hence, by coherence, we have 
    $E_\alpha = C_{\vec{\gamma}} \cap \alpha = C_\alpha$, so again $E^* \cap \alpha = C_\alpha$.
  \end{proof}
  But now $E^*$ witnesses that $\mc{C}$ is not a $\Sq{n}^w(D)$-sequence 
  in $V[G \ast H]$, thus establishing the theorem.
\end{proof}
The next result will prove Theorem A.
\begin{corollary}
  Let $\kappa_0 = \aleph_1$ and let $\langle \kappa_n \mid 1 \leq n < \omega \rangle$ be 
  an increasing sequence of supercompact cardinals, and let $\P = \langle \P_n, \dot{\Q}_m 
  \mid n \leq \omega, \ m < \omega \rangle$ be a full-support forcing iteration such that, 
  for all $n < \omega$, we have
  \[
    \Vdash_{\P_n}``\dot{\Q}_n = \dot{\Coll}(\kappa_n, {<}\kappa_{n+1}).
  \]
  Then, in $V^{\P}$, for every positive integer $n$, every ordinal $\lambda \in 
  \mathrm{cof}({>}\aleph_n)$, and every club $D \subseteq \lambda$, 
  $\Sq{n}^w(D)$ fails. In particular, in $V^{\P}$,
  $\aleph_n$ is the unique regular cardinal $\kappa$ for which 
  $\Sq{n}^w(\kappa)$ holds.
\end{corollary}

\subsection{Forcing to add higher square sequences}

In this subsection, we introduce a forcing to add a $\Sq{n}^s(\lambda)$-sequence 
to a regular uncountable cardinal $\lambda \geq \aleph_n$. Fix a positive integer 
$n$ and a regular uncountable $\lambda \geq \aleph_n$. Define a forcing 
$\bb{S} = \bb{S}(\lambda,n)$ as follows. Conditions in $\bb{S}$ are all coherent 
$n$-$C$-sequences $s = \langle C^s_{\vec{\gamma}} \mid 
\gamma \in I(s) \rangle$ such that $s$ is a coherent $n$-$C$-sequence on 
$\delta^s + 1$ for some $\delta^s < \lambda$. If $s,s' \in \bb{S}$, then
$s' \leq_{\bb{S}} s$ if and only iff $s'$ end-extends $s$, i.e.,
\begin{itemize}
  \item $\delta^{s'} \geq \delta^s$;
  \item $I(s') \cap [\delta^s]^{\leq n} = I(s)$;
  \item $C^s_{\vec{\gamma}} = C^{s'}_{\vec{\gamma}}$ for all nonempty 
  $\vec{\gamma} \in [\delta^s]^{\leq n}$.
\end{itemize}

\begin{lemma} \label{lemma: strategic_closure}
  $\bb{S}$ is $\lambda$-strategically closed.
\end{lemma}

\begin{proof}
  We describe a winning strategy for Player II in $\Game_\lambda(\bb{S})$. 
  Given a (partial) run $\langle s_\eta \mid \eta < \xi \rangle$ of the game, 
  for readability we will write $C^\eta_{\vec{\gamma}}$ and $\delta^\eta$ 
  instead of $C^{s_\eta}_{\vec{\gamma}}$ and $\delta^{s_\eta}$, respectively. 
  We will arrange so that, if Player II plays according to their winning strategy, 
  then
  \begin{enumerate}
    \item $\langle \delta^\eta \mid \eta < \xi, ~ \eta \textrm{ even}\rangle$ 
    is strictly increasing;
    \item for every limit ordinal $\eta < \xi$, we have $\delta^\eta = 
    \sup\{\delta^\zeta \mid \zeta < \eta\}$;
    \item for all $m \leq n$ and all increasing sequences $\langle \eta_i 
    \mid i < m \rangle$ of limit ordinals below $\xi$, we have 
    $\langle \delta^{\eta_i} \mid i < m \rangle \in I(s_{\eta_{m-1}})$ and
    $C^{\eta_{m-1}}_{\langle \delta^{\eta_i} \mid i < m \rangle} = 
    \{\delta^\zeta \mid \zeta < \eta_0\}$.
  \end{enumerate}
  
  With the above requirements in place, the description of Player II's winning 
  strategy is now almost automatic. Suppose that $\xi < \lambda$ is an even 
  ordinal and $\langle s_\eta \mid \eta < \xi \rangle$ is a partial play of 
  the game with Player II playing thus far to ensure requirements (1)--(3) 
  above. If $\xi$ is a successor ordinal, say $\xi = \xi_0 + 1$, then 
  set $\delta^\xi = \delta^{\xi_0}+1$. Let $I(s_\xi) = I(s_{\xi_0}) \cup 
  \{\langle \delta^\xi \rangle, \langle \delta^{\xi_0},\delta^\xi \rangle\}$, 
  $C^\xi_{\delta^\xi} = \{\delta^{\xi_0}\}$, and $C^\xi_{\delta^{\xi_0}\delta^\xi} 
  = \emptyset$.  This completes the description of $s^\xi$. 
  
  If $\xi$ is a limit ordinal, then set $\delta^\xi = \sup\{\delta^\eta \mid 
  \eta < \xi\}$. Set
  \begin{align*}
    I(s_\xi) = & \bigcup\{I(s_\eta) \mid \eta < \xi\} \cup 
    \{\langle \delta^{\eta_i} \mid i < |\vec{\eta}| \rangle ^\frown 
    \langle \delta^\xi \rangle \mid \vec{\eta} \in [\acc(\xi)]^{<n} \} \\ 
    & \cup \{\langle \delta^{\eta},\delta^\xi \rangle \mid \eta \in 
    \nacc(\xi)\} \cup \{\langle \delta^{\eta}, \delta^{\eta+1},\delta^\xi \rangle 
    \mid \eta < \xi\}.
  \end{align*}
  Finally, set $C_{\delta^\xi} = \{\delta^\eta \mid \eta < \xi\}$ and, for all
  nonempty $\vec{\eta} \in [\acc(\xi)]^{<n}$, set $C^\xi_{\langle \delta^{\eta_i} \mid 
  i < |\vec{\eta}|\rangle \delta^\xi} = \{\delta^\zeta \mid \zeta < \eta_0\}$. 
  Set $C^\xi_{\delta^0 \delta^\xi} = \emptyset$ and, for all $\eta < \xi$, set 
  $C^\xi_{\delta^{\eta+1},\delta^\xi} = \{\delta^\eta\}$ and 
  $C^\xi_{\delta^\eta,\delta^{\eta+1},\delta^\xi} = \emptyset$. It is straightforward 
  to verify that this defines a lower bound to $\langle s_\eta \mid \eta < \xi 
  \rangle$ in $\bb{S}$ and continues to satisfy requirements (1)--(3) above, 
  thus completing the proof of the lemma.
\end{proof}

Lemma \ref{lemma: strategic_closure} implies that forcing with $\bb{S}$ preserves 
all cardinalities and cofinalities ${\leq} \lambda$.
By the proof of Lemma \ref{lemma: strategic_closure}, for all 
$\alpha < \lambda$, the set $E = \{s \in \bb{S} \mid \delta^s \geq \alpha\}$ 
is dense in $\bb{S}$. As a result, if $G$ is $\bb{S}$-generic over 
$V$, then, in $V[G]$, we can define a coherent $n$-$C$-sequence $\mc{C}$ on 
$\lambda$ by setting $I(\mc{C}) = \bigcup \{I(s) \mid s \in G\}$, 
$C_\emptyset = \lambda$, and, for all nonempty $\gamma \in I(\mc{C})$, 
$C_{\vec{\gamma}} = C^s_{\vec{\gamma}}$ for any $s \in G$ such that 
$\vec{\gamma} \in I(s)$. In $V$, let $\dot{\mc{C}}$ be an $\bb{S}$-name for 
$\mc{C}$. The following lemma will imply that 
$\dot{\mc{C}}$ is forced to be a $\Sq{n}^s(\lambda)$-sequence in $V^{\bb{S}}$.

\begin{lemma}
  If $n = 1$, then 
  \[
    \Vdash_{\bb{S}} ``\{\delta \in S^\lambda_\omega \mid \otp(C_\delta) = 
    \omega\}" \text{ is stationary in }\lambda".
  \]
  If $n > 1$, then
  \[
    \Vdash_{\bb{S}} ``\{\delta \in S^\lambda_{\omega_{n-1}} \mid \dot{\mc{C}}^\delta 
    \text{ is order-minimal}\} \text{ is stationary in } \lambda".
  \]
\end{lemma}

\begin{proof}
  We provide the proof in case $n > 1$. The proof in case $n = 1$ is similar 
  and already appears in the literature. 
  Fix a coherent order-minimal $(n-1)$-$C$-sequence $\mc{D}$ on $\omega_{n-1}$, 
  and note that $X(\mc{D}) = \emptyset$. 
  Fix a condition $s_0 \in \bb{S}$ and an $\bb{S}$-name $\dot{E}$ for a club 
  in $\lambda$. We will recursively define a decreasing sequence $\langle s_\eta \mid 
  \eta \leq \omega_{n-1} \rangle$. Among other things, our construction will 
  ensure that $\langle \delta^\eta \mid \eta \leq \omega_{n-1} \rangle$ is 
  strictly increasing and continuous at limits, and that for every limit 
  ordinal $\eta < \omega_{n-1}$ and every $\xi$ in the interval $[\eta,\omega_{n-1}]$, 
  we have $\delta^\eta \notin X(s_\xi)$. Since $\langle \delta^\eta \mid \eta 
  < \omega_{n-1} \rangle$ will be continuous, it suffices to explicitly ensure 
  this only for limit ordinals $\eta < \omega_{n-1}$ with $\cf(\eta) = \omega$.
  
  Suppose first that $\eta < \omega_{n-1}$ and we are given $s_\eta$. 
  If $\cf(\eta) = \omega$, then note that $\cf(\delta^\eta) = \omega$, and first 
  extend $s_\eta$ to $s'_\eta$ as follows. Set $\delta^{s'_\eta} = \delta^\eta + 1$, 
  and let $C^{s'_\eta}_{\delta^{s'_\eta}} = A_\eta \cup \{\delta^\eta\}$, where 
  $A_\eta$ is an $\omega$-sequence converging to $\delta^\eta$ such that 
  $A_\eta \neq C^\eta_{\delta_\eta}$. Set $C^{s'_\eta}_{\delta^\eta,\delta^{s'_\eta}} 
  = A_\eta$, and extend to an $n$-$C$-sequence in the obvious 
  way: set $C^{s'_\eta}_{A_\eta(0), \delta^{s'_\eta}} = C^{s'_\eta}_{A_\eta(0), 
  \delta^\eta,\delta^{s'_\eta}} = \emptyset$, and, for all $k < \omega$, set 
  $C^{s'_\eta}_{A_\eta(k+1),\delta^{s'_\eta}} = C^{s'_\eta}_{A_\eta(k+1),\delta^\eta, 
  \delta^{s'_\eta}} = \{A_\eta(k)\}$ and 
  $C^{s'_\eta}_{A_\eta(k),A_\eta(k+1),\delta^{s'_\eta}} = 
  C^{s'_\eta}_{A_\eta(k),A_\eta(k+1),\delta^\eta,\delta^{s'_\eta}} = \emptyset$. 
  The point of this is that, since $C^{s'_\eta}_{\delta^\eta} \neq 
  C^{s'_\eta}_{\delta^\eta,\delta^{s'_\eta}}$, and both are cofinal in 
  $\delta^\eta$, for any $t \leq_{\bb{S}} s'_\eta$, we must have 
  $\delta^\eta \notin X(t)$. If $\cf(\eta) \neq \omega$, then simply let 
  $s'_\eta = s_\eta$. Now find an ordinal $\beta_\eta \in [\delta^\eta,\lambda)$ 
  and a condition $s_{\eta+1} \leq s'_\eta$ such that $\delta^{\eta+1} > \beta_\eta$ 
  and $s_{\eta+1} \Vdash_{\bb{S}}``\beta_\eta \in \dot{E}"$.
  
  Now suppose that $\xi < \omega_{n-1}$ is a limit ordinal and we are given 
  $\langle s_\eta \mid \eta < \xi \rangle$. Define a lower bound $s_\xi$ as follows. 
  Let $\delta^\xi = \sup\{\delta^\eta \mid \eta < \xi\}$, and let 
  \[
    I(s_\xi) = \bigcup\{I(s_\eta) \mid \eta < \xi\} \cup 
    \{\langle \delta^{\eta_i} \mid i < |\vec{\eta}| \rangle \mid 
    \vec{\eta} \in I(\mc{D}) \cap [\xi+1]^{{<}n} \wedge 
    \max(\vec{\eta}) = \xi\}.
  \]
  For $\vec{\eta} \in I(\mc{D}) \cap [\xi+1]^{{<}n}$ such that 
  $\max(\vec{\eta}) = \xi$, let $C^{\xi}_{\langle \delta^{\eta_i} \mid 
  i < |\vec{\eta}|} = \{\delta^\zeta \mid \zeta \in D_{\vec{\eta}}\}$.
  The fact that, for all limit $\zeta < \xi$, we have 
  $\delta^\zeta \notin \bigcup \{X(s_\eta) \mid \eta < \xi\}$ 
  ensures that $s_\xi$ remains $n$-coherent, so $s_\xi$ is indeed a lower bound 
  for $\langle s_\eta \mid \eta < \xi \rangle$. 
  
  Finally, if $\xi = \omega_{n-1}$ and $\langle s_\eta \mid \eta < \xi \rangle$ 
  is given, define $s_\xi$ as follows. Set $\delta^{\xi} = \sup\{\delta^\eta 
  \mid \eta < \xi\}$ and 
  \[ 
    I(s_\xi) = \bigcup \{I(s_\eta) \mid \eta < \xi\} \cup 
    \{\langle \delta^{\eta_i} \mid i < |\vec{\eta}| \rangle^\frown 
    \langle \delta^\xi \rangle \mid \vec{\eta} \in I(\mc{D}) \}.
  \]
  Set $C^\xi_{\delta^\xi} = \{\delta^\eta \mid \eta < \xi\}$ and, for 
  nonempty $\vec{\eta} \in \mc{D}$, set 
  $C^\xi_{\langle \delta^{\eta_i} \mid i < |\vec{\eta}| \rangle^\frown 
  \langle \delta^\xi \rangle} = \{\delta^\zeta \mid \zeta \in D_{\vec{\eta}}\}$.
  Again, the fact that $\delta^\zeta \notin \bigcup \{X(s_\eta) \mid \eta < \xi\}$ 
  for all limit $\zeta < \xi$ ensures that $s_\xi$ remains $n$-coherent and is 
  thus a lower bound for $\langle s_\eta \mid \eta < \xi \rangle$. 
  Moreover, for all $\eta < \omega_{n-1}$, we have 
  $\delta^\eta \leq \beta_\eta < \delta^{\eta+1}$ and 
  $s_{\omega_{n-1}} \Vdash ``\beta_\eta \in \dot{E}"$. It follows that 
  $s_{\omega_{n-1}} \Vdash ``\delta^{\omega_{n-1}} \in \dot{E}"$. Moreover, 
  our construction guarantees that, via the unique order-preserving map from 
  $\{\delta^\eta \mid \eta < \omega_{n-1}\}$ to $\omega_{n-1}$, the 
  coherent $(n-1)$-$C$-sequence $\dot{\mc{C}}^{\delta^{\omega_{n-1}}}$ is 
  isomorphic to $\mc{D}$ and is therefore order-minimal. Since $\dot{E}$ was 
  an arbitrary $\bb{S}$-name for a club in $\lambda$, this establishes the 
  lemma.
\end{proof}

The above lemmata establish the following result.

\begin{theorem}
  Suppose that $n$ is a positive integer and $\lambda \geq \aleph_n$ is a 
  regular uncountable cardinal. Then forcing with $\bb{S} = \bb{S}(\lambda, n)$ 
  preserves all cardinalities and cofinalities ${\leq}\lambda$ and 
  \[
    \Vdash_{\bb{S}} \Sq{n}^s(\lambda).
  \]
  Moreover, if $\lambda^{<\lambda} = \lambda$, then $|\bb{S}| = \lambda$, and 
  hence forcing with $\bb{S}$ preserves all cardinalities and cofinalities.
\end{theorem}

\section{Higher walks}\label{Walks}

\subsection{Preliminaries}

In this section, we recall the machinery of higher-dimensional walks from 
\cite{bergfalk2024introductionhigherwalks} and then prove that coherence of the 
$n$-$C$-sequence used for walks is sufficient to guarantee coherence of the 
associated $\rho^n_2$ and $\resh_n$ functions. We note that the $n = 1$ case is 
classical, while a version of the $n = 2$ case, using a stronger notion of coherence 
for $n$-$C$-sequences than we employ here, was proven in 
\cite[Theorem 7.2(3)]{bergfalk2024introductionhigherwalks}.
Before getting to the main result, Theorem \ref{mainn}, which deals with an arbitrary $n\ge 1$, we will first prove Theorem \ref{main3}, dealing with the special case $n=3$. This will introduce all of the important ideas that go into the proof of Theorem \ref{mainn}. Having said this, before getting to Theorem \ref{main3}, we will need to introduce a slew of terminology and lemmas. These will be given, for the most part, in their most general version (i.e. an arbitrary $n$ rather than $3$), for the jump from $3$ to $n$ does not represent a hurdle in understanding. We will permit ourselves to break this rule a few times though, most notably in Lemmas \ref{alphaAlphaTerminal}, \ref{badImmediate} and \ref{badNodes}, as the diagrams included in their proofs only make sense if $n=3$. We have strived to maintain a balance between economizing repetitions and expository considerations, and hope that the reader is content with the outcome.

\begin{definition}
  Let $n$ be a positive integer. An \emph{$n$-tree} is a subset $S \subseteq 
  {^{<\omega}}n$ that is downward closed, i.e., closed under taking initial segments. 
  Given an $n$-tree $S$, the elements of $S$ are referred to as \emph{nodes} and, 
  given $x \in S$, we say that $x$ is a \emph{terminal node} of $S$ if there does 
  not exist $i < n$ such that $x^\smallfrown \langle i \rangle \in S$. We say that 
  an $n$-tree $S$ is \emph{full} if, for all $x \in S$, one of the following two 
  alternatives holds:
  \begin{enumerate}
    \item $x$ is a terminal node of $S$;
    \item for all $i < n$, $x^\smallfrown \langle i \rangle \in S$.
  \end{enumerate}
  If $S$ is a full $n$-tree and $x \in S$ is not a terminal node of $S$, then we call 
  $x$ a \emph{splitting node} of $S$.
\end{definition}

We are now ready to introduce the higher walks machinery, which will describe, given an 
$n$-$C$-sequence $\mc{C}$ on a club $D$, a method of performing walks from 
$(n+1)$-tuples $\vec{\gamma} \in D^{[n+1]}$. The most important information associated with 
this walk will be recorded in the order-$n$ upper trace function $\Tr_n^{\mc{C}}$ defined 
below. We should note that our $\Tr_n^{\mc{C}}$ function is cosmetically different from 
that defined in \cite{bergfalk2024introductionhigherwalks}; on its face our upper trace function  
sees to record more data about the walk, though given the $n$-$C$-sequence $\mc{C}$, the two functions 
are informationally equivalent.

Recall that, if $\vec{\gamma}$ is a sequence of length $n$ and $i < n$, then $\vec{\gamma}^i$ 
denotes the sequence of length $n - 1$ obtained by removing $\gamma_i$ from $\vec{\gamma}$.
For notational simplicity, we will assume throughout this section that our $n$-$C$-sequences 
are on ordinals $\lambda$ rather than arbitrary clubs $D$. The general case involves no new ideas 
and can be reduced to this special case via the order-preserving bijection between a club 
$D$ and its order type.

\begin{definition}\label{Tr}
Given an $n$-$C$-sequence $\C$ on an ordinal $\lambda$, the \emph{order-$n$ upper trace function} $\Tr^{\C}_n$  
has domain $\{-1,1\} \times \lambda^{[n+1]}$ and is defined as follows: first, if $\vec\gamma\in\lambda^{[n+1]}$, write $\vec{\gamma}=\iota(\vec\gamma)^\smallfrown\tau(\vec\gamma)$, where $\tau(\vec\gamma)$ is the longest final segment of $\vec\gamma$ such that $\tau(\vec\gamma)\in I(\C)$. Note that $\tau(\vec\gamma)\neq\vec\gamma$, so we can set $j:=|\iota(\vec\gamma)|-1$. 

Suppose now that 
$(\alpha,\vec{\gamma}) \in \lambda^{[n+1]}$. We will recursively define a function
\[
\Tr^{\C}_n((-1)^k,\alpha,\vec\gamma):S^{\C}_n(\alpha,\vec\gamma)\to \{-1,1\}\times \lambda^{[n+1]},
\]
where $S^{\C}_n(\alpha,\vec\gamma)\subset {}^{<\w}n$ is a non-empty full $n$-tree that is constructed together with $\Tr^{\C}_n((-1)^k,\alpha,\vec\gamma)$. We think of $\Tr^{\C}_n((-1)^k,\alpha,\vec\gamma)$ as a labelled tree, and we will think of its labels, i.e., of elements of 
$\{-1,1\} \times \lambda^{[n+1]}$, as sequences of length $n+2$. So, for 
example, if $x \in S^{\C}_n(\alpha,\vec{\gamma})$ and 
$\Tr^{\C}_n((-1)^k,\alpha,\vec{\gamma})(x) = (-1,\alpha,\vec{\gamma}')$, 
then the expression $(\Tr^{\C}_n((-1)^k,\alpha,\vec{\gamma})(x))^1$ would denote the 
sequence $(-1,\vec{\gamma}') \in \{-1,1\} \times \lambda^{[n]}$. The underlying tree, $S^{\C}_n(\alpha,\vec\gamma)$, will depend only on the ordinal sequence $(\alpha,\vec\gamma)$ and not on the sign $(-1)^k$. The first ordinal entry, i.e. the entry in position $1$, of $\Tr^{\C}_n((-1)^k,\alpha,\vec\gamma)(x)$ will be $\alpha$ for every $x\in S^{\C}_n(\alpha,\vec\gamma)$.
\begin{itemize}
\item To start the recursion, decree that $\emptyset\in S^{\C}_n(\alpha,\vec\gamma)$ and set
\[
\Tr^{\C}_n((-1)^k,\alpha,\vec\gamma)(\emptyset):=((-1)^k,\alpha,\vec\gamma).
\]
\item Suppose that $x\in S^{\C}_n(\alpha,\vec\gamma)$ and let $((-1)^m,\vec\beta)=\Tr^{\C}_n((-1)^k,\alpha,\vec\gamma)(x)$, with $\beta_0=\alpha$. Set $j+1=|\iota(\vec\beta)|$.
\begin{itemize}
\item If $C_{\tau(\vec\beta)}\setminus\beta_j=\emptyset$, then $x$ is a terminal node of $S^{\C}_n(\alpha,\vec\gamma)$.
\item If $C_{\tau(\vec\beta)}\setminus\beta_j\neq\emptyset$, let $\seq{\ell_i:i<n}$ be the increasing enumeration of the set $\{1,\dots,n+1\}\setminus \{j+1\}$. We demand that $x^\smallfrown\seq{i}\in S^{\C}_n(\alpha,\vec\gamma)$ for every $i<n$ and set
\[
\Tr^{\C}_n((-1)^k,\alpha,\vec\gamma)(x^\smallfrown\seq{i}):=((-1)^{m+j+\ell_i},(\iota(\vec\beta),\min(C_{\tau(\vec\beta)}\setminus\beta_j),\tau(\vec\beta))^{\ell_i}).
\]
\end{itemize}
\end{itemize}
If $m\in\w$, we will write $(-1)^m\Tr_n^\C((-1)^k,\alpha,\vec\gamma)(x)$ to denote the result of multiplying the first coordinate of $\Tr_n^\C((-1)^k,\alpha,\vec\gamma)(x)$ by $(-1)^m$. More precisely, if $\Tr_n^\C((-1)^k,\alpha,\vec\gamma)(x)=((-1)^\ell,\alpha,\vec\beta)$, then 
\[
(-1)^m\Tr_n^\C((-1)^k,\alpha,\vec\gamma)(x)=((-1)^{\ell+m},\alpha,\vec\beta)
\]
As mentioned above, the signs play no role in determining the tree $S^{\C}_n(\alpha,\vec\gamma)$. For much of the analysis that awaits us, the signs in $\Tr^{\C}_n((-1)^k,\alpha,\vec\gamma)$ are also irrelevant. We will therefore introduce 
the symbol $\Tr^{\C}_n(\alpha,\vec\gamma)$ to denote the unsigned version of the upper trace function. This can be defined by a recursion similar to the one above, but can also be recovered from the signed upper trace by letting $\Tr^{\C}_n(\alpha,\vec\gamma)=(\Tr^{C}_n(+,\alpha,\vec\gamma))^0$. We will cavalierly refer to $\Tr^{\C}_n(\alpha,\vec\gamma)$ as the \emph{walk (down) from} $(\alpha,\vec\gamma)$. Given $x\in S^{\C}_n(\alpha,\vec\gamma)$, the \emph{step (down) from $x$} is the set
\[
\Tr^{\C}_n(\alpha,\vec\gamma)\restr\left(\{x\}\cup \{x^\smallfrown\seq{i}:i<n\}\right).
\]
If $x\in S^{\C}_n(\alpha,\vec\gamma)$, then $x^\smallfrown\seq{i}$ is said to be an \emph{immediate successor} of $x$ whenever $x^\smallfrown\seq{i}\in S^{\C}_n(\alpha,\vec\gamma)$. The same terminology applies to the signed versions of these objects. A sequence $\vec{\delta}$ is said to \emph{occur along the walk down from $(\alpha,\vec\gamma)$} iff there exists $x\in S^{\C}_n(\alpha,\vec\gamma)$ such that $\Tr^{\C}_n(\alpha,\vec\gamma)(x)=\vec\delta$.

\begin{remark}
  While it is not immediately evident from the definition, a well-foundedness argument 
  implies that higher walks are always \emph{finite}, i.e., the tree 
  $S_n^{\mc{C}}(\alpha,\vec{\gamma})$ is finite for all $(\alpha,\vec{\gamma}) \in \lambda^{[n+1]}$. 
  We refer the reader to \cite[Theorem 5.4]{bergfalk2024introductionhigherwalks} for a proof of 
  this fact.
\end{remark}

\begin{remark}
If the $n$-$C$-sequence $\C$ is clear from context, it will sometimes be omitted 
from $\Tr_n^{\C}$, $S_n^{\C}$, and similar notation introduced below.
We will frequently blur the distinction between a node $x\in S_n(\alpha,\vec\gamma)$ and its label $\Tr_n(\alpha,\vec\gamma)(x)$, referring, for example, to ``the terminal node $(\alpha,\vec\beta)$", where $(\alpha,\vec\beta)=\Tr_n(\alpha,\vec\gamma)(x)$. The same conventions will be used with other walk-related concepts. For example, if $(\alpha,\vec\beta)=\Tr_n(\alpha,\vec\gamma)(x)$, we will speak of the ``step down from $(\alpha,\vec\beta)$" or ``immediate successors of $(\alpha,\vec\beta)$". While this is not entirely unambiguous (for different nodes on $S_n(\alpha,\vec\gamma)$ can have the same label), the point is that determining whether $x$ is terminal or not, or determining the structure of the step down from $x$, depends only on its label $(\alpha,\vec\beta)$. 
\end{remark}

\end{definition}
The next lemma says that a walk can be restarted at any node along itself, without affecting what happens afterwards:
\begin{lemma} \label{lemma: restart}
Let $\C$ be an $n$-$C$-sequence on an ordinal $\lambda$. Suppose that 
$(\alpha,\vec{\gamma}) \in \lambda^{[n+1]}$ and $k\in\w$. Let $x\in S_n(\alpha,\vec\gamma)$ and set $((-1)^m,\alpha,\vec\beta)=\Tr_n((-1)^k,\alpha,\vec\gamma)(x)$. Then
\[
\{z\in S_n(\alpha,\vec\gamma): x\subset z\}=\{x^\smallfrown y:y\in S_n(\alpha,\vec\beta)\}
\]
and, if $y\in S_n(\alpha,\vec\beta)$,
\[
\Tr_n((-1)^m,\alpha,\vec\beta)(y)=\Tr_n((-1)^k,\alpha,\vec\gamma)(x^\smallfrown y).
\]
The analogous statements for the unsigned $\Tr_n(\alpha,\vec\gamma)$ are true as well.
\end{lemma}
\begin{proof}
Immediate from the recursive definition of $\Tr_n$.
\end{proof}

Recall that $\lambda \otimes [\lambda]^n$ denotes the set of all 
$(\alpha,\vec{\gamma}) \in \lambda^{[n+1]}$ such that $\vec{\gamma} \in [\lambda]^n$.
While we have defined $\Tr_n$ and $S_n$ on $\lambda^{[n+1]}$ (or $\{-1,1\} \times 
\lambda^{[n+1]}$ for the signed version of $\Tr_n$), we will be most interested in its values 
on $\lambda \otimes [\lambda]^n$. The following proposition indicates that we can safely 
restrict our attention to this domain.

\begin{proposition} \label{prop: increasing}
  Let $\mc{C}$ be an $n$-$C$-sequence on an ordinal $\lambda$ and $(\alpha,\vec{\gamma}) 
  \in \lambda \otimes [\lambda]^n$. Then, for all $x \in S_n(\alpha,\vec{\gamma})$, we have 
  $\Tr_n(\alpha,\vec{\gamma}) \in \lambda \otimes [\lambda]^n$.
\end{proposition}

\begin{proof}
  The straightforward proof, by induction on $|x|$ using Definition \ref{Tr} and 
  Lemma \ref{lemma: restart}, is left to the reader.
\end{proof}

Note that, in the case $n = 1$, the walk described above from a pair $(\beta,\gamma) \in \lambda^{[2]}$ 
recreates the classical walk from $\gamma$ to $\beta$ and, modulo cosmetic differences, the 
upper trace function $\Tr_1$ is equivalent to the classical upper trace function $\Tr$.

Recall that, in the classical, one-dimensional setting, if we have some (one-dimensional) $C$-sequence on an ordinal $\lambda$ and we are using it to walk down from $\gamma$ to $\beta$ for 
some $\beta < \gamma < \lambda$, then the \emph{lower trace} $\L(\beta,\gamma)$ is a finite set of ordinals such that, if $\alpha < \beta$ and $\max(\L(\beta,\gamma))<\alpha<\beta$, then the walk from $\gamma$ down to $\alpha$ starts like the walk from $\gamma$ down to $\beta$, reaches $\beta$, and then continues like the walk from $\beta$ down to $\alpha$, i.e., $\Tr(\alpha,\gamma) = \Tr(\beta,\gamma) \cup \Tr(\alpha,\beta)$. In the higher-dimensional setting, an analogous function, $\L_n$, was introduced by Bergfalk in \cite[Definition 7.5]{bergfalk2024introductionhigherwalks}. Before reproducing it below (see Definition \ref{lowerTrace}), we make a small observation. Fix $n>0$. Let $\C$ be an $n$-$C$-sequence on an ordinal $\lambda$ and $(\alpha,\vec{\gamma}) \in \lambda^{[n+1]}$. We wish to compare $\Tr_n(\alpha,\vec\gamma)$ with $\Tr_n(\xi,\vec\gamma)$ for $\xi<\alpha$. We note that $\tau(\alpha,\vec\gamma)=\tau(\xi,\vec\gamma)$ and, if $|\tau(\alpha,\vec{\gamma})| < n$, 
then the step from $(\xi,\vec\gamma)$ looks exactly like that from $(\alpha,\vec\gamma)$, with the $\alpha$'s in coordinate $0$ being replaced with $\xi$'s.

\begin{definition}\label{extreme}
Fix $n>0$. Let $\C$ be an $n$-$C$-sequence on an ordinal $\lambda$ 
and $(\alpha,\vec{\gamma}) \in \lambda^{[n+1]}$. Define
\[
E_n(\alpha,\vec\gamma):=\{x\in S_n(\alpha,\vec\gamma):|\tau(\Tr_n(\alpha,\vec\gamma)(x))|=n\}.
\]
In words, $E_n(\alpha,\vec\gamma)$ is the set of all nodes $x\in S_n(\alpha,\vec\gamma)$ such that, when stepping down from $x$, we use a club of $\C$ whose index has the maximum possible length (namely, $n$).
\end{definition}

\begin{definition}\label{lowerTrace}
Fix $n>0$. Let $\C$ be an $n$-$C$-sequence on an ordinal $\lambda$ and $(\alpha,\vec{\gamma}) \in \lambda^{[n+1]}$. The \emph{order-$n$ lower trace function} $\L_n(\alpha,\vec\gamma):S_n(\alpha,\vec\gamma)\to \lambda$ is defined by
\[
\L_n(\alpha,\vec\gamma)(x):=\max\{\sup(\alpha\cap C_{(\Tr_n(\alpha,\vec\gamma)(y))^0}):y\subset x \wedge y\in E_n(\alpha,\vec\gamma)\}
\]
whenever $x\in S_n(\alpha,\vec\gamma)$, with the convention that $\max(\emptyset)=0$.
\end{definition}

\begin{notation}
If $f$ is a function, $I\subset \dom(f)$ and $t$ is a function with $\dom(t)=I$, then we let $\sub_I^t(f):=t\cup f\restr(\dom(f)\setminus I)$.

If $I=\{i_0,\dots,i_n\}$ is a finite set and $t(i_k)=\xi_k$, we will write $\sub_{i_0,\dots,i_n}^{\xi_0,\dots\xi_n}(f)$ instead of $\sub_I^t(f)$. We will only need to use this notation for $n\in \{0,1\}$.
\end{notation}

\begin{lemma}\label{7.7}
Fix $n>0$. Let $\C$ be an $n$-$C$-sequence on an ordinal $\lambda$ and $(\alpha,\vec{\gamma}) \in \lambda^{[n+1]}$. Let $x\in S_n(\alpha,\vec\gamma)$ and $\xi<\lambda$ be such that
\[
\L_n(\alpha,\vec\gamma)(x\restr(|x|-1))<\xi\le \alpha.
\]
Then $x\in S_n(\xi,\vec\gamma)$ and, for every $j\le|x|$,
\[
\Tr_n(\pm,\xi,\vec\gamma)(x\restr j)=\sub_1^\xi(\Tr_n(\pm,\alpha,\vec\gamma)(x\restr j)).
\]
In particular, if $\max(\L_n(\alpha,\vec\gamma))<\xi\le \alpha$, then $S_n(\xi,\vec\gamma)$ is an end-extension of $S_n(\alpha,\vec\gamma)$, and the labels $\Tr(\pm,\xi,\vec\gamma)$ on the initial segment $S_n(\xi,\vec\gamma)$ are obtained by taking the labels given by $\Tr(\pm,\alpha,\vec\gamma)$ and switching the $\alpha$ in coordinate $1$ with a $\xi$.
\end{lemma}

\begin{proof}
See \cite[Lemma 7.7]{bergfalk2024introductionhigherwalks}. 
\end{proof}

\begin{definition}\label{sigma}
Fix $n > 0$. Let $\C$ be an $n$-$C$-sequence on an ordinal $\lambda$ and $(\alpha,\vec{\gamma}) \in \lambda^{[n+1]}$. If $x\in S_n(\alpha,\vec\gamma)$, then $\sigma_{\alpha\vec\gamma}(x)$ denotes the signed basis element of the free abelian group $\bigoplus_{\lambda^{[n-1]}}\Z$ corresponding to $\Tr_n(+,\alpha,\vec\gamma)(x)$ with the first two ordinal coordinates deleted. More precisely,
\[
\sigma_{\alpha\vec\gamma}(x):=\lfloor ((\Tr_n(+,\alpha,\vec\gamma)(x))^1)^1\rfloor.
\]

We also let $\sgn_{\alpha\vec\gamma}(x)$ denote the sign of $\Tr_n(+,\alpha,\vec\gamma)(x)$, so that $\sgn_{\alpha\vec\gamma}(x)\in\{-1,1\}$. More formally,
\[
\sgn_{\alpha\vec\gamma}(x)=[\Tr_n(+,\alpha,\vec\gamma)(x)](0).
\]
Observe that the sign of $\Tr_n((-1)^k,\alpha,\vec\gamma)(x)$ is then $(-1)^k\sgn_{\alpha\vec\gamma}(x)$.
\end{definition}

We are now almost ready to introduce the walk-characteristic that will be the main object of study for 
the remainder of the paper. We first recall the ``generalized number of steps function'' $\rho_2^n$ 
from \cite{bergfalk2024introductionhigherwalks}.

\begin{definition}
  Fix $n > 0$, and let $\C$ be an $n$-$C$-sequence on an ordinal $\lambda$. Then the function 
  $\rho_2^n : \{-1,1\} \times \lambda^{[n+1]} \ra \bb{Z}$ is defined by setting
  \[
    \rho_2^n((-1)^k,\alpha,\vec{\gamma}) = \sum_{x \in S_n(\alpha,\vec{\gamma})} (-1)^k 
    \sgn_{\alpha\vec{\gamma}}(x)
  \]
  for all $k < \omega$ and $(\alpha,\vec{\gamma}) \in \lambda^{[n+1]}$. We also let 
  $\rho_2^n(\alpha,\vec{\gamma}) = \rho_2^n(+1,\alpha,\vec{\gamma})$.
\end{definition}

Note that, if $n = 1$, then $\rho_2^1$ is essentially the classical number of steps function $\rho_2$. 
More precisely, we have $\rho_2^1(\beta,\gamma) = \rho_2(\beta,\gamma) + 1$ for all 
$\beta \leq \gamma < \lambda$; this ``$+1$'' term is purely cosmetic and makes no difference 
in the relevant questions about coherence and triviality. In 
\cite[Theorem 7.2]{bergfalk2024introductionhigherwalks}, Bergfalk identifies some special 
cases in which the family 
\[
  \Phi(\rho_2^n) = \langle \rho_2^n(\cdot,\vec{\gamma}):\gamma_0 \ra \bb{Z} \mid 
  \vec{\gamma} \in [\lambda]^n \rangle
\]
is coherent modulo locally semi-constant functions. In particular, this holds if 
$\mc{C}$ is an order-type-minimal $n$-$C$-sequence on $\omega_n$. He asks for more general 
conditions on the $n$-$C$-sequence $\mc{C}$ that guarantee coherence of $\Phi(\rho_2^n)$. 
We show now that the coherence condition introduced in Definition \ref{def: n_c_sequence}(\ref{coherence_condition}) suffices. Looking ahead to an eventual proof of nontriviality in 
the next section, though, we prove this not directly for $\rho_2^n$ but for a richer function, 
$\resh_n$, which we now introduce.\footnote{The character $\resh$ is pronounced ``resh''. It is 
a letter in the Phoenician alphabet from which the Greek letter $\rho$ later descended.}

\begin{definition}\label{rho2}
Let $\C$ be a coherent $n$-$C$-sequence on an infinite ordinal $\lambda$ and $(\alpha,\vec{\gamma}) \in \lambda^{[n+1]}$. Then $\resh_n(\alpha,\vec\gamma)$ denotes the sum, in the free abelian group $\bigoplus_{\lambda^{[n-1]}}\Z$, of the outputs $\Tr_n(+,\alpha,\vec\gamma)(x)$ with the first two ordinal coordinates deleted, where $x$ rangers over $S_n(\alpha,\vec\gamma)$. More formally,
\[
\resh_n(\alpha,\vec\gamma):=\sum_{x\in S_n(\alpha,\vec\gamma)}\sigma_{\alpha\vec{\gamma}}(x)
\]
\end{definition}

\begin{remark}
  We will be most interested in the values $\resh_n(\alpha,\vec{\gamma})$ for 
  $(\alpha,\vec{\gamma}) \in \lambda \otimes [\lambda]^n$. Note that, by Proposition 
  \ref{prop: increasing}, we have $\resh_n(\alpha,\vec{\gamma}) \in \bigoplus_{[\lambda]^{n-1}} 
  \bb{Z}$ for all such $(\alpha,\vec{\gamma})$. In particular, the $n$-family
  \[
    \Phi(\resh_n) = \left\langle \resh_n(\cdot,\vec{\gamma}):\gamma_0 \ra 
    \bigoplus_{[\lambda]^{n-1}} \bb{Z} \ \middle| \ \vec{\gamma} \in [\lambda]^n
    \right\rangle
  \]
  consists of functions mapping into $\bigoplus_{[\lambda]^{n-1}}$.
\end{remark}

We first note that, in the case of $n = 1$, we have $\resh_1 = \rho_2^1$, while, for 
$n > 1$, it is a strictly richer object. There is a natural ``projection" from 
$\resh_n$ to $\rho_2^n$, though. More precisely, let 
$\varpi: \bigoplus_{\lambda^{[n-1]}} \Z \ra \Z$ be the group homomorphism defined by setting 
$\varpi(\lfloor \vec{\delta} \rfloor) = 1$
for all $\vec{\delta} \in \lambda^{[n-1]}$ and extending linearly. Then it follows immediately 
from the definitions that $\varpi(\resh_n(\alpha,\vec{\gamma})) = \rho_2^n(\alpha,\vec{\gamma})$ 
for all $(\alpha,\vec{\gamma}) \in \lambda^{[n+1]}$. It then follows that, if the $n$-family 
$\Phi(\resh_n)$ is coherent modulo locally semi-constant functions, then so is $\Phi(\rho_2^n)$.

Before turning directly to establishing the coherence of $\Phi(\resh_n)$, we will need some additional 
machinery. First, it will be convenient to have notation for a signed version of the tree $S_n$:
\begin{definition}\label{signS}
Let $\C$ be an $n$-$C$-sequence on an ordinal $\lambda$ and $(\alpha,\vec{\gamma}) \in 
\lambda^{[n+1]}$. Define
\[
S_n(\pm,\alpha,\vec\gamma):=\{(x,\pm\sgn_{\alpha\vec\gamma}(x)):x\in S_n(\alpha,\vec\gamma)\}.
\]
Note that $S_n(\pm,\alpha,\vec\gamma)\subset {}^{<\w}n \times \{-1,1\}$ and that
\[
S_n(\alpha,\vec\gamma)=\left\{t^{|t|-1}:t\in S_n(\pm,\alpha,\vec\gamma)\right\}.
\]
\end{definition}

\begin{lemma}[{\cite[Lemma 7.12]{bergfalk2024introductionhigherwalks}}]\label{7.12}
Let $\C$ be an $n$-$C$-sequence on an infinite ordinal $\lambda$ and $(\alpha,\vec{\gamma}) \in 
\lambda^{[n+2]}$. Then
\[
\bigsqcup_{i\le n}\bS_n(\alpha,\vec\gamma^i)
\]
admits a partition
\[
\bigcup_{t\in Z}\{(i(t),t),(i(t_-),t_-)\}
\]
such that
\[
\Tr_n((-1)^{i(t)},\alpha,\vec\gamma^{i(t)})(t^0)=-\Tr_n((-1)^{i(t_-)},\alpha,\vec\gamma^{i(t_-)})(t_-^0)
\]
for every $t\in Z$. 
\end{lemma}

\subsection{The case $n=3$} In this subsection, we present the proof that $\Phi(\resh_3)$ is 
coherent if $\mc{C}$ is a coherent $3$-$C$-sequence. We begin with a few technical lemmas.

\begin{lemma}\label{alphaAlphaTerminal}
Let $\C$ be a coherent $3$-$C$-sequence on an infinite ordinal $\lambda$. Let $\alpha<\beta<\gamma<\lambda$ and consider the node $(\alpha,\alpha,\beta,\gamma)$. Then every node of the form $(\alpha,\alpha',\beta',\gamma')$ with $\alpha<\alpha'$ occurring along the walk down from $(\alpha,\alpha,\beta,\gamma)$ must be terminal. In particular, we must have $\beta' \in C_{\gamma'}$. Moreover, there exists $\nu<\alpha$ such that, for every $\xi\in [\nu,\alpha]$, the node $(\xi,\alpha',\beta',\gamma')$ is terminal.

\end{lemma}

\begin{proof}
Suppose first that $(\alpha,\alpha',\beta',\gamma')$ is an immediate descendant of $(\alpha,\alpha,\beta,\gamma)$. We distinguish two cases:

\begin{itemize}
\item Suppose first that $\beta\not\in C_\gamma$. Set $\gamma':=\min(C_\gamma\setminus\beta)$. The situation is displayed in Figure \ref{alphaAlphaTerminal1}.

\begin{figure}[ht]
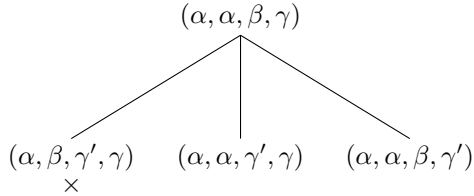

\ctikzfig{alphaAlphaTerminal}
\caption{The first step on the walk down from $(\alpha,\alpha,\beta,\gamma)$. The symbol $\times$ underneath a node indicates that said node is terminal.}
\label{alphaAlphaTerminal1}
\end{figure}
Note that $(\alpha,\beta,\gamma',\gamma)$ must be terminal, because $\beta\not\in C_\gamma$ and $C_{\gamma'\gamma}\setminus\beta \subset C_\gamma \cap [\gamma',\gamma)=\emptyset$. By the same reasoning, $(\xi,\beta,\gamma',\gamma)$ is terminal for every $\xi\le \alpha$.

\item Suppose now that $\beta\in C_{\gamma}$. If $\alpha\in C_{\beta\gamma}$ or $C_{\beta\gamma}\setminus\alpha=\emptyset$, then $(\alpha,\alpha,\beta,\gamma)$ is terminal and there's nothing to do. Suppose therefore that $\alpha\not\in C_{\beta\gamma}$ and $C_{\beta\gamma}\setminus\alpha\neq\emptyset$ and let $\alpha':=\min(C_{\beta\gamma}\setminus\alpha)$. The next step of the walk is displayed in Figure \ref{alphaAlphaTerminal2}.

\begin{figure}[ht]
\ctikzfig{alphaAlphaTerminal2}
\caption{}
\label{alphaAlphaTerminal2}
\end{figure}
Note that $(\alpha,\alpha',\beta,\gamma)$ must be terminal, because $C_{\alpha'\beta\gamma}\setminus\alpha\subset C_{\beta\gamma}\cap [\alpha,\alpha')=\emptyset$.

To see the ``Moreover", note that $\alpha\not\in C_{\beta\gamma}$ implies that there exists some $\nu<\alpha$ such that $[\nu,\alpha]\cap C_{\beta\gamma}=\emptyset$. Then, if $\xi \in [\nu,\alpha]$, we have $C_{\alpha'\beta\gamma}\setminus \xi\subset [\nu,\alpha]\cap C_{\beta\gamma}=\emptyset$.
\end{itemize}
Note that in either of the two previous cases, the other two nodes appearing in the next step of the walk are of the form $(\alpha,\alpha,\beta^*,\gamma^*)$, i.e. their first two coordinates are the same. Therefore, the general case follows by an easy induction.
\end{proof}

\begin{definition}\label{bad}
Let $\C$ be a $n$-$C$-sequence on an ordinal $\lambda$, with $n\ge 2$. Let 
$(\alpha,\vec{\gamma}) \in \lambda \otimes [\lambda]^{n}$. We will say that $x\in S_n(\alpha,\vec\gamma)$ is \emph{bad for $(\alpha,\vec\gamma)$} iff, letting $(\alpha,\vec\beta)=\Tr_n(\alpha,\vec\gamma)(x)$, we have that $\vec \beta\in I(\C)$ and $\alpha\in\acc(C_{\vec\beta})$. When the sequence $(\alpha,\vec\gamma)$ is clear from context, we will simply say that \emph{$x$ is bad}.

We will often blur the distinction between $x$ and its label $(\alpha,\vec\beta)$ and speak of $(\alpha,\vec\beta)$ as being \emph{bad}, the key realization being that it is only the label of $x$ that determines its badness, and not its position in $S_n(\alpha,\vec\gamma)$. We note that, if 
$(\alpha,\vec{\beta})$ is bad, then $\alpha$ must be a limit ordinal.
\end{definition}

\begin{lemma}\label{badImmediate}
Let $\C$ be a coherent $3$-$C$-sequence on an ordinal $\lambda$. Let $\alpha\le\beta<\gamma<\delta<\lambda$. Suppose that $x\in S_3(\alpha,\beta,\gamma,\delta)$ is bad. Then no $y\supsetneq x$ is bad.
\end{lemma}

\begin{proof}
Say $\Tr_3(\alpha,\beta,\gamma,\delta)(x)=(\alpha,\beta^*,\gamma^*,\delta^*)$, so that $\alpha\in \acc(C_{\beta^*\gamma^*\delta^*})$. The first step down from $(\alpha,\beta^*,\gamma^*,\delta^*)$ is displayed in Figure \ref{first_step_bad}.
\begin{figure}[ht]
\ctikzfig{first_step_bad}
\caption{}
\label{first_step_bad}
\end{figure}
The symbol $\times$ indicates that the node $(\alpha,\alpha,\gamma^*,\delta^*)$ is terminal, Indeed, $\alpha\in \acc(C_{\beta^*\gamma^*\delta^*})\subset \acc(C_{\gamma^*\delta^*})\subset C_{\gamma^*\delta^*}$ and so $C_{\alpha\gamma^*\delta^*}$ is defined and, by definition, $C_{\alpha\gamma^*\delta^*}\setminus\alpha=\emptyset$.

By Lemma \ref{alphaAlphaTerminal}, every non-terminal node underneath $(\alpha,\beta^*,\gamma^*,\delta^*)$ must therefore have its first two coordinates identical. But this precludes badness.
\end{proof}

Tuples of the form $(\alpha,\alpha,\dots)$, already appearing above in Lemma \ref{alphaAlphaTerminal}, will continue playing a pivotal role. As such, it will be helpful to introduce some terminology:

\begin{definition}\label{spectacled}
Let $\C$ be a $n$-$C$-sequence on an ordinal $\lambda$, with $n\ge 2$. Let 
$(\alpha,\vec{\gamma}) \in \lambda \otimes [\lambda]^n$. We will say that $x\in S_n(\alpha,\vec\gamma)$ is \emph{spectacled} iff $\Tr_n(\alpha,\vec\gamma)(x)=(\alpha,\alpha,\vec\beta)$ for some sequence $\vec\beta$.

Note that being spectacled depends on the parameters $(\C,\alpha,\vec\gamma)$, but these will be clear from context, so we will avoid mentioning them for reasons of brevity. Also, being spectacled is really a property of the label of $x$, so we will usually abuse notation and refer to the ``spectacled node" $(\alpha,\alpha,\vec\beta)$, or its signed version $(\pm,\alpha,\alpha,\vec\beta)$.
\end{definition}

\begin{lemma}\label{badNodes}
Let $\C$ be a coherent $3$-$C$-sequence on an ordinal $\lambda$. Let $\alpha\leq\beta<\gamma<\delta<\lambda$. Suppose that $x\in S_3(\alpha,\beta,\gamma,\delta)$ is bad, say $\Tr_3(\alpha,\beta,\gamma,\delta)(x)=(\alpha,\beta^*,\gamma^*,\delta^*)$. If $\xi<\alpha$, let $\eta_\xi:=\min(C_\alpha\setminus\xi)$. Then there exists $\xi^*<\alpha$ such that, for every $\xi\in [\xi^*,\alpha)$,
\begin{enumerate}[(i)]
\item If $y\in S_3(\alpha,\beta,\gamma,\delta)$ and $y\subset x$, then $y\in S_3(\xi,\beta,\gamma,\delta)$ and
\[
\Tr_3(\pm,\xi,\beta,\gamma,\delta)(y)=\sub_1^\xi(\Tr_3(\pm,\alpha,\beta,\gamma,\delta)(y)).
\]
\item If $y\in S_3(\alpha,\beta,\gamma,\delta)$, $x\subset y$ and $y$ is spectacled, then $y\in S_3(\xi,\beta,\gamma,\delta)$ and
\[
\Tr_3(\pm,\xi,\beta,\gamma,\delta)(y):=\sub_{1,2}^{\xi,\eta_\xi}(\Tr_3(\pm,\alpha,\beta,\gamma,\delta)(y)).
\]
\item If $y\in S_3(\alpha,\beta,\gamma,\delta)$ is not spectacled, $x\subset y$ and $y$ is terminal, then $y\in S_3(\xi,\beta,\gamma,\delta)$ and
\[
\Tr_3(\pm,\xi,\beta,\gamma,\delta)(y)=\sub_1^\xi(\Tr_3(\pm,\alpha,\beta,\gamma,\delta)(y)).
\]
\end{enumerate}
\end{lemma}

\begin{proof}
By assumption, $\alpha\in\acc(C_{\beta^*\gamma^*\delta^*})$, so by coherence $C_\alpha=C_{\beta^*\gamma^*\delta^*}\cap\alpha$.

By Lemma \ref{badImmediate}, no predecessor of $x$ is bad, hence
\[
\xi_0:=L_3(\alpha,\beta,\gamma,\delta)(x\restr (|x|-1))<\alpha.
\]
By Lemma \ref{7.7}, if $\xi\in [\xi_0,\alpha)$, then (i) holds.

The goal is to show that (ii) and (iii) hold for a tail of $\xi<\alpha$. In fact, if (ii) holds for a tail of $\xi<\alpha$, then so does (iii): indeed, fix $y\in S_3(\alpha,\beta,\gamma,\delta)$ as in (iii), say\footnote{In this paragraph, we omit any reference to the signs. The reader can check that these behave as in (iii), but writing them out would make the notation too cumbersome.} $\Tr_3(\alpha,\beta,\gamma,\delta)(y)=(\alpha,\beta',\gamma',\delta')$ with $\alpha<\beta'$. By Lemma \ref{alphaAlphaTerminal}, the immediate predecessor of $y$ in $S_3(\alpha,\beta,\gamma,\delta)$, call it $z$, was spectacled. Since $y$ is non-spectacled, we infer that  $\Tr_n(\alpha,\beta,\gamma,\delta)(z)=(\alpha,\alpha,\gamma',\delta')$. By (iii), $\Tr_n(\xi,\beta,\gamma,\delta)(z)=(\xi,\eta_\xi,\gamma',\delta')$. Since $z$ is not terminal, $\tau(\alpha,\alpha,\gamma',\delta')$ is a final segment of $ (\gamma',\delta')$. We will now assume that $\gamma'\in C_{\delta'}$ and leave the easier case $\gamma'\not\in C_{\delta'}$ to the reader. Note that $\tau(\alpha,\alpha,\gamma',\delta')=(\gamma',\delta')$. Fix $\bar\alpha<\alpha$ such that $[\bar\alpha,\alpha]\cap C_{\gamma'\delta'}=\emptyset$. If $\xi \in [\bar\alpha,\alpha)$, then $\eta_\xi\not\in C_{\gamma'\delta'}$, $\tau(\xi,\eta_\xi,\gamma',\delta')=\tau(\alpha,\alpha,\gamma',\delta')$ and
\[
\beta'=\min(C_{\tau(\alpha,\alpha,\gamma'\delta')}\setminus \alpha)=\min(C_{\tau(\xi,\eta_\xi,\gamma'\delta')}\setminus \eta_\xi).
\] 
It follows that, in stepping down from either $(\alpha,\alpha,\gamma',\delta')$ or $(\xi,\eta_\xi,\gamma',\delta')$, the same new ordinal is inserted. Moreover, when going from $z$ to $y$ in $\Tr_n(\alpha,\beta,\gamma,\delta)$, it is the second ordinal (namely, $\alpha$) that gets replaced with $\beta'$ (because $y$ is not spectacled), hence the same must be true in $\Tr_n(\xi,\beta,\gamma,\delta)$. Therefore
\[
\Tr_n(\xi,\beta,\gamma,\delta)(y)=(\xi,\eta_\xi,\beta',\gamma',\delta')^1=(\xi,\beta',\gamma',\delta')=\sub_1^\xi(\Tr_n(\alpha,\beta,\gamma,\delta)(y)),
\]
as desired.

We will be done if we can argue that (ii) holds. It clearly suffices to show that for every $y\in S_3(\alpha,\beta,\gamma,\delta)$ with $x\subset y$ and $y$ non-spectacled, there exists some $\xi_y<\alpha$ such that, for every $\xi \in [\xi_y,\alpha)$, $\Tr_3(\pm,\xi,\beta,\gamma,\delta)(y)=\sub_1^\xi(\Tr_3(\pm,\alpha,\beta,\gamma,\delta)(y))$. This will be proven by induction on $y$.

Our proof will be aided by a series of figures. In order to keep the clutter down, we have rendered these figures without indicating the signs of the nodes. Of course, the signs are very important, but the fact that they remain unchanged upon replacing $\alpha$ by $\xi$ will be immediate from the analysis, so we will not comment on it any further.

To start the induction, we look at the step down from $(\alpha,\beta^*,\gamma^*,\delta^*)$, noting that $\alpha=\min(C_{\beta^*\gamma^*\delta^*}\setminus\alpha)$, so that the situation is as in Figure \ref{first_step_bad}. Note that the node $(\alpha,\alpha,\gamma^*,\delta^*)$ is terminal, because $\alpha\in \acc(C_{\beta^*\gamma^*\delta^*})\subset C_{\gamma^*\delta^*}$ and therefore $C_{\alpha\gamma^*\delta^*}\setminus\alpha=\emptyset$.

Now let $\xi\in [\xi_0,\alpha)$. By coherence, $C_\alpha=C_{\beta^*\gamma^*\delta^*}\cap\alpha$, so $\min(C_{\beta^*\gamma^*\delta^*}\setminus \xi)=\eta_\xi$. Applying (i) and (ii), we have that the walk down from $(\xi,\beta,\gamma,\delta)$, up to the step down from $(\xi,\beta^*,\gamma^*,\delta^*)$, is as in Figure \ref{first_step_bad_xi}.

\begin{figure}[ht]
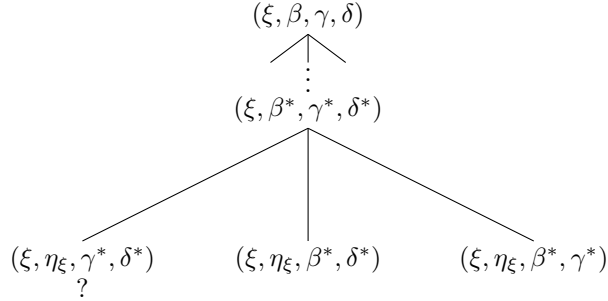

\ctikzfig{first_step_bad_xi}
\caption{The ? indicates that we do not know whether the node $(\xi,\eta_\xi,\gamma^*,\delta^*)$ is terminal. Compare with Figure \ref{first_step_bad}.}
\label{first_step_bad_xi}
\end{figure}

Comparing figures \ref{first_step_bad} and \ref{first_step_bad_xi}, we see that (ii) holds for 
the immediate successors of $x$.

For the inductive step, assume that (ii) holds for some spectacled $y\supseteq x$, say $\Tr_3(\alpha,\beta,\gamma,\delta)(y)=(\alpha,\alpha,\beta',\gamma')$. We will show that (ii) holds for the immediate successors of $y$. We must now embark on some case analysis. 

\begin{itemize}
\item Suppose that $\beta'\not\in C_{\gamma'}$ and set $\beta'':=\min(C_{\gamma'}\setminus \beta')$. The step down from $(\alpha,\alpha,\beta',\gamma')$ is then as in Figure \ref{second_step_new}.

\begin{figure}[ht]
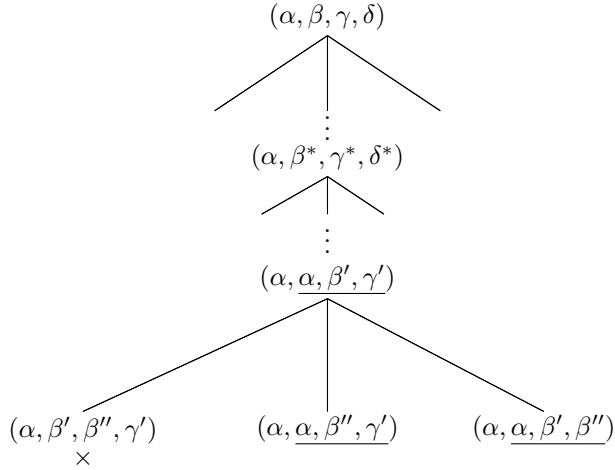

\ctikzfig{second_step_new}
\caption{Focusing on the underlined nodes, we see that our walk is ``simulating" a two-dimensional walk down from the node $(\alpha,\beta',\gamma')$, in the sense that we append an $\alpha$ to our labels and have a dummy (i.e. terminal) off-shoot. The fact that $(\alpha,\beta',\beta'',\gamma')$ is terminal follows from Lemma \ref{alphaAlphaTerminal}.}
\label{second_step_new}
\end{figure}
Now, if $\xi\in [\xi_0,\alpha)$, then the next step of the walk down from $(\xi,\eta_\xi,\beta',\gamma')$ is as in Figure \ref{second_step_new_xi}.

\begin{figure}[H]
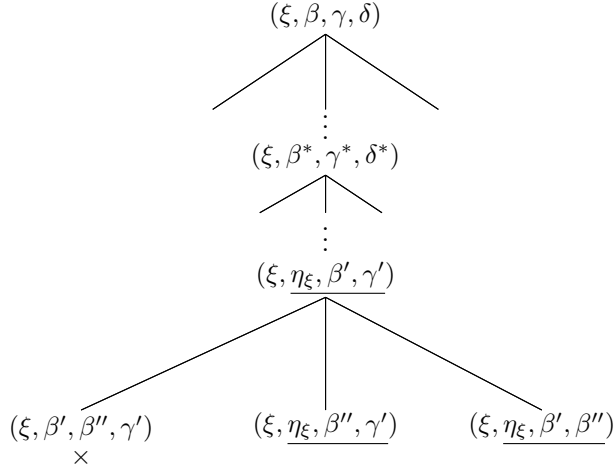

\ctikzfig{second_step_new_xi}
\caption{Again, the walk simulates a two-dimensional walk, this time down from the node $(\eta_\xi,\beta',\gamma')$. The node $(\xi,\beta',\beta'',\gamma')$ is terminal by Lemma \ref{alphaAlphaTerminal}.
Note also that the node $(\xi,\beta'',\beta',\gamma')$ showcases an instance of (iii).}
\label{second_step_new_xi}
\end{figure}

Comparing figures \ref{second_step_new} and \ref{second_step_new_xi}, we see that (ii) holds at the immediate successors of $y$.

\item Suppose now that $\beta'\in C_{\gamma'}$. If $\alpha\in C_{\beta'\gamma'}$, then $(\alpha,\alpha,\beta',\gamma')$ is terminal, and there is nothing to do. Suppose therefore that $\alpha\not\in C_{\beta'\gamma'}$. If $C_{\beta'\gamma'}\setminus\alpha=\emptyset$, then $y$ is again terminal, so we may assume that $C_{\beta'\gamma'}\setminus\alpha\neq\emptyset$ and therefore set $\alpha':=\min(C_{\beta'\gamma'}\setminus\alpha)$. The step down from $(\alpha,\alpha,\beta',\gamma')$ is now as in Figure \ref{second_step_new2}.

\begin{figure}[h]
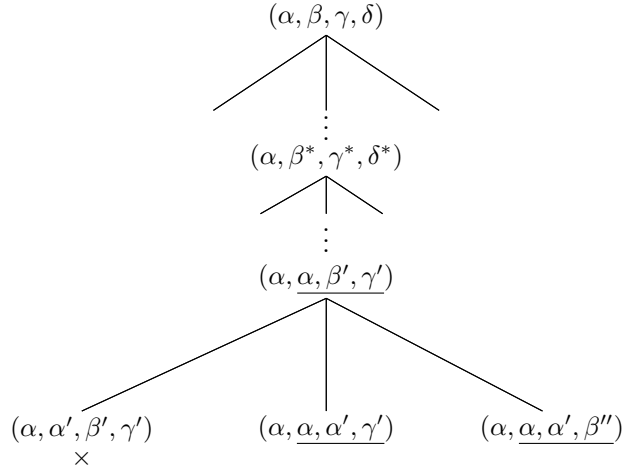

\ctikzfig{second_step_new2}
\caption{We are once again simulating a two-dimensional walk between the underlined nodes. The fact that $(\alpha,\alpha',\beta',\gamma')$ is terminal follows from Lemma \ref{alphaAlphaTerminal}.}
\label{second_step_new2}
\end{figure}

Let $\nu<\alpha$ be as in Lemma \ref{alphaAlphaTerminal}. Since $C_{\beta'\gamma'}$ is a club, there exists some $\xi_1\in [\max\{\nu,\xi_0\},\alpha)$ such that $[\xi_1,\alpha]\cap C_{\beta'\gamma'}=\emptyset$.  Let $\xi\in [\xi_1,\alpha)$. Then $\xi_1\le \xi \le \eta_\xi < \alpha$, so that $\eta_\xi\not\in C_{\beta'\gamma'}$ and $\min(C_{\beta'\gamma'}\setminus\eta_\xi)=\alpha'$. Therefore, the step down from $(\xi,\eta_\xi,\beta',\gamma')$ is as in Figure \ref{second_step_new_xi2}.

\begin{figure}[h]
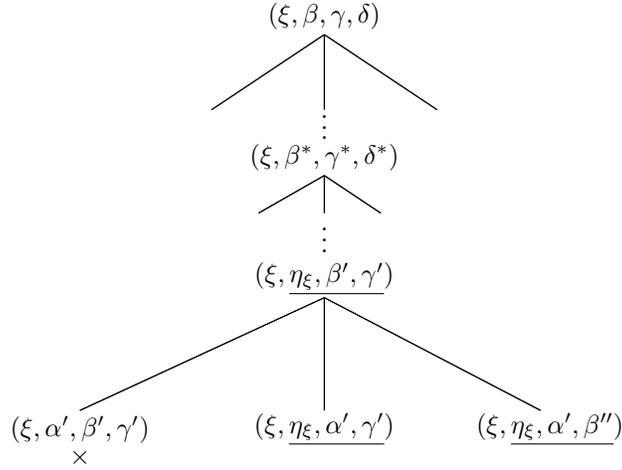

\ctikzfig{second_step_new_xi2}
\caption{Node $(\xi,\alpha',\beta',\gamma')$ is terminal by Lemma \ref{alphaAlphaTerminal}.}
\label{second_step_new_xi2}
\end{figure}
\end{itemize}
We are done proving (ii), and therefore the lemma.\end{proof}

\begin{lemma}\label{end_ext}
Let $\C$ be a coherent $3$-$C$-sequence on an ordinal $\lambda$. Let $\alpha\leq\beta<\gamma<\delta<\lambda$, with $\alpha$ a limit ordinal. Then there exists $\xi^*<\alpha$ such that, for every $\xi\in [\xi^*,\alpha]$, $S_3(\pm,\xi,\beta,\gamma,\delta)$ is an end-extension of $S_3(\pm,\alpha,\beta,\gamma,\delta)$ and, for all $x \in S_3(\alpha,\beta,\gamma,\delta)$, 
$\sigma_{\xi\beta\gamma\delta}(x) = \sigma_{\alpha\beta\gamma\delta}(x)$.
\end{lemma}

\begin{proof}
For each bad node $x\in S_3(\alpha,\beta,\gamma,\delta)$, let $\xi^*_x<\alpha$ be a witness to Lemma \ref{badNodes}. Then let
\[
\xi_0:=\max\{\xi_x^*:x\in S_3(\alpha,\beta,\gamma,\delta)\text{ is bad}\}.
\]
Also, set
\[
\xi_1:=\max\{L_n(\alpha,\beta,\gamma,\delta)(x):x\in S_3(\alpha,\beta,\gamma,\delta) \wedge L_n(\alpha,\beta,\gamma,\delta)(x)<\alpha\}.
\]
Not let $\xi^*:=\max\{\xi_0,\xi_1\}$, so that $\xi^*<\alpha$. The conclusion now follows by combining Lemmas \ref{7.7} and \ref{badNodes}.
\end{proof}

The next lemma is little more than a summary of things we've already proven, but will nevertheless be quite convenient in the sequel.

\begin{lemma}\label{easyNodes}
Let $\C$ be a coherent $3$-$C$-sequence on an ordinal $\lambda$. Let $\alpha\leq\beta<\gamma<\delta<\lambda$, with $\alpha$ a limit ordinal. Let $x\in S_3(\alpha,\beta,\gamma,\delta)$ be terminal with $\Tr_3(\pm,\alpha,\beta,\gamma,\delta)(x)=((-1)^k,\alpha,\beta',\gamma',\delta')$ and $\alpha<\beta'$ (i.e., $x$ is not spectacled). Then there exists $\xi^*<\alpha$ such that, for all $\xi\in [\xi^*,\alpha)$,
\[
\Tr_3(\pm,\xi,\beta,\gamma,\delta)(x)=((-1)^k,\xi,\beta',\gamma',\delta').
\]
\end{lemma}

\begin{proof}
If $x$ descends from a bad node in $\Tr_3(\pm,\alpha,\beta,\gamma,\delta)$, then the conclusion follows from Lemma \ref{badNodes}(ii). If $x$ does not descend from any bad nodes, then the conclusion follows from Lemma \ref{7.7}.
\end{proof}

Recall that, if $\gamma$ is an ordinal, $f$ is a function with domain $\gamma$, and 
$\alpha \leq \gamma$ is a limit ordinal,  we say that $f$ is \emph{locally semi-constant at $\alpha$} 
iff there exists $\xi_0<\alpha$ such that $f\restr [\xi_0,\alpha)$ is constant. We say 
that $f$ is locally semi-constant if it is locally semi-constant at every limit ordinal
$\alpha \leq \gamma$.

\begin{theorem}\label{main3}
Let $\C$ be a coherent $3$-$C$-sequence on an infinite ordinal $\lambda$. Then the family
\[
\Phi(\resh_3):=\left\langle\resh_3(\cdot,\vec\gamma):\gamma_0 \ra \bigoplus_{[\lambda]^{2}} \bb{Z} \ \middle| \ \vec\gamma\in[\lambda]^3\right\rangle
\]
is $3$-coherent modulo locally semi-constant functions, i.e., for every $\vec\gamma\in[\lambda]^{4}$,
\begin{equation}\label{sumRho}
\sum_{i=0}^3(-1)^i\resh_3(\cdot,\vec\gamma^i)
\end{equation}
is locally semi-constant at every limit ordinal $\alpha\le\gamma_0$.
\end{theorem}

\begin{notation}\label{xiImage}
Suppose that $x\in S(\alpha,\vec\gamma)$ and $\Tr_n(\pm,\alpha,\vec\gamma)(x)=((-1)^m,\alpha,\vec\beta)$. We will refer to $\Tr_n(\pm,\xi,\vec\gamma)(x)$ as the \emph{$\xi$-image of $((-1)^m,\alpha,\vec\beta)$} whenever this makes sense. In the proof that follows, we will have multiple walks taking place simultaneously (because of \eqref{sumRho}), so the $\xi$-image terminology will allow us to remain agnostic as to which of these walks the node we're looking at is descended from. This ambiguity will not compromise the argument, but it will reduce the clutter of symbols.
\end{notation}

\begin{proof}[Proof of Theorem \ref{main3}]

Let us fix $\vec\gamma\in [\lambda]^4$ and $\alpha\leq\gamma_0$ a limit ordinal. Let $f:\gamma_0\to\bigoplus_{[\lambda]^2}\Z$ be the function defined by setting
\begin{equation}\label{f}
f(\xi):=\sum_{i\le 3}\resh(\xi,\vec\gamma^i)
\end{equation}
for all $\xi < \gamma_0$. We want to argue that $f$ is locally semi-constant at $\alpha$.

Let $\xi^*<\alpha$ simultaneously satisfy the conclusion of Lemma \ref{end_ext} for the tuple 
$(\alpha,\vec{\gamma}^i)$ for all $i \leq 3$. If $i\le 3$ and $\xi\in [\xi^*,\alpha)$, then
\begin{equation}\label{rhosplit}
\resh_3(\xi,\vec\gamma^i)=\sum_{x\in S_3(\xi,\vec\gamma^i)}\sigma_{\xi\vec\gamma^i}(x)=\sum_{x\in S_3(\alpha,\vec\gamma^i)}\sigma_{\alpha\vec\gamma^i}(x)+\sum_{t\in \bS_3(\alpha,\vec\gamma^i)}\sum_{x\supsetneq t}\sigma_{\xi\vec\gamma^i}(x).
\end{equation}

Multiplying \eqref{rhosplit} by $(-1)^i$, summing over $i\le 3$ and applying Lemma \ref{7.12}, we obtain
\begin{multline}\label{rhosum}
\sum_{i\le 3}(-1)^i\resh_3(\xi,\vec\gamma^i)= \\ \left[\sum_{i\le 3}(-1)^i\resh_3(\alpha,\vec\gamma^i)\right]+\sum_{t\in Z}\underbrace{\left[\sum_{x\supsetneq t}(-1)^{i(t)}\sigma_{\xi\vec\gamma^{i(t)}}(x)+\sum_{x\supset t_-}(-1)^{i(t_-)}\sigma_{\xi\vec\gamma^{i(t_-)}}(x)\right]}_{g_t(\xi)}
\end{multline}
for some finite set $Z$. Now, the quantity between the first pair of square brackets does not depend on $\xi$, hence can be ignored. We aim to show that, for each $t\in Z$, there is some $\xi_t<\alpha$ so that $g_t\restr [\xi_t,\alpha)$ is constant. Then, if $\hat{\xi}:=\max\{\xi_t:t\in Z\}$, we will have 
shown that $f \restriction (\hat{\xi},\alpha)$ is constant, completing the proof.

Suppose now that we have fixed $t\in Z$. Going forward, we will almost exclusively not be working with $t$ and $t_-$, but rather with their labels. Let us suppose, then, that 
$\Tr_3((-1)^{i(t)},\alpha,\vec{\gamma}^{i(t)})(t) = (+,\alpha,\beta',\gamma',\delta')$ and 
$\Tr_3((-1)^{i(t_-)},\alpha,\vec{\gamma}^{i(t_-)})(t_-) = (-,\alpha,\beta',\gamma',\delta')$ (the case 
in which the signs are reversed is symmetric). Recall that both are terminal nodes in their 
respective walks. The non-spectacled case (i.e. $\alpha<\beta'$) is easy and can be dispatched quickly. Indeed, by Lemma \ref{easyNodes}, there exists $\xi_t<\alpha$ such that, for all $\xi \in 
[\xi_t,\alpha)$, we have $\Tr_3((-1)^{i(t)},\xi,\vec{\gamma}^{i(t)})(t) = (+,\xi,\beta',\gamma',
\delta')$ and $\Tr_3((-1)^{i(t_-)},\xi,\vec{\gamma}^{i(t_-)})(t) = (-,\xi,\beta',\gamma',\delta')$. 
It follows that, for all $x \supseteq t$, we have $x \in S_3(\xi,\vec{\gamma}^{i(t)})$ if and only 
if $x \in S_3(\xi,\vec{\gamma}^{i(t_-)})$ and, for all such $x$ in $S_3(\xi,\vec{\gamma}^{i(t)})$, 
we have
\[
  \Tr_3((-1)^{i(t)},\xi,\vec{\gamma}^{i(t)})(x) = - \Tr_3((-1)^{i(t_-)},\xi,\vec{\gamma}^{i(t_-)})(x).
\]
The relevant expressions in $\resh_3$ then cancel out in the computation of $g_t(\xi)$, that is to say, $g_t(\xi)=0$ for all $\xi\in [\xi_t,\alpha)$.

We are therefore left with studying nodes of the form $(\alpha,\alpha,\beta',\gamma')$ which are terminal in $S_3(\alpha,\vec\gamma^i)$ for some $i\le 3$. Our strategy is now as follows:

\begin{enumerate}[(a)]
\item Each terminal $(+,\alpha,\alpha,\beta',\gamma')$ is of one of two types: \textit{good} and \textit{bad}, with the bad case falling under the domain of Lemma \ref{badNodes}.
\item We then consider three cases: both nodes good, both bad, one good and one bad.
\item Argue that, in each case, the function $g_t$ is locally semi-constant at $\alpha$.
\end{enumerate}

By relabeling certain ordinals, suppose that $\Tr_3((-1)^{i(t)},\alpha,\vec{\gamma}^{i(t)})(t) = (+,\alpha,\alpha,\beta',\gamma')$ and 
$\Tr_3((-1)^{i(t_-)},\alpha,\vec{\gamma}^{i(t_-)})(t_-) = (-,\alpha,\alpha,\beta',\gamma')$. We will now case on whether they descend from bad nodes (simultaneously or not). If neither of $(+,\alpha,\alpha,\beta',\gamma')$ and $(-,\alpha,\alpha,\beta',\gamma')$ descend from a bad node, then Lemma \ref{7.7} implies that, for all large enough $\xi<\alpha$, the $\xi$-images of $(+,\alpha,\alpha,\beta',\gamma')$ and $(-,\alpha,\alpha,\beta',\gamma')$ are $(+,\xi,\alpha,\beta',\gamma')$ and $(-,\xi,\alpha,\beta',\gamma')$, respectively. Therefore, the cones\footnote{i.e. where the cone at a node is the set of descendants of that node in the walk being performed.} beneath these nodes must cancel out in the computation of $f$. If instead both $(+,\alpha,\alpha,\beta',\gamma')$ and $(-,\alpha,\alpha,\beta',\gamma')$ descend from bad nodes, then Lemma \ref{badNodes} implies that, for all large enough $\xi<\alpha$, the $\xi$-images of $(+,\alpha,\alpha,\beta',\gamma')$ and $(-,\alpha,\alpha,\beta',\gamma')$ are $(+,\xi,\eta_\xi,\beta',\gamma')$ and $(-,\xi,\eta_\xi,\beta',\gamma')$ respectively, where $\eta_\xi=\min(C_\alpha\setminus\xi)$. Once again, the cones underneath these nodes cancel out.

Suppose now that $(+,\alpha,\alpha,\beta',\gamma')$ descends from a bad node $(\pm,\alpha,\beta^*,\gamma^*,\delta^*)$, while $(-,\alpha,\alpha,\beta',\gamma')$ descends from no bad node (the reverse case is symmetric). For all large enough $\xi<\alpha$, the situation is summarized by Figures \ref{alpha_xi_sides} and \ref{goodFig}; note in particular that the $\xi$-images of $(+,\alpha,\alpha,\beta',\gamma')$ and $(-,\alpha,\alpha,\beta',\gamma')$ are $(+,\xi,\eta_\xi,\beta',\gamma')$ and $(-,\xi,\alpha,\beta',\gamma')$, respectively.

\begin{figure}[ht]
\begin{tikzpicture}
	\begin{pgfonlayer}{nodelayer}
		\node [style=none] (0) at (0.5, 3.25) {$(\xi,\beta^*,\gamma^*,\delta^*)$};
		\node [style=none] (4) at (-3.5, 1.5) {};
		\node [style=none] (5) at (0.5, 0.5) {};
		\node [style=none] (6) at (5, 0) {};
		\node [style=none] (7) at (0.5, 2.75) {};
		\node [style=none] (8) at (-3.5, 1) {$(\xi,\eta_\xi,\gamma^*,\delta^*)$};
		\node [style=none] (12) at (4.25, -3) {};
		\node [style=none] (13) at (5.25, -3) {};
		\node [style=none] (14) at (6.75, -3.5) {};
		\node [style=none] (15) at (5.25, -2) {};
		\node [style=none] (18) at (5, -0.5) {$(\xi,\eta_\xi,\beta^*,\gamma^*)$};
		\node [style=none] (19) at (0.5, -0.25) {$(\xi,\eta_\xi,\beta^*,\delta^*)$};
		\node [style=none] (21) at (5.25, -1.25) {$\vdots$};
		\node [style=none] (22) at (-3.5, 0) {?};
		\node [style=none] (23) at (6.75, -4.25) {$(\xi,\eta_\xi,\beta',\gamma')$};
		\node [style=none] (24) at (-13.25, 3.75) {$(\alpha,\beta^*,\gamma^*,\delta^*)$};
		\node [style=none] (25) at (-17, 1.25) {$(\alpha,\alpha,\gamma^*,\delta^*)$};
		\node [style=none] (26) at (-13.5, 0.25) {$(\alpha,\alpha,\beta^*,\delta^*)$};
		\node [style=none] (27) at (-9.5, -0.75) {$(\alpha,\alpha,\beta^*,\gamma^*)$};
		\node [style=none] (28) at (-17, 1.75) {};
		\node [style=none] (29) at (-13.5, 0.75) {};
		\node [style=none] (30) at (-9.5, -0.25) {};
		\node [style=none] (31) at (-13.5, 3) {};
		\node [style=none] (32) at (-17.25, 0.25) {$\times$};
		\node [style=none] (35) at (-7.75, -4) {$(\alpha,\alpha,\beta',\gamma')$};
		\node [style=none] (36) at (-10.5, -3) {};
		\node [style=none] (37) at (-9.5, -3) {};
		\node [style=none] (38) at (-8.25, -3.5) {};
		\node [style=none] (39) at (-9.5, -2) {};
		\node [style=none] (41) at (-9.5, -1.25) {$\vdots$};
		\node [style=none] (42) at (-7.75, -4.75) {$\times$};
		\node [style=none] (43) at (6.75, -5) {?};
		\node [style=none] (44) at (-19, 3) {};
		\node [style=none] (45) at (-7.5, 3) {};
		\node [style=none] (46) at (-7.5, -3.5) {};
		\node [style=none] (47) at (-19, 0.75) {};
		\node [style=none] (48) at (-5.5, 2.75) {};
		\node [style=none] (49) at (-5.5, 0.5) {};
		\node [style=none] (50) at (8, -3.5) {};
		\node [style=none] (51) at (8, 2.75) {};
		\node [style=none] (52) at (-13.5, -0.25) {};
		\node [style=none] (53) at (-14, -1) {};
		\node [style=none] (55) at (-14, -1.25) {$\vdots$};
		\node [style=none] (56) at (-14, -2.5) {$(\alpha,\alpha',\beta'',\delta'')$};
		\node [style=none] (57) at (-14, -3) {$\times$};
		\node [style=none] (58) at (0, -1.25) {};
		\node [style=none] (59) at (0, -1.5) {$\vdots$};
		\node [style=none] (60) at (0, -2.5) {$(\xi,\alpha',\beta'',\delta'')$};
		\node [style=none] (61) at (0, -3) {$\times$};
		\node [style=none] (62) at (0.5, -0.75) {};
		\node [style=none] (63) at (-16.25, 0.75) {};
		\node [style=none] (64) at (-16.25, -3.5) {};
		\node [style=none] (65) at (-2.25, 0.5) {};
		\node [style=none] (66) at (-2.25, -3.5) {};
	\end{pgfonlayer}
	\begin{pgfonlayer}{edgelayer}
		\draw (7.center) to (4.center);
		\draw (7.center) to (6.center);
		\draw (5.center) to (7.center);
		\draw (15.center) to (12.center);
		\draw (15.center) to (14.center);
		\draw (13.center) to (15.center);
		\draw (31.center) to (28.center);
		\draw (31.center) to (30.center);
		\draw (29.center) to (31.center);
		\draw (39.center) to (36.center);
		\draw (39.center) to (38.center);
		\draw (37.center) to (39.center);
		\draw [style=punteada] (46.center) to (45.center);
		\draw [style=punteada] (45.center) to (44.center);
		\draw [style=punteada] (51.center) to (50.center);
		\draw [style=punteada] (49.center) to (48.center);
		\draw [style=punteada] (48.center) to (51.center);
		\draw (52.center) to (53.center);
		\draw (62.center) to (58.center);
		\draw [style=punteada] (44.center) to (47.center);
		\draw [style=punteada] (47.center) to (63.center);
		\draw [style=punteada] (63.center) to (64.center);
		\draw [style=punteada] (64.center) to (46.center);
		\draw [style=punteada] (49.center) to (65.center);
		\draw [style=punteada] (65.center) to (66.center);
		\draw [style=punteada] (66.center) to (50.center);
	\end{pgfonlayer}
\end{tikzpicture}
\caption{The bad case. By Lemma \ref{badNodes}, the parts of $S_3(\alpha,\beta^*,\gamma^*,\delta^*)$ and $S_3(\xi,\beta^*,\gamma^*,\delta^*)$ that lie within the dashed lines are the same.}
\label{alpha_xi_sides}
\end{figure}
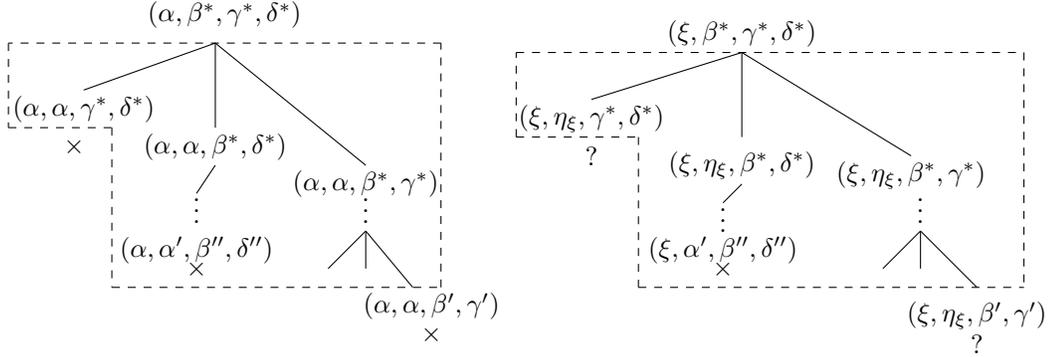

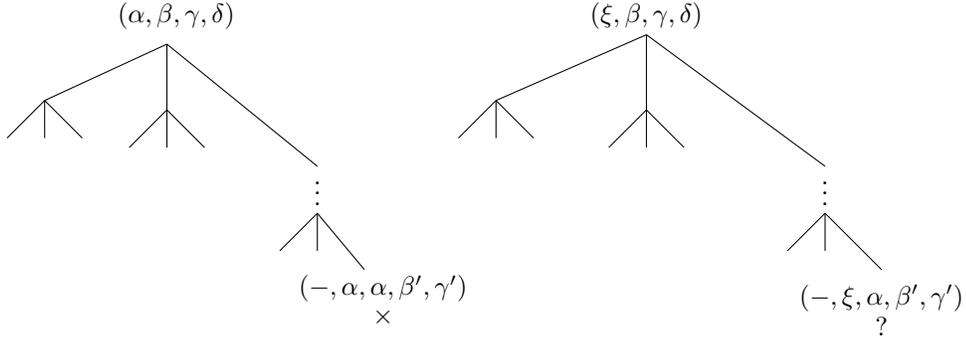
\begin{figure}[ht]
\begin{tikzpicture}
	\begin{pgfonlayer}{nodelayer}
		\node [style=none] (0) at (7.25, 4.25) {$(\xi,\beta,\gamma,\delta)$};
		\node [style=none] (3) at (12, 0.25) {};
		\node [style=none] (4) at (7.25, 3.75) {};
		\node [style=none] (6) at (11, -2) {};
		\node [style=none] (7) at (12, -2) {};
		\node [style=none] (8) at (13.5, -2.5) {};
		\node [style=none] (9) at (12, -1) {};
		\node [style=none] (12) at (12, -0.25) {$\vdots$};
		\node [style=none] (14) at (13.5, -3.25) {$(-,\xi,\alpha,\beta',\gamma')$};
		\node [style=none] (15) at (-5.25, 4.25) {$(\alpha,\beta,\gamma,\delta)$};
		\node [style=none] (19) at (-8.75, 2) {};
		\node [style=none] (20) at (-5.5, 1.75) {};
		\node [style=none] (21) at (-1.5, 0.25) {};
		\node [style=none] (22) at (-5.5, 3.5) {};
		\node [style=none] (24) at (0.25, -3) {$(-,\alpha,\alpha,\beta',\gamma')$};
		\node [style=none] (25) at (-2.5, -2) {};
		\node [style=none] (26) at (-1.5, -2) {};
		\node [style=none] (27) at (-0.25, -2.5) {};
		\node [style=none] (28) at (-1.5, -1) {};
		\node [style=none] (30) at (0.25, -3.75) {$\times$};
		\node [style=none] (31) at (13.5, -4) {?};
		\node [style=none] (32) at (-10.75, 3.5) {};
		\node [style=none] (33) at (0.5, 3.5) {};
		\node [style=none] (34) at (0.5, -2.5) {};
		\node [style=none] (35) at (-10.75, -3) {};
		\node [style=none] (36) at (2.75, 3.25) {};
		\node [style=none] (37) at (2.75, -3) {};
		\node [style=none] (38) at (14.75, -2.5) {};
		\node [style=none] (39) at (14.75, 3.75) {};
		\node [style=none] (43) at (-1.5, -0.25) {$\vdots$};
		\node [style=none] (44) at (-6.5, 0.75) {};
		\node [style=none] (45) at (-5.5, 0.75) {};
		\node [style=none] (46) at (-4.5, 0.75) {};
		\node [style=none] (47) at (-9.75, 1) {};
		\node [style=none] (48) at (-8.75, 1) {};
		\node [style=none] (49) at (-7.75, 1) {};
		\node [style=none] (50) at (3.25, 2) {};
		\node [style=none] (51) at (7.25, 1.75) {};
		\node [style=none] (52) at (6.25, 0.75) {};
		\node [style=none] (53) at (7.25, 0.75) {};
		\node [style=none] (54) at (8.25, 0.75) {};
		\node [style=none] (55) at (2.25, 1) {};
		\node [style=none] (56) at (3.25, 1) {};
		\node [style=none] (57) at (4.25, 1) {};
	\end{pgfonlayer}
	\begin{pgfonlayer}{edgelayer}
		\draw (4.center) to (3.center);
		\draw (9.center) to (6.center);
		\draw (9.center) to (8.center);
		\draw (7.center) to (9.center);
		\draw (22.center) to (19.center);
		\draw (22.center) to (21.center);
		\draw (20.center) to (22.center);
		\draw (28.center) to (25.center);
		\draw (28.center) to (27.center);
		\draw (26.center) to (28.center);
		\draw (44.center) to (20.center);
		\draw (20.center) to (45.center);
		\draw (20.center) to (46.center);
		\draw (19.center) to (47.center);
		\draw (19.center) to (48.center);
		\draw (19.center) to (49.center);
		\draw (52.center) to (51.center);
		\draw (51.center) to (53.center);
		\draw (51.center) to (54.center);
		\draw (50.center) to (55.center);
		\draw (50.center) to (56.center);
		\draw (50.center) to (57.center);
		\draw (51.center) to (4.center);
		\draw (4.center) to (50.center);
	\end{pgfonlayer}
\end{tikzpicture}
\caption{The good case. We have assumed for definiteness that the starting node is $(\alpha,\beta,\gamma,\delta)$. Here, the left and the right trees are the same, by Lemma \ref{7.7}.}%
\label{goodFig}
\end{figure}

Our aim is to compare what happens underneath the nodes $(+,\xi,\eta_\xi,\beta',\gamma')$ and $(-,\xi,\alpha,\beta',\gamma')$, and argue that the contributions to the value of $g_t(\xi)$ do not depend on $\xi$ for large enough $\xi<\alpha$. We will argue this in two, combined, ways: first, by identifying pairs of nodes with opposite-signed labels, which therefore cancel out in the computation of $g_t(\xi)$, and, second, by arguing that the signed trees under the remaining nodes, and the values of the $\sigma_{\xi\vec\gamma^i}$, for $i\in \{i(t),i(t_-)\}$ on said trees, do not depend on the value of $\xi$ (for all large enough $\xi<\alpha$). This will imply that the values they contribute to $g_t(\xi)$ is constant for all large enough $\xi<\alpha$.

Our analysis will require some case analysis. To guide us, let us first explore what it means for $(+,\alpha,\alpha,\beta',\gamma')$ to be terminal:

\begin{equation*}
(+,\alpha,\alpha,\beta',\gamma') \text{ is terminal } \iff \begin{cases}
\hspace{1cm}\beta'\not\in C_{\gamma'}\wedge C_{\gamma'}\setminus \beta'=\emptyset\\
\hspace{2.5cm} \text{ or }\\
\beta'\in C_{\gamma'}\wedge \alpha\not\in C_{\beta'\gamma'}\wedge C_{\beta'\gamma'}\setminus \alpha=\emptyset\\
\hspace{2.5cm}\text{ or }\\
\hspace{1cm}\beta'\in C_{\gamma'}\wedge \alpha\in C_{\beta'\gamma'}.
\end{cases}
\end{equation*}
However, the first alternative is impossible, for if $\beta'<\gamma'$, then $C_{\gamma'}\setminus\beta'\neq\emptyset$ because $C_{\gamma'}$ is cofinal in $\gamma'$. We are thus left with two possibilities.
\begin{itemize}
\item Suppose first that $\beta'\in C_{\gamma'}$ and $\alpha\in C_{\beta'\gamma'}$. This will require a further case analysis, depending on whether $\alpha\in\acc(C_{\beta'\gamma'})$
\begin{itemize}
\item Suppose that $\alpha\in\acc(C_{\beta'\gamma'})$. By coherence, $C_\alpha=C_{\alpha\beta'\gamma'}$. %
The step down from the node $(-,\xi,\alpha,\beta',\gamma')$ (consult Figure \ref{goodFig}) is therefore as displayed in Figure \ref{good_step}.

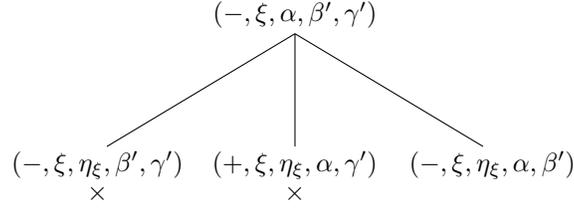
\begin{figure}[ht]
\begin{tikzpicture}
	\begin{pgfonlayer}{nodelayer}
		\node [style=none] (13) at (0, 0.5) {$(-,\xi,\alpha,\beta',\gamma')$};
		\node [style=none] (21) at (-5, -3) {};
		\node [style=none] (22) at (0, -3) {};
		\node [style=none] (23) at (5, -3) {};
		\node [style=none] (24) at (0, 0) {};
		\node [style=none] (25) at (-5.25, -3.5) {$(-,\xi,\eta_\xi,\beta',\gamma')$};
		\node [style=none] (26) at (0, -3.5) {$(+,\xi,\eta_\xi,\alpha,\gamma')$};
		\node [style=none] (27) at (5.25, -3.5) {$(-,\xi,\eta_\xi,\alpha,\beta')$};
		\node [style=none] (28) at (-5.25, -4.25) {$\times$};
		\node [style=none] (29) at (0, -4.25) {$\times$};
	\end{pgfonlayer}
	\begin{pgfonlayer}{edgelayer}
		\draw (24.center) to (21.center);
		\draw (22.center) to (24.center);
		\draw (24.center) to (23.center);
	\end{pgfonlayer}
\end{tikzpicture}
\caption{The good case. One can argue that the node $(-,\xi,\eta_\xi,\beta',\gamma')$ is terminal, but this isn't needed, see below.}
\label{good_step}
\end{figure}

Comparing Figure \ref{alpha_xi_sides} and Figure \ref{good_step}, we see that the nodes $(+,\xi,\eta_\xi,\beta',\gamma')$ and $(-,\xi,\eta_\xi,\beta',\gamma')$ cancel out, and so their net contribution to $g_t(\xi)$ is $0$. Also, as $\alpha\in \acc(C_{\beta'\gamma'})\subset \acc(C_{\gamma'})$, by coherence $C_{\alpha\gamma'}=C_\alpha$ and so $C_{\eta_\xi\alpha\gamma'}\setminus\xi=\emptyset$, justifying that $(+,\xi,\eta_\xi,\alpha,\gamma'$) is terminal. The key point is that this is independent of the value of $\xi$, so the contribution of this node to $g_t(\xi)$ is constant for large enough $\xi<\alpha$. To deal with the node $(-\xi,\eta_\xi,\alpha,\beta')$, we must split into cases once more. First, suppose that $\alpha\in\acc(C_{\beta'})$. Then $C_{\alpha\beta'}=C_\alpha$ by coherence and so $C_{\eta_\xi\alpha\beta'}\setminus\xi=\emptyset$ by the choice of $\eta_\xi$, hence $(-\xi,\eta_\xi,\alpha,\beta')$ is terminal. If instead $\alpha\not\in\acc(C_{\beta'})$, let $\bar\alpha:=\max(\alpha\cap C_{\beta'})$. If $\alpha\in C_{\beta'}$, then whenever $\bar\alpha<\xi<\alpha$ we have $C_{\alpha\beta'}\setminus \eta_\xi=\emptyset$, hence $(-\xi,\eta_\xi,\alpha,\beta')$ is terminal. If instead $\alpha\not\in C_{\beta'}$, then $(-\xi,\eta_\xi,\alpha,\beta')$ will simulate a one-dimensional walk from $\beta'$ down to $\alpha$ as follows. First, let $\alpha=\beta_m<\beta_{m-1}<\dots<\beta_0=\beta'$ be the one-dimensional walk from $\beta'$ down to $\alpha$ using the $C$-sequence $\seq{C_\nu:\nu<\kappa}$. For $i<n$, we have that $\alpha\not\in C_{\beta_i}$, hence $\sup(\alpha\cap C_{\beta_i})<\alpha$. Set $\epsilon:=\max_{i<m-1}\sup(\alpha\cap C_{\beta_i})$, so that $\epsilon<\alpha$. The situation can be seen in Figures \ref{1dwalk} and \ref{1dstep}.

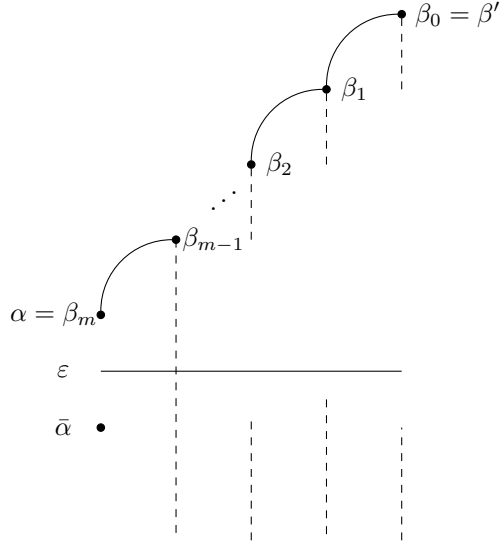
\begin{figure}[ht]
\begin{tikzpicture}
	\begin{pgfonlayer}{nodelayer}
		\node [style=label] (0) at (8, 8) {};
		\node [style=label] (5) at (6, 6) {};
		\node [style=label] (7) at (4, 4) {};
		\node [style=label] (9) at (2, 2) {};
		\node [style=label] (10) at (0, 0) {};
		\node [style=none] (11) at (9.5, 8) {$\beta_0=\beta'$};
		\node [style=none] (12) at (6.75, 6) {$\beta_1$};
		\node [style=none] (13) at (3, 2) {$\beta_{m-1}$};
		\node [style=none] (14) at (-1.25, 0) {$\alpha=\beta_m$};
		\node [style=none] (15) at (4.75, 4) {$\beta_2$};
		\node [style=label] (16) at (0, -3) {};
		\node [style=none] (17) at (-1, -3) {$\bar\alpha$};
		\node [style=none] (18) at (8, 6) {};
		\node [style=none] (19) at (8, -3) {};
		\node [style=none] (20) at (8, -6) {};
		\node [style=none] (21) at (6, 4) {};
		\node [style=none] (22) at (0, -1.5) {};
		\node [style=none] (23) at (-1, -1.5) {$\epsilon$};
		\node [style=none] (24) at (8, -1.5) {};
		\node [style=none] (25) at (6, -2.25) {};
		\node [style=none] (26) at (6, -6) {};
		\node [style=none] (27) at (4, -2.75) {};
		\node [style=none] (28) at (4, -6) {};
		\node [style=none] (29) at (2, 2) {};
		\node [style=none] (30) at (2, -6) {};
		\node [style=none] (31) at (3.25, 3.25) {$\iddots$};
		\node [style=none] (32) at (4, 2) {};
	\end{pgfonlayer}
	\begin{pgfonlayer}{edgelayer}
		\draw [bend right=45] (9) to (10);
		\draw [bend right=45] (0) to (5);
		\draw [bend right=45] (5) to (7);
		\draw [style=punteada, in=270, out=90] (18.center) to (0);
		\draw [style=punteada] (20.center) to (19.center);
		\draw [style=punteada] (21.center) to (5);
		\draw (24.center) to (22.center);
		\draw [style=punteada] (25.center) to (26.center);
		\draw [style=punteada] (28.center) to (27.center);
		\draw [style=punteada] (29.center) to (30.center);
		\draw [style=punteada] (32.center) to (7);
	\end{pgfonlayer}
\end{tikzpicture}
\caption{The one dimensional walk between $\beta'$ and $\alpha$ and all the ordinals involved. The dashed lines represent the relevant clubs, and the absence thereof indicates that the club in question is disjoint from the interval at hand.}
\label{1dwalk}
\end{figure}

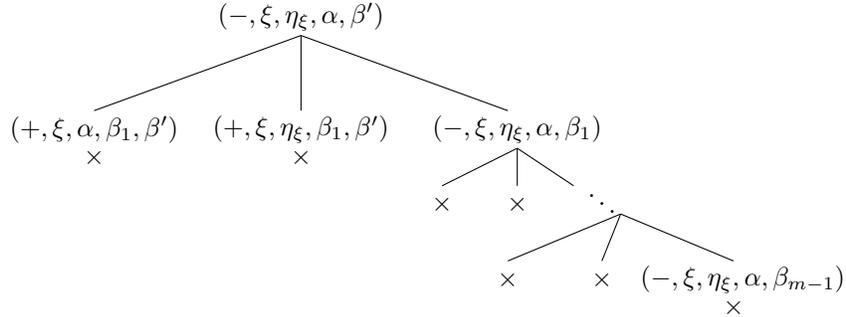
\begin{figure}[ht]
\begin{tikzpicture}
	\begin{pgfonlayer}{nodelayer}
		\node [style=none] (0) at (0, 2) {};
		\node [style=none] (1) at (0, 2.5) {$(-,\xi,\eta_\xi,\alpha,\beta')$};
		\node [style=none] (2) at (-5.5, 0) {};
		\node [style=none] (3) at (0, 0) {};
		\node [style=none] (4) at (5.5, 0) {};
		\node [style=none] (5) at (-5.5, -0.5) {$(+,\xi,\alpha,\beta_1,\beta')$};
		\node [style=none] (6) at (0, -0.5) {$(+,\xi,\eta_\xi,\beta_1,\beta')$};
		\node [style=none] (7) at (5.75, -0.5) {$(-,\xi,\eta_\xi,\alpha,\beta_1)$};
		\node [style=none] (8) at (-5.5, -1.25) {$\times$};
		\node [style=none] (9) at (0, -1.25) {$\times$};
		\node [style=none] (10) at (5.75, -1) {};
		\node [style=none] (11) at (3.75, -2) {};
		\node [style=none] (12) at (5.75, -2) {};
		\node [style=none] (13) at (7.25, -2) {};
		\node [style=none] (14) at (3.75, -2.5) {$\times$};
		\node [style=none] (15) at (5.75, -2.5) {$\times$};
		\node [style=none] (17) at (8.5, -2.75) {};
		\node [style=none] (18) at (5.5, -4) {};
		\node [style=none] (19) at (8, -4) {};
		\node [style=none] (20) at (11.5, -4) {};
		\node [style=none] (21) at (5.5, -4.5) {$\times$};
		\node [style=none] (22) at (8, -4.5) {$\times$};
		\node [style=none] (23) at (11.75, -4.5) {$(-,\xi,\eta_\xi,\alpha,\beta_{m-1})$};
		\node [style=none] (24) at (11.5, -5.25) {$\times$};
		\node [style=none] (25) at (8, -2.25) {$\ddots$};
	\end{pgfonlayer}
	\begin{pgfonlayer}{edgelayer}
		\draw (0.center) to (2.center);
		\draw (3.center) to (0.center);
		\draw (0.center) to (4.center);
		\draw (10.center) to (11.center);
		\draw (12.center) to (10.center);
		\draw (10.center) to (13.center);
		\draw (17.center) to (18.center);
		\draw (19.center) to (17.center);
		\draw (17.center) to (20.center);
	\end{pgfonlayer}
\end{tikzpicture}
\caption{The walk from $(\xi,\eta_\xi,\alpha,\beta')$. }
\label{1dstep}
\end{figure}

In Figure \ref{1dstep}, for $\xi \in (\varepsilon,\alpha)$, the node $(\xi,\alpha,\beta_1,\beta')$ is terminal, because $\beta_1=\min(C_\beta'\setminus\alpha)$ and hence $C_{\beta_1\beta'}\setminus\alpha=\emptyset$. The node $(\xi,\eta_\xi,\beta_1,\beta')$ must then also terminal, since $\alpha\not\in C_{\beta'}$ and so $C_{\beta'}\cap [\eta_\xi,\beta_1)=\emptyset$ for large enough $\xi<\alpha$. The remaining node, $(\xi,\eta_\xi,\alpha,\beta_1)$, is obtained from the original node $(\xi,\eta_\xi,\alpha,\beta')$ by performing a one-dimensional step from $\beta'$ down to $\alpha$ and replacing $\beta'$ with the next step on that lower walk. 

Studying Figure \ref{1dstep}, we see that the entire walk down from $(\xi,\eta_\xi,\alpha,\beta')$ is a simulacrum of the one-dimensional walk from $\beta'$ down to $\alpha$, with two dummy off shoots at each step. Finally, the node $(\xi,\eta_\xi,\alpha,\beta_{m-1})$ is always terminal, for large enough $\xi$: if $\alpha\not\in\acc(C_{\beta_{m-1}})$, then $C_{\alpha\beta_{m-1}}\setminus \xi=\emptyset$ for large enough $\xi<\alpha$. If instead $\alpha\in\acc(C_{\beta_{m-1}})$, then by coherence $C_\alpha=C_{\alpha\beta_{m-1}}$ and, since $\eta_\xi=\min(C_\alpha\setminus\xi)$, it follows that $C_{\eta_\xi\alpha\beta_{m-1}}\setminus\xi=\emptyset$.

Taking stock, we have that the signed tree underneath the node $(-,\xi,\eta_\xi,\alpha,\beta')$ is as in Figure \ref{1dstep} for all large enough $\xi<\alpha$. In particular, the contribution to $g_t(\xi)$ is a constant value in the group $\bigoplus_{[\lambda]^2}\Z$; namely that which results from considering the labelled tree in Figure \ref{1dstep}, deleting the first two ordinal coordinates, and adding up the results (in the aforementioned group). Once again, the key point is that this does not depend on $\xi$, as long as $\xi$ is sufficiently large (below $\alpha$).

\item Suppose that $\alpha\not\in\acc(C_{\beta'\gamma'})$ and set $\bar\alpha:=\max(\alpha\cap C_{\beta'\gamma'})$. The situation at the bad node for $\xi>\bar\alpha$ is displayed in Figure \ref{bad_node_not_acc}.

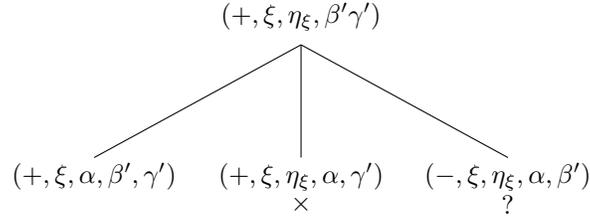
\begin{figure}[ht]
\begin{tikzpicture}
	\begin{pgfonlayer}{nodelayer}
		\node [style=none] (0) at (0, 3) {};
		\node [style=none] (1) at (-5.5, 0) {};
		\node [style=none] (2) at (0, 0) {};
		\node [style=none] (3) at (5.5, 0) {};
		\node [style=none] (4) at (0, 3.75) {$(+,\xi,\eta_\xi,\beta'\gamma')$};
		\node [style=none] (5) at (-5.5, -0.5) {$(+,\xi,\alpha,\beta',\gamma')$};
		\node [style=none] (6) at (0, -0.5) {$(+,\xi,\eta_\xi,\alpha,\gamma')$};
		\node [style=none] (7) at (5.5, -0.5) {$(-,\xi,\eta_\xi,\alpha,\beta')$};
		\node [style=none] (8) at (0, -1.25) {$\times$};
		\node [style=none] (9) at (5.5, -1.25) {?};
	\end{pgfonlayer}
	\begin{pgfonlayer}{edgelayer}
		\draw (0.center) to (1.center);
		\draw (2.center) to (0.center);
		\draw (0.center) to (3.center);
	\end{pgfonlayer}
\end{tikzpicture}
\caption{If $\xi>\bar\alpha$, then $\eta_\xi\not\in C_{\beta'\gamma'}$ and $\min(C_{\beta'\gamma'}\setminus \eta_\xi)=\alpha$.}
\label{bad_node_not_acc}
\end{figure}
Comparing with Figure \ref{good_step}, we see that the cones underneath $(\pm,\xi,\alpha,\beta',\gamma')$ cancel out. We now claim that the node $(+,\xi,\eta_\xi,\alpha,\gamma')$ is always terminal (for sufficiently large $\xi$). Indeed, if $\alpha\in\acc(C_{\gamma'})$, then by coherence $C_\alpha=C_{\alpha\gamma'}$, so in particular $\eta_\xi\in C_{\alpha\gamma'}$ and $C_{\eta_\xi\alpha\gamma'}\subset\xi=C_\alpha\cap [\xi,\eta_\xi)=\emptyset$. If instead $\alpha\not\in\acc(C_{\gamma'})$, set $\alpha^*=\max(\alpha\cap C_{\gamma'})$. Then, if $\xi>\alpha^*$, it follows that $\eta_\xi\not\in C_{\gamma'}$ and therefore $C_{\alpha\gamma'}\setminus \eta_\xi \subset C_{\gamma'}\cap[\eta_\xi,\alpha)=\emptyset$. As for the node $(-,\xi,\eta_\xi,\alpha,\beta')$, this was already analysed before, see Figure \ref{1dstep} and the surrounding discussion.

The upshot now is that, for large enough $\xi$, the values of $g_t(\xi)$ are constant, for we have two cancelling cones plus a constant tree stemming from the simulated one-dimensional walk, as in Figure \ref{1dstep}.
\end{itemize}
\item Suppose now that $\beta'\in C_{\gamma'}$, $\alpha\not\in C_{\beta'\gamma'}$ and $C_{\beta'\gamma'}\setminus\alpha=\emptyset$. Referring to Figure \ref{goodFig}, note that the node $(-\xi,\alpha,\beta',\gamma')$ is still terminal. On the other hand, since $C_{\beta'\gamma'}$ is closed, we can let $\bar\alpha:=\max(\alpha\cap C_{\beta'\gamma'})<\alpha$. If $\xi>\bar\alpha$, then $C_{\beta'\gamma'}\setminus\eta_\xi=\emptyset$, hence the node $(+,\xi,\eta_\xi,\beta',\gamma')$ in Figure \ref{alpha_xi_sides} is also terminal. Hence they both contribute equally to the computation of $g_t$.
\end{itemize}
This completes the proof.
\end{proof}

\subsection{Interlude: truncated and simulated walks}

In the proof of Theorem \ref{main3}, we bore witness to the following phenomenon: a $3$-dimensional walk ``simulated" a $1$-dimensional walk. This is not a coincidence. In our analysis of higher walks, the following scenario will repeatedly occur: 
during the course of an $(n+2)$-dimensional walk, we will encounter a node of 
the form $(\xi,\eta,\alpha,\vec{\gamma})$, and the walk from this node will 
have essentially the same structure as a truncated $n$-dimensional walk from 
$(\alpha,\vec{\gamma})$. In what follows, we will make this more precise and 
isolate the key features of this scenario.

\begin{definition}
  Suppose that $\mc{C}$ is an $n$-$C$-sequence on an ordinal $\lambda$ and
  $(\alpha,\vec{\gamma}) \in \lambda^{[n+1]}$. Define a 
  tree $S_n^-(\alpha, \vec{\gamma}) \subseteq {^{<\omega}}n$ as follows; 
  note that the previously defined $S_n(\alpha,\vec{\gamma})$ will be an end-extension 
  of $S_n^-(\alpha, \vec{\gamma})$. First, require that $\emptyset \in 
  S_n^-(\alpha, \vec{\gamma})$. Next, suppose that $x \in S_n^-(\alpha,\vec{\gamma})$ 
  and that $\mathrm{Tr}_n(\alpha, \vec{\gamma})(x) = (\alpha, 
  \vec{\beta})$. Then declare $x$ to be a terminal node of $S_n^-(\alpha,\vec{\gamma})$ if either
  \begin{enumerate}
    \item $x$ is a terminal node of $S_n(\alpha, \vec{\gamma})$; \emph{or}
    \item $\tau(\alpha,\vec{\beta}) = \vec{\beta}$ and $\alpha \in C_{\vec{\beta}}$.
  \end{enumerate}
  If $x$ is not a terminal node of $S_n^-(\alpha,\vec{\gamma})$, then demand that 
  $x^\smallfrown \langle i \rangle \in S_n^-(\alpha,\vec{\gamma})$ for all $i < n$. 
  Set $\mathrm{Tr}_n^-((-1)^k,\alpha,\vec{\gamma}) = \mathrm{Tr}_n((-1)^k,\alpha,
  \vec{\gamma}) \restriction S_n^-(\alpha,\vec{\gamma})$, and define
  \[
    L_n^-(\alpha,\vec{\gamma}) = \max\{L_n(\alpha,\vec{\gamma})(x) \mid 
    x \in S_n^-(\alpha,\vec{\gamma}) \wedge L_n(\alpha,\vec{\gamma})(x) < \alpha\}.
  \]
  Note that we always have $L_n^-(\alpha,\vec{\gamma}) < \alpha$.
\end{definition}

Note that, if $1 \leq m < n < \omega$ and $\mc{C}$ is an $n$-$C$-sequence 
on an ordinal $\lambda$, then $\mc{C}$ naturally induces an $m$-$C$-sequence $\mc{C}'$ on 
$\lambda$. Namely, $I(\mc{C}') = I(\mc{C}) \cap [\lambda]^{\leq m}$ and 
$C'_{\vec{\gamma}} = C_{\vec{\gamma}}$ for all $\vec{\gamma} \in I(\mc{C}')$. 
If we are in a context in which we have fixed some $n$-$C$-sequence $\mc{C}$ and 
we discuss performing an $m$-dimensional walk for some $1 \leq m < n$, it should be 
understood that this walk is the one defined using this $\mc{C}'$.

We next introduce some notation regarding stretching $m$-trees to higher dimensions.

\begin{definition}
  Suppose that $1 \leq m < n < \omega$.
  \begin{enumerate}
    \item Given $x \in {^{<\omega}}m$, define $s_{m,n}(x) \in {^{<\omega}}n$ by 
    setting $\dom(s_{m,n}(x)) = \dom(x)$ and, for all $k < \dom(x)$, setting 
    $s_{m,n}(x)(k) = x(k) + (n-m)$.
    \item Given a full $m$-tree $S \subseteq {^{<\omega}}m$, define a full 
    $n$-tree $S^n$ as follows:
    \begin{enumerate}
      \item for all $x \in S$, put $s_{m,n}(x) \in S^n$;
      \item for all $x \in S$, $x$ is a terminal node of $S$ if and only if 
      $s_{m,n}(x)$ is a terminal node of $S^n$;
      \item for all splitting nodes $x \in S$ and all $i < (n-m)$, 
      $s_{m,n}(x)^\smallfrown \langle i \rangle$ is a terminal node of $S^n$.
    \end{enumerate}
  \end{enumerate}
\end{definition}

\begin{figure}[ht]
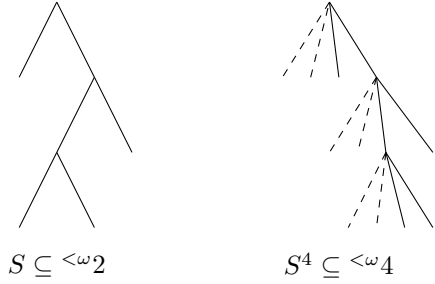

\ctikzfig{augmented_tree}
\caption{The effect of stretching a full $2$-tree into a full $4$-tree.}
\label{augmented_tree}
\end{figure}

We now present our main technical lemma about simulating $m$-dimensional walks inside 
$(m+2)$-dimensional walks. Before reading the lemma, recall the definition of 
$X(\mc{C})$ from Definition \ref{def: n_c_sequence}.

\begin{lemma}\label{simulation}
  Suppose that $m$ is a positive integer and $\mc{C}$ is a coherent 
  $(m+2)$-sequence on an ordinal $\lambda$. Suppose that $\vec{\gamma} \in 
  \lambda^{[m]}$ and $\alpha \leq \gamma_0$ is in $X(\mc{C})$.
  Suppose moreover that $L_m^-(\alpha,\vec{\gamma}) < \xi < \alpha$, and
  let $\eta_\xi = \min(C_\alpha \setminus \xi)$. Then
  \begin{enumerate}
    \item $S_{m+2}(\xi,\eta_\xi,\alpha,\vec{\gamma}) = 
    (S^-_m(\alpha,\vec{\gamma}))^{m+2}$; and
    \item for all $x \in S^-_m(\alpha,\vec{\gamma})$, if 
    $\mathrm{Tr}_m(+,\alpha,\vec{\gamma})(x) = ((-1)^j,\alpha,\vec{\beta})$, 
    then 
    \[
      \mathrm{Tr}_{m+2}(+,\xi,\eta_\xi,\alpha,\vec{\gamma})(s_{m,m+2}(x)) = 
      ((-1)^j,\xi,\eta_\xi,\alpha,\vec{\beta}).
    \]
  \end{enumerate}
\end{lemma}

\begin{proof}
  We will prove the following statements for all $x \in S^-_m(\alpha,\vec{\gamma})$ 
  by induction on $|x|$:
  \begin{enumerate}[label = (\alph*)]
    \item $s_{m,m+2}(x) \in S_{m+2}(\xi,\eta_\xi,\alpha,\vec{\gamma})$;
    \item if $\mathrm{Tr}_m(+,\alpha,\vec{\gamma})(x) = 
    ((-1)^j,\alpha,\vec{\beta})$, then 
    $\mathrm{Tr}_{m+2}(+,\xi,\eta_\xi,\alpha,\vec{\gamma})(s_{m,m+2}(x)) = 
    ((-1)^j,\xi,\eta_\xi,\alpha,\vec{\beta})$;
    \item $s_{m,m+2}(x)$ is a terminal node of $S_{m+2}(\xi,\eta_\xi,\alpha, 
    \vec{\gamma})$ if and only if $x$ is a terminal node of 
    $S^-_m(\alpha,\vec{\gamma})$;
    \item if $x$ is a splitting node of $S^-_m(\alpha,\vec{\gamma})$, then, for $i < 2$, 
    $s_{m,m+2}(x)^\smallfrown \langle i \rangle$ is a terminal node of 
    $S_{m+2}(\xi,\eta_\xi,\alpha,\vec{\gamma})$.
  \end{enumerate}
  This will clearly suffice to establish the lemma.
  We begin with the base case, i.e., $x = \emptyset$. Items (a) and (b) are 
  immediate. To see (c), let $\tau(\vec{\gamma})$ denote the longest final segment 
  of $\vec{\gamma}$ that is in $I(\mc{C})$. Let 
  $k = |\tau(\vec{\gamma})|$ and $j = (m-1)-k$. In particular, setting 
  $\alpha = \gamma_{-1}$ for convenience, $\gamma_j$ is the largest ordinal appearing 
  in $\langle \alpha \rangle^\smallfrown \vec{\gamma}$ that does not appear in 
  $\tau(\vec{\gamma})$. There are now a number of options:
  \begin{itemize}
    \item If $\tau(\vec{\gamma}) = \vec{\gamma}$ and $\alpha \in C_{\vec{\gamma}}$, 
    then $x$ is a terminal node of $S^-_m(\alpha,\vec{\gamma})$. There are now two 
    subcases to consider:
    \begin{itemize}
      \item If $\alpha \in \acc(C_{\vec{\gamma}})$, then we have 
      $C_{\alpha\vec{\gamma}}= C_\alpha$. In particular, 
      $\langle \eta_\xi,\alpha \rangle ^\smallfrown \vec{\gamma} \in I(\mc{C})$ 
      and $C_{\eta_\xi\alpha\vec{\gamma}} \setminus \xi = \emptyset$, so $x = s_{m,m+2}(x)$ 
      is a terminal node of $S_{m+2}(\xi,\eta_\xi,\alpha,\vec{\gamma})$.
      \item If $\alpha \in \nacc(C_{\vec{\gamma}})$, then we have 
      $\max(C_{\vec{\gamma}} \cap \alpha) \leq L^-_m(\alpha,\vec{\gamma}) < 
      \xi < \eta_\xi < \alpha$. In particular, $\langle \alpha \rangle ^\smallfrown 
      \vec{\gamma}$ is the longest final segment of $\langle \xi, \eta_\xi, 
      \alpha, \vec{\gamma} \rangle$ in $I(\mc{C})$ and $C_{\alpha\vec{\gamma}} 
      \setminus \eta_\xi = \emptyset$, so again $x$ is a terminal node of 
      $S_{m+2}(\xi,\eta_\xi,\alpha,\vec{\gamma})$.
    \end{itemize}
    \item If $C_{\tau(\vec{\gamma})} \setminus \gamma_j = \emptyset$, then again 
    $x$ is a terminal node of $S^-_m(\alpha,\vec{\gamma})$. In this case, 
    $\tau(\vec{\gamma})$ is also the longest final segment of 
    $\langle \xi,\eta_\xi,\alpha \rangle^\smallfrown \vec{\gamma}$ that is in 
    $I(\mc{C})$ and $\gamma_j$ is the largest element of $\langle \xi, 
    \eta_\xi, \alpha \rangle ^\smallfrown \vec{\gamma}$ not appearing in 
    $\tau(\vec{\gamma})$, so it immediately follows 
    that $x$ is also a terminal node of $S_{m+2}(\xi,\eta_\xi,\alpha,\vec{\gamma})$.
    \item If we are in none of the above cases, then $x$ is a splitting node 
    of $S^-_m(\alpha,\vec{\gamma})$. In particular, we have $\gamma_j \notin 
    C_{\tau(\vec{\gamma})}$ and $C_{\tau(\vec{\gamma})} \setminus \gamma_j 
    \neq \emptyset$. As in the previous item, we again see that $\tau(\vec{\gamma})$ 
    is the longest final segment of $\langle \xi,\eta_\xi,\alpha \rangle^\smallfrown 
    \vec{\gamma}$ that is in $I(\mc{C})$, and the fact that 
    $C_{\tau(\vec{\gamma})} \setminus \gamma_j \neq \emptyset$ implies that 
    $x$ is a splitting node of $S_{m+2}(\xi,\eta_\xi,\alpha,\vec{\gamma})$.
  \end{itemize}
  We finally verify (d). If $x$ is a splitting node of $S^-_m(\alpha, \vec{\gamma})$, 
  then we are in the case of the final bullet point above. Let 
  $\beta = \min(C_{\tau(\vec{\gamma})} \setminus \gamma_j)$, and note that 
  $\gamma_j < \beta$. In this case, recalling the definition of $\mathrm{Tr}_{m+2}$, we have
  \[
    \mathrm{Tr}_{m+2}(\xi,\eta_\xi,\alpha,\vec{\gamma})(\langle 0 \rangle) = 
    (\xi, \alpha, \gamma_0, \ldots, \gamma_j, \beta, \tau(\vec{\gamma})),
  \]
  which is terminal due to the fact that $\langle \beta \rangle ^\smallfrown 
  \tau(\vec{\gamma}) \in I(\mc{C})$ and $C_{\beta \tau(\vec{\gamma})} \setminus 
  \gamma_j = \emptyset$.
  
  If $j > -1$, then we have
  \[
    \mathrm{Tr}_{m+2}(\xi,\eta_\xi,\alpha,\vec{\gamma})(\langle 1 \rangle) = 
    (\xi, \eta_\xi, \gamma_0, \ldots, \gamma_j, \beta, \tau(\vec{\gamma})),
  \]
  which is terminal for the same reason. If $j = -1$, and hence 
  $\tau(\vec{\gamma}) = \vec{\gamma}$, then we have
  \[
    \mathrm{Tr}_{m+2}(\xi,\eta_\xi,\alpha,\vec{\gamma})(\langle 1 \rangle) = 
    (\xi, \eta_\xi, \beta, \tau(\vec{\gamma})).
  \]
  In this case, we know that $\max(C_{\tau(\vec{\gamma})} \cap \alpha) 
  \leq L^-_m(\alpha,\vec{\gamma}) < \xi < \eta_\xi < \alpha$, and hence we have 
  $\langle \beta \rangle ^\smallfrown \tau(\vec{\gamma}) \in I(\mc{C})$ and 
  $C_{\beta \tau(\vec{\gamma})} \setminus \eta_\xi = \emptyset$, so again 
  $\langle 1 \rangle$ is terminal in $S_{m+2}(\xi,\eta_\xi,\alpha,\vec{\gamma})$. 
  This completes the base case of the induction.
  
  Now suppose that $x \in S^-_m(\alpha, \vec{\gamma})$ is a splitting node and 
  we have established (a)--(d) for $x$. Fix $i < m$; we will establish 
  (a)--(d) for $x^\smallfrown \langle i \rangle$. First note that 
  $s_{m,m+2}(x^\smallfrown \langle i \rangle) = s_{m,m+2}(x)^\smallfrown \langle i+2 \rangle$. 
  Item (a) now follows immediately from the fact that, by the inductive hypothesis, 
  $s_{m,m+2}(x)$ is a splitting node of $S_{m+2}(\xi,\eta_\xi,\alpha,\vec{\gamma})$.
  
  To see (b), let $\mathrm{Tr}_m(+,\alpha,\vec{\gamma})(x) = ((-1)^\ell, \alpha, 
  \vec{\beta})$. As in the base case, let $\tau(\vec{\beta})$ be the longest final 
  segment of $\vec{\beta}$ that is in $I(\mc{C})$, let $k = |\tau(\vec{\beta})|$, 
  let $j = (m-1)-k$ and, for convenience, set $\beta_{-1} = \alpha$. Since $x$ is a 
  splitting node, it must be the case that $\beta_j \notin C_{\tau(\vec{\beta})}$ 
  and $C_{\tau(\vec{\beta})} \setminus \beta_j \neq \emptyset$. By the inductive 
  hypothesis, we have $\mathrm{Tr}_{m+2}(+,\xi,\eta_\xi,\alpha,\vec{\gamma})
  (s_{m,m+2}(x)) = ((-1)^\ell, \xi, \eta_\xi, \alpha, \vec{\beta})$; moreover, $\tau(\vec{\beta})$ 
  must be the longest final segment of $\langle \xi,\eta_\xi,\alpha \rangle^\frown \vec{\beta}$ 
  that is in $I(\mc{C})$.
  
  Let $\delta = \min(C_{\tau(\vec{\beta})} \setminus \beta_j)$. If $i \leq j$, then 
  \[
    \mathrm{Tr}_m(+,\alpha,\vec{\gamma})(x^\smallfrown \langle i \rangle) = 
    ((-1)^{\ell+j+i},\alpha, (\vec{\beta} \restriction (j+1))^i, \delta, \tau(\vec{\beta}))
  \]
  and 
  \[
    \mathrm{Tr}_{m+2}(+,\xi,\eta_\xi,\alpha,\vec{\gamma})(s_{m,m+2}(x)^\smallfrown 
    \langle i+2 \rangle) = ((-1)^{\ell+j+i}, \xi,\eta_\xi,\alpha, (\vec{\beta} \restriction (j+1))^i, 
    \delta, \tau(\vec{\beta})).
  \]
  If $i > j$, then 
  \[
    \mathrm{Tr}_m(+,\alpha,\vec{\gamma})(x^\smallfrown \langle i \rangle) = 
    ((-1)^{\ell+j+i+1}, \alpha, \vec{\beta} \restriction (j+1), \delta, \tau(\vec{\beta})^{i-(j+1)})
  \]
  and 
  \[
    \mathrm{Tr}_{m+2}(+,\xi,\eta_\xi,\alpha,\vec{\gamma})(s_{m,m+2}(x)^\smallfrown 
    \langle i+2 \rangle) = ((-1)^{\ell+j+i+1}, \xi,\eta_\xi,\alpha, \vec{\beta} \restriction (j+1), 
    \delta, \tau(\vec{\beta})^{i-(j+1)}).
  \]
  In either case, $x^\smallfrown \langle i \rangle$ satisfies item (b).
  
  Items (c) and (d) are now verified exactly as in the base case, \emph{mutatis 
  mutandis}. We therefore leave this to the reader.
\end{proof}

\subsection{The general case}

Before giving the proof of the main result, we need to generalize some lemmas from the three-dimensional setting. First, we wish to extend Lemma \ref{alphaAlphaTerminal} to the case of a general $n\ge 3$. For reasons of clarity, it is convenient to unfold the lemma into two halves:

\begin{lemma}\label{alphaAlphaimmediate}
Let $n\ge 3$. Let $\C$ be a coherent $n$-$C$-sequence on an ordinal $\lambda$, and let 
$(\alpha,\vec{\gamma}) \in [\lambda]^n$. Let $(\alpha,\beta,\vec\delta)$ with $\alpha<\beta$ be an immediate descendant of the node $(\alpha,\alpha,\vec\gamma)$. Then $(\alpha,\beta,\vec\delta)$ is terminal.
\end{lemma}

\begin{proof} Write $(\alpha,\alpha,\vec\gamma)=\iota(\alpha,\alpha,\vec\gamma)^\smallfrown\tau(\alpha,\alpha,\vec\gamma)$ and set $\tau:=\tau(\alpha,\alpha,\vec\gamma)$ to simplify the notation. We case on the length of $\tau$:
\begin{itemize}
\item Suppose first that $\tau$ is a proper final segment of $\vec\gamma$, say $\tau=\vec\gamma\restr (j+1,n-1)$. Put $\beta:=C_\tau\setminus\gamma_j$, and note that the only immediate descendant of $(\alpha,\alpha,\vec\gamma)$ whose first and second coordinates are different is
\[
(\alpha,\vec\gamma\restr(j+1),\beta,\tau).
\]
Clearly, $\beta\in C_\tau$ and $\gamma_j\not\in C_{\tau}$. Since $C_{\beta\tau}\setminus \gamma_j \subset C_\tau \cap [\gamma_j,\beta)=\emptyset$, it follows that $(\alpha,\vec\gamma\restr(j+1),\beta,\tau)$ must be terminal.

\item Suppose that $\tau=\vec\gamma$. By definition of $\tau$, $\alpha\not\in C_{\vec\gamma}$. Let $\alpha':=\min(C_{\vec\gamma}\setminus \alpha)$. Then the only immediate descendant of $(\alpha,\alpha,\vec\gamma)$ which has distinct first and second coordinates is the node $(\alpha,\alpha',\vec\gamma)$. Now note that $C_{\alpha'\vec\gamma}\setminus\alpha\subset C_{\vec{\gamma}}[\alpha,\alpha')=\emptyset$, so $(\alpha,\alpha',\vec\gamma)$ is terminal.

\item The case $\tau=(\alpha,\vec\gamma)$ is impossible because it would imply that $(\alpha,\alpha,\vec\gamma)$ is terminal.\qedhere
\end{itemize}
\end{proof}

\begin{lemma}\label{nAlphaTerminal}
Let $n\ge 3$. Let $\C$ be a coherent $n$-$C$-sequence on an ordinal $\lambda$, and let
$(\alpha,\vec{\gamma}) \in [\lambda]^n$. Then every node of the form $(\alpha,\beta,\vec\delta)$ with $\alpha<\beta$ occurring along the walk down from $(\alpha,\alpha,\vec\gamma)$ must be terminal. Moreover, for each such $(\alpha,\beta,\vec{\delta})$, there exists $\nu<\alpha$ such that for every $\xi\in [\nu,\alpha)$, the node $(\xi,\beta,\vec\delta)$ is terminal.
\end{lemma}
\begin{proof}
First, note that there is exactly one immediate descendant of $(\alpha,\alpha,\vec\gamma)$ which is not spectacled. Therefore, by iterating Lemma \ref{alphaAlphaimmediate}, we may assume that $(\alpha,\alpha,\vec\gamma)$ immediately descends to $(\alpha,\beta,\vec\delta)$ and therefore that the latter is terminal. To see the moreover, let $\tau=\tau(\alpha,\alpha,\vec\gamma)$ and case on $|\tau|$:
\begin{itemize}
\item If $\tau$ is a proper final segment of $\vec\gamma$, say $\tau=\vec\gamma\restr(j+1,n-1)$, then $(\alpha,\beta,\vec\delta)=(\alpha,\vec\gamma\restr(j+1),\alpha',\tau)$, where $\alpha':=\min(C_\tau\setminus\gamma_j)$. Note that $\gamma_j\not\in C_{\alpha'\tau}$ by the choice of $\alpha'$ and $C_{\alpha'\tau}\setminus \gamma_j=\emptyset$, hence $(\xi,\vec\gamma\restr(j+1),\alpha',\tau)$ is terminal for any $\xi\le\alpha$.
\item Suppose that $\tau=\vec\gamma$. By definition of $\tau$, $\alpha\not\in C_{\vec\gamma}$. Setting $\alpha':=\min(C_{\vec\gamma}\setminus\alpha)$, we see that $(\alpha,\beta,\vec\delta)=(\alpha,\alpha',\vec\gamma)$. Also,
there exists $\nu<\alpha$ such that $[\nu,\alpha]\cap C_{\vec\gamma}=\emptyset$. But then $(\xi,\beta,\vec\delta)$ is terminal whenever $\nu<\xi<\alpha$.
\item $\tau=(\alpha,\vec\gamma)$ is impossible because it implies that $(\alpha,\alpha,\vec\gamma)$ is terminal.\qedhere
\end{itemize}
\end{proof}

\begin{lemma}\label{nbadImmediate}
Let $n\ge 3$. Let $\C$ be a coherent $n$-$C$-sequence on an ordinal $\lambda$, and let 
$(\alpha,\vec{\gamma}) \in \lambda \otimes [\lambda]^n$. Suppose that $x\in S_n(\alpha,\vec\gamma)$ is bad. Then no $y\supsetneq x$ is bad for $(\alpha,\vec{\gamma})$.
\end{lemma}

\begin{proof}
The argument is identical to that of Lemma \ref{badImmediate}, appealing to Lemmas \ref{nAlphaTerminal} and \ref{alphaAlphaimmediate} instead of Lemma \ref{alphaAlphaTerminal}.
\end{proof}

\begin{lemma}\label{nbadNodes}
Let $n\ge 3$. Let $\C$ be a coherent $n$-$C$-sequence on an ordinal $\lambda$, and let
$(\alpha,\vec{\gamma}) \in \lambda \otimes [\lambda]^n$. Suppose that $x\in S_n(\alpha,\vec\gamma)$ is bad. If $\xi<\alpha$, let $\eta_\xi:=\min(C_\alpha\setminus\xi)$. There exists $\xi^*<\alpha$ such that, for every $\xi\in [\xi^*,\alpha)$,
\begin{enumerate}[(i)]
\item if $y\in S_n(\alpha,\vec\gamma)$ and $y\subset x$, then $y\in S_n(\xi,\vec\gamma)$ and
\[
\Tr_n(\pm,\xi,\vec\gamma)(y)=\sub_1^\xi(\Tr_n(\pm,\alpha,\vec\gamma)(y));
\]
\item if $y\in S_n(\alpha,\vec\gamma)$ is spectacled and $x\subset y$, then $y\in S_n(\xi,\vec\gamma)$ and
\[
\Tr_n(\pm,\xi,\vec\gamma)(y)=\sub_{1,2}^{\xi,\eta_\xi}(\Tr_n(\pm,\alpha,\vec\gamma)(y));
\]
\item if $y\in S_n(\alpha,\vec\gamma)$ is not spectacled and $x\subset y$, then $y\in S_n(\xi,\vec\gamma)$ and
\[
\Tr_n(\pm,\xi,\vec\gamma)(y)=\sub_1^\xi(\Tr_n(\pm,\alpha,\vec\gamma)(y)).
\]
\end{enumerate}
\end{lemma}

\begin{proof}
This is proven in the same way as Lemma \ref{badNodes}.
\end{proof}

\begin{lemma}\label{end_ext_n}
Let $n\ge 3$. Let $\C$ be a coherent $n$-$C$-sequence on an ordinal $\lambda$, and let 
$(\alpha,\vec{\gamma}) \in \lambda \otimes [\lambda]^n$, with $\alpha$ a limit ordinal.
Then there exists $\xi^*<\alpha$ such that, for all $\xi\in [\xi^*,\alpha]$, $S_n(\pm,\xi,\vec\gamma)$ is an end-extension of $S_n(\pm,\alpha,\vec\gamma)$.
\end{lemma}

\begin{proof}
This is the same as the proof of Lemma \ref{end_ext}, replacing the appeal to Lemma \ref{badNodes} with one to \ref{nbadNodes}.
\end{proof}

\begin{lemma}\label{nEasyNodes}
Let $n\ge 3$. Let $\C$ be a coherent $n$-$C$-sequence on an ordinal $\lambda$, and let 
$(\alpha,\vec{\gamma}) \in \lambda \otimes [\lambda]^n$, with $\alpha$ a limit ordinal. Let $x\in S_n(\alpha,\vec\gamma)$ be terminal and non-spectacled. Then there exists $\xi^*<\alpha$ such that, for all $\xi\in [\xi^*,\alpha)$,
\[
\Tr_n(\pm,\xi,\vec\gamma)(x)=\sub_1^\xi(\Tr(\pm,\alpha,\vec\gamma)(x)).
\]
\end{lemma}
\begin{proof}
If $x$ descends from a bad node in $\Tr_n(\pm,\alpha,\vec\gamma)$, then use Lemma \ref{nbadNodes}(iii). Otherwise, apply Lemma \ref{7.7}.
\end{proof}

The following establishes Theorem B:

\begin{theorem}\label{mainn}
Let $n$ be a positive integer. Let $\C$ be a coherent $n$-$C$-sequence on an ordinal $\lambda$. Then the $n$-family
\[
\Phi(\resh_n):=\left\langle \resh_n(\cdot, \vec{\gamma}):\gamma_0 \ra \bigoplus_{[\lambda]^{n-1}} 
\ \middle| \ \vec{\gamma} \in [\lambda]^n \right\rangle
\]
is coherent modulo locally semi-constant functions, i.e., for every $\vec\beta\in[\lambda]^{n+1}$,
\[
\sum_{i=0}^n(-1)^i\resh_n(\cdot,\vec\beta^i)
\]
is locally semi-constant at every limit ordinal $\alpha\le\beta_0$.
\end{theorem}

\begin{proof}
The case $n=1$ is classical and the case $n=2$ is as in \cite[Theorem 7.2(3)]{bergfalk2024introductionhigherwalks}, \textit{mutatis mutandis} (it can also be established by an easy modification of the argument below). We therefore focus on the case $n\ge 3$.

Fix $n\ge 3$. We proceed like in the three-dimensional case: fix a limit ordinal $\alpha$ and a sequence $\vec\beta\in [\lambda]^{n+1}$ with $\alpha\le\beta_0$, and argue that the function $\sum_{i=0}^n(-1)^i\resh_n(\cdot,\vec\beta^i)$ is eventually constant for large enough $\xi<\alpha$.

Having given all the details in Theorem \ref{main3}, we now give a more informal sketch; all the main ideas already appeared in the $n=3$ case. As before, by Lemma \ref{7.12}, $\bigsqcup_{i\le n}\bS_n((-1)^i,\alpha,\vec\beta^i)$ can be partitioned into pairs $\{(i_+,t_+),(i_-,t_-)\}$ so that the values of $t_+$ and $t_-$ under the respective $\Tr_n$-functions are signed $(n+1)$-tuples of ordinals 
with opposite signs but equal ordinal entries. Fix such a pair of tuples, say $+\vec\zeta$ and $-\vec\zeta$, coming from inputs $(i_+,t_+)$ and $(i_-,t_-)$. If $\vec\zeta$ is not spectacled, then, by Lemma \ref{nEasyNodes}, the cones of $S_n(\xi,\vec\beta^{i_+})$ and $S_n(\xi,\vec\beta^{i_-})$ below $t_+$ and $t_-$, respectively, must coincide for all large enough $\xi<\alpha$. Moreover, the values of $\Tr_n((-1)^{i_+},\xi,\vec\beta^{i_+})$ and $\Tr_n((-1)^{i_-},\xi,\vec\beta^{i_-})$ on said cones must be the same but with opposite signs, hence cancel out in the computation of $\resh_n$.

We may therefore assume that $\vec\zeta$ is spectacled. If both $+\vec\zeta$ and $-\vec\zeta$ descend from no bad nodes, then, by Lemma \ref{7.7}, for all large enough $\xi<\alpha$ the portions of the signed trees underneath $t_+$ and $t_-$ must again cancel out in the computation of $\resh_n$. If both $+\vec\zeta$ and $-\vec\zeta$ descend from bad nodes, the same holds by an application of Lemma \ref{nbadNodes}(iv). All that remains is the case of a pair of spectacled nodes $+\vec\zeta=(+,\alpha,\alpha,\vec\gamma)$ and $-\vec\zeta=(-,\alpha,\alpha,\vec\gamma)$ where (say) the former descends from a bad node but the latter does not. Their $\xi$-images (recall Notation \ref{xiImage}) are therefore $(+,\xi,\eta_\xi,\vec\gamma)$ and $(-,\xi,\alpha,\vec\gamma)$, respectively.

Let us first suppose that $\vec\gamma$ is not a $\C$-index, i.e. that $\vec\gamma\not\in I(\C)$. Then there exists $i<n-1$ such that $\gamma_i\not\in C_{\gamma_{i+1}\dots\gamma_{n-2}}$ and $\seq{\gamma_{i+1},\dots,\gamma_{n-2}}\in I(\C)$. Since $(\pm,\alpha,\alpha,\vec\gamma)$ is terminal, it follows that $C_{\gamma_{i+1}\dots\gamma_{n-2}}\setminus\gamma_i=\emptyset$. But then $(\pm,\xi,\nu,\vec\gamma)$ is terminal for every $\xi<\alpha$ and every $\nu$ with $\xi\le\nu\le\gamma_0$, so the $\xi$-images $(+,\xi,\eta_\xi,\vec\gamma)$ and $(-,\xi,\alpha,\vec\gamma)$ are also terminal, hence their contribution to $\resh_n$ is constant.

By the previous paragraph, we may assume from now on that $\vec\gamma\in I(\C)$. Suppose now that $\alpha\not\in C_{\vec\gamma}$. Since $(\pm,\alpha,\alpha,\vec\gamma)$ is terminal, we know that $C_{\vec\gamma}\setminus \alpha=\emptyset$. As $C_{\vec\gamma}$ is closed, it follows that $C_{\vec\gamma}\setminus \xi=\emptyset$ for all large enough $\xi<\alpha$, hence the $\xi$-images of $(\pm,\alpha,\alpha,\vec\gamma)$ are also eventually terminal.

We now come to the harder case, $\alpha\in C_{\vec\gamma}$, which will occupy the bulk of our work. Say $(+,\alpha,\alpha,\vec\gamma)$ descends from the bad node $(\pm,\alpha,\vec\delta)$. By coherence, we have that $C_\alpha=\alpha\cap C_{\vec\delta}$. We now embark on some case analysis:
\begin{itemize}
\item Suppose that $\alpha\in\acc(C_{\vec\gamma})$. By coherence, $C_\alpha=\alpha\cap C_{\vec\gamma}=\alpha\cap C_{\vec\delta}$. Letting $\eta_\xi:=\min(C_\alpha\setminus\xi)$ for $\xi<\alpha$, we see that the walk down from $(\xi,\vec\delta)$ is as in Figure \ref{bad_n_xi}. For comparison, the $\xi$-image of the good case is displayed in Figure \ref{good_n}.

\begin{figure}[ht]
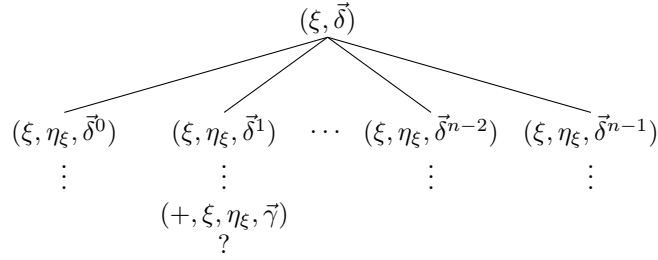

\ctikzfig{bad_n_xi}
\caption{We assume for concreteness that $(\alpha,\alpha,\vec\gamma)$ descends from $(\alpha,\alpha,\vec\delta^1)$.}
\label{bad_n_xi}
\end{figure}

\begin{figure}[ht]
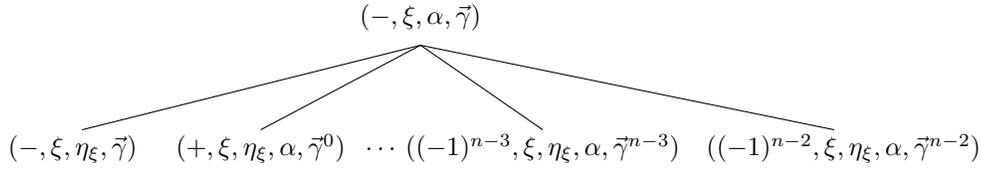

\ctikzfig{good_n}
\caption{The $\xi$-image of the good case.}
\label{good_n}
\end{figure}

We thus see that everything under the nodes $(+,\xi,\eta_\xi,\vec\gamma)$ and $(-,\xi,\eta_\xi,\vec\gamma)$ cancels out. We are therefore left with arguing that the behaviour under the nodes $(\xi,\eta_\xi,\alpha,\vec\gamma^i)$ for $i<n-1$ does not depend on $\xi$ so long as $\xi$ is sufficiently large. Fix $i<n-1$.

\begin{itemize}
\item If $\alpha\in \acc(C_{\vec \gamma^i})$, then $(\xi,\eta_\xi,\alpha,\vec\gamma^i)$ is terminal for all $\xi<\alpha$. Indeed, $\alpha\vec\gamma^i$ is a $\C$-index and, by coherence, $C_\alpha=C_{\alpha\vec\gamma^i}$, hence $\eta_\xi=\min(C_{\alpha\vec\gamma^i}\setminus\xi)$ and therefore 
\[C_{\eta_\xi\alpha\vec\gamma^i}\setminus\xi\subset C_{\alpha\vec\gamma^i}\cap [\xi,\eta_\xi)=\emptyset.\]

\item If $\alpha\in \nacc(C_{\vec \gamma^i})$, set $\bar\alpha=\max(\alpha\cap C_{\vec \gamma^i})$. If $\xi>\bar\alpha$, then $\eta_\xi\not\in C_{\vec\gamma^i}$, and hence $(\alpha,\vec\gamma^i)$ is the longest final segment of $(\xi,\eta_\xi,\alpha,\vec\gamma^i)$ that is a $\C$-index. But $C_{\alpha\vec\gamma_i}\setminus \eta_\xi\subset C_{\vec\gamma^i}\cap (\bar\alpha,\alpha)=\emptyset$. Therefore, $(\xi,\eta_\xi,\alpha,\vec\gamma^i)$ is terminal.

\item If $\alpha\not\in C_{\vec\gamma^i}$, then by Lemma \ref{simulation}, the walk simulates an $(n-2)$-dimensional walk down from the node $(\alpha,\vec\gamma^i)$. The tree $S_n(\xi,\eta_\xi,\alpha,\vec\gamma^i)$ is therefore independent of $\xi$ (for large enough $\xi<\alpha$) and moreover its labels (minus the first two ordinal entries) must also be independent of $\xi$ (for large enough $\xi<\alpha$), again by Lemma \ref{simulation}(1)
\end{itemize}
\item Suppose now that $\alpha\in\nacc(C_{\vec\gamma})$. Set $\bar\alpha=\max(\alpha\cap C_{\vec\gamma})$. If $\xi\in (\bar\alpha,\alpha)$, then the node $(-,\xi,\alpha,\vec\gamma)$ is terminal, because $C_{\alpha\vec\gamma}\setminus\xi \subset C_{\vec\gamma}\cap (\bar\alpha,\alpha)=\emptyset$. Similarly, $\eta_\xi\not\in C_{\vec\gamma}$ whenever $\xi\in (\bar\alpha,\alpha)$, and so the situation underneath the $\xi$-image $(+,\xi,\eta_\xi,\vec\gamma)$ of the node $(+,\alpha,\alpha,\vec\gamma)$ is as in Figure \ref{bad_n_nacc}.

\begin{figure}[h]
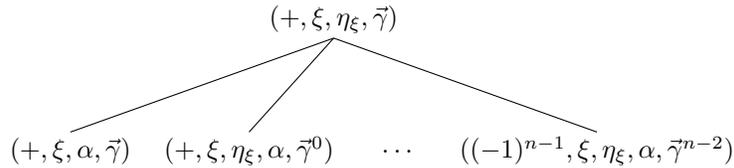

\ctikzfig{bad_n_nacc}
\caption{The $\xi$-image of the bad case when $\alpha\in\nacc(C_{\vec\gamma})$.}
\label{bad_n_nacc}
\end{figure}

The behaviour underneath the nodes $(\xi,\eta_\xi,\alpha,\vec\gamma^i)$ is then analysed as before.
\end{itemize}
This completes the proof.
\end{proof}

\section{Nontriviality} \label{section: nontriviality}

In this section, we prove that, for every positive integer $n$, performing 
walks along a $\Sq{n}^s(\lambda)$-sequence $\mc{C}$ yields a function 
$\resh^{\mc{C}}_n$ such that the $n$-family $\Phi(\resh^{\mc{C}}_n)$ is 
nontrivial in addition to being coherent. We begin with the classical case 
of $n = 1$, where this result was already known but, as far as we can tell, 
not recorded in the literature in this precise form. Recalling that $\resh_1$ is 
(modulo irrelevant cosmetic differences) the classical number of steps function $\rho_2$, 
we work here with $\rho_2$ rather than $\resh_1$.

\begin{fact}{\cite[Theorem 6.3.2]{todorcevic_walks_book}} \label{fact: unbddness}
  Suppose that $\lambda$ is an ordinal, $D$ is a club in $\lambda$, and 
  $\mc{C}$ is a $\Sq{1}^s(D)$-sequence. Then, for every unbounded set 
  $A \subseteq D$ and every $k < \omega$, there are $\alpha < \beta$, both 
  in $A$, such that $\rho_2^{\mc{C}}(\alpha,\beta) > k$.
\end{fact}

This immediately yields the following corollary stating that, if $\mc{C}$ is a $\Sq{1}^s(D)$ 
sequence, then the family $\Phi(\rho^{\mc{C}}_2)$ is nontrivial. Recall that, 
if $\lambda$ is an ordinal of uncountable cofinality, then the nonstationary 
ideal on $\lambda$ is weakly normal, i.e., for every stationary set 
$S \subseteq \lambda$ and every regressive function $f:S \ra \lambda$, there 
is a stationary set $S' \subseteq S$ and an ordinal $\alpha < \lambda$ such 
that, for all $\gamma \in S$, we have $f(\gamma) \leq \alpha$.

\begin{corollary} \label{cor: rho_nontriv}
  Suppose that $\lambda$ is an ordinal, $D$ is a club in $\lambda$, and 
  $\mc{C}$ is a $\Sq{1}^s(D)$-sequence. Then 
  \[
    \Phi(\rho_2^{\mc{C}}) = \langle \rho^{\mc{C}}_2(\cdot, \gamma) : D \cap \gamma 
    \ra \bb{Z} \mid \gamma \in D \rangle
  \]
  is nontrivial.
\end{corollary}

\begin{proof}
  Suppose for the sake of contradiction that 
  $\psi : D \ra \bb{Z}$ trivializes $\Phi(\rho_2^{\mc{C}})$. Then, for each 
  $\gamma \in \acc(D)$, we can find $\alpha_\gamma < \gamma$ and $j_\gamma \in 
  \bb{Z}$ such that $\rho_2^{\mc{C}}(\beta,\gamma) - \psi(\beta) = j_\gamma$ for 
  all $\beta \in (\alpha_\gamma,\gamma)$. Note that, since $\mc{C}$ is a $\Sq{1}^s(D)$-
  sequence, we must have $\cf(\lambda) \geq \aleph_1$. Therefore, we can fix 
  a stationary $S \subseteq D$, an $\alpha \in D$, and a $j \in \bb{Z}$ such that, 
  for all $\gamma \in S$, we have $\alpha_\gamma \leq \alpha$ and $j_\gamma = j$.
  We can also fix an unbounded set $A \subseteq S \setminus (\alpha + 1)$ and a 
  $k \in \bb{Z}$ such that, for all $\gamma \in A$, we have $\psi(\gamma) = k$. 
  It follows that, for all $\gamma < \delta$, both in $A$, we have 
  $\rho^{\mc{C}}_2(\gamma,\delta) = j_\delta + \psi(\gamma) = j + k$. 
  However, by Fact \ref{fact: unbddness}, we can find $\gamma < \delta$, both in 
  $A$, such that $\rho^{\mc{C}}_2(\gamma,\delta) > |j + k|$. This contradiction 
  completes the proof.
\end{proof}

We now proceed to the general case. We first need some technical lemmas.

\begin{proposition} \label{prop: prop_613}
  Suppose that $n$ is a positive integer, $\lambda$ is an ordinal, 
  $D$ is a club in $\lambda$, and $\mc{C}$ is an $(n+1)$-$C$-sequence on $D$. 
  Fix $\vec{\gamma} \in [D]^{n+2}$. For each $x \in S^{\mc{C}}_{n+1}(\vec{\gamma})$, 
  let $\varepsilon(\vec{\gamma})(x)$ denote the maximal ordinal in 
  $\Tr^{\mc{C}}_{n+1}(+,\vec{\gamma})(x)$. Then, for each $x \in 
  S^{\mc{C}}_{n+1}(\vec{\gamma})$, the following are equivalent:
  \begin{enumerate}
    \item $\varepsilon(\vec{\gamma})(x) = \gamma_{n+1}$;
    \item $x \in {^{<\omega}}n$.
  \end{enumerate}
\end{proposition}

\begin{proof}
  The proof is a straightforward induction on $|x|$, using the definition of 
  $\Tr^{\mc{C}}_{n+1}$ together with the observation that, for all 
  non-terminal nodes $x \in S^{\mc{C}}_{n+1}(\vec{\gamma})$ and all 
  $i < n+1$, we have $\varepsilon(\vec{\gamma})(x^\frown \langle i \rangle) 
  \leq \varepsilon(\vec{\gamma})(x)$.
\end{proof}

\begin{lemma} \label{lemma: dimension_reduction}
  Suppose that $n$ is a positive integer, $\lambda$ is an ordinal, $D$ is a club 
  in $\lambda$, and $\mc{C}$ is an $(n+1)$-$C$-sequence on $D$. 
  Fix $\delta \in \acc(D)$, and 
  let $\mc{C}^\delta = \langle C_{\vec{\gamma}\delta} \mid \vec{\gamma} \in I_\delta(\mc{C}) 
  \rangle$. Then, for all $\vec{\gamma} \in [C_\delta]^{n+1}$, 
  \begin{enumerate}
    \item $S^{\mc{C}^\delta}_n(\vec{\gamma}) = S^{\mc{C}}_{n+1}(\vec{\gamma},\delta) 
    \cap {^{<\omega}}n$;
    \item for all $x \in S^{\mc{C}^\delta}_n(\vec{\gamma})$, we have
    \begin{itemize}
      \item $\Tr^{\mc{C}}_{n+1}(+,\vec{\gamma},\delta)(x) = 
      (\Tr^{\mc{C}^\delta}_n(+,\vec{\gamma})(x))^\frown \langle 
      \delta \rangle$;
      \item $L_{n+1}^{\mc{C}}(\vec{\gamma},\delta)(x) = 
      L_n^{\mc{C}^\delta}(\vec{\gamma})(x)$.
    \end{itemize}
  \end{enumerate}
\end{lemma}

\begin{proof}
  Fix $\vec{\gamma} \in [C_\delta]^{n+1}$.
  By induction on $|x|$, we will prove that, for all $x \in 
  S^{\mc{C}}_{n+1}(\vec{\gamma},\delta) \cap {^{<\omega}}n$, the following 
  statements hold:
  \begin{enumerate}[label = (\alph*)]
    \item $x \in S^{\mc{C}^\delta}_n(\vec{\gamma})$;
    \item $\Tr^{\mc{C}}_{n+1}(+,\vec{\gamma},\delta)(x) = 
      (\Tr^{\mc{C}^\delta}_n(+,\vec{\gamma})(x))^\frown \langle 
      \delta \rangle$;
    \item $x$ is a terminal node of $S^{\mc{C}^\delta}_n(\vec{\gamma})$ 
    if and only if it is a terminal node of $S^{\mc{C}}_{n+1}(\vec{\gamma},\delta)$.
  \end{enumerate}
  This will clearly suffice to establish the lemma.
  
  If $x = \emptyset$, then (a) and (b) hold by assumption. To establish 
  (c), note that, since $\vec{\gamma} \in [C_\delta]^{n+1}$, we certainly 
  have $|\tau^{\mc{C}}(\vec{\gamma},\delta)| > 1$. It then follows that 
  $\tau^{\mc{C}^\delta}(\vec{\gamma})^\frown \langle \delta \rangle = 
  \tau^{\mc{C}}(\vec{\gamma},\delta)$ and $\iota^{\mc{C}^\delta}(\vec{\gamma}) 
  = \iota^{\mc{C}}(\vec{\gamma},\delta)$. Let $\gamma^*$ denote the maximal 
  entry in $\iota^{\mc{C}^\delta}(\vec{\gamma})$.
  Then $\emptyset$ is a terminal node of $S^{\mc{C}^\delta}_n(\vec{\gamma})$ 
  if and only if $C_{\tau^{\mc{C}^\delta}(\vec\gamma)^\frown\langle \delta\rangle} \setminus \gamma^* = \emptyset$, 
  which in turn holds if and only if $\emptyset$ is a terminal node of 
  $S^{\mc{C}}_{n+1}(\vec{\gamma},\delta)$, thus establishing (c).
  
  Suppose now that $x$ is a non-terminal node of $S^{\mc{C}}_{n+1}(\vec{\gamma}, 
  \delta)$ and we have established (a)--(c) for $x$. Fix $i < n$; we will 
  establish (a)--(c) for $x^\frown \langle i \rangle$. Note first that, 
  by the induction hypothesis, we know that $x$ is not a terminal node of 
  $S^{\mc{C}^\delta}(\vec{\gamma})$, and hence $x^\frown \langle i \rangle 
  \in S^{\mc{C}^\delta}(\vec{\gamma})$, establishing (a).
  
  Suppose that $\Tr^{\mc{C}^\delta}_n(+,\vec{\gamma})(x) = 
  ((-1)^m, \vec{\beta})$. By the induction hypothesis, it follows that 
  $\Tr^{\mc{C}}_{n+1}(+,\vec{\gamma},\delta) = ((-1)^m, \vec{\beta}, 
  \delta)$. In particular, we know that $\vec{\beta} \in [C_\delta]^{n+1}$, 
  and hence we have
  \begin{itemize}
    \item $\tau^{\mc{C}^\delta}(\vec{\beta})^\frown \langle \delta \rangle 
    = \tau^{\mc{C}}(\vec{\beta},\delta)$;
    \item $\iota^{\mc{C}^\delta}(\vec{\beta}) = \iota^{\mc{C}}(\vec{\beta}, 
    \delta)$.
  \end{itemize}
  Set $j + 1 = |\iota^{\mc{C}^\delta}(\vec{\beta})|$. Note that 
  $j \leq n$, and let $\langle \ell_i \mid i < n \rangle$ be the increasing 
  enumeration of the set $\{1, \ldots, n+1\} \setminus \{j+1\}$. 
  Let $\beta^*$ denote the maximal element of $\iota^{\mc{C}^\delta}(\vec{\beta})$, 
  and let $\alpha^* = \min(C_{\tau^{\mc{C}}(\vec{\beta},\delta)} 
  \setminus \beta^*)$. Then, we have
  \begin{align*}
    \Tr^{\mc{C}}_{n+1}(+,\vec{\gamma},\delta)(x^\frown \langle i \rangle) 
    & = ((-1)^{m+\ell_i}, (\iota^{\mc{C}}(\vec{\beta},\delta), 
    \beta^*, \tau^{\mc{C}}(\vec{\beta},\delta))^{\ell_i}) \\ 
    & = ((-1)^{m + \ell_i}, (\iota^{\mc{C}^\delta}(\vec{\beta}), 
    \beta^*, \tau^{\mc{C}^\delta}(\vec{\beta}))^{\ell_i}) ^\frown \langle 
    \delta \rangle \\ 
    & = (\Tr^{\mc{C}^\delta}_n(+,\vec{\gamma})(x^\frown \langle i 
    \rangle))^\frown \langle \delta \rangle.
  \end{align*}
  This establishes (b). Using the above calculations, item (c) is verified in 
  exactly the same way as it was in the base case, so we leave this to the 
  reader.
\end{proof}

Suppose that $n$ is a positive integer, $\delta < \lambda$ are infinite 
ordinals, and $D$ is a cofinal subset of $\delta$. Define a homomorphism $\pi^D_n : \bigoplus_{[\lambda]^n} 
\bb{Z} \ra \bigoplus_{[D]^{n-1}} \bb{Z}$ by letting 
$\pi^\delta_n(\lfloor \vec{\gamma}^\frown \langle \delta \rangle 
\rfloor) = \lfloor \vec{\gamma} \rfloor$ for all 
$\vec{\gamma} \in [D]^n$ and $\pi^\delta_n(\lfloor \vec{\beta} 
\rfloor) = 0$ for all $\vec{\beta} \in [\lambda]^{n+1}$ with 
$\beta_n \neq \delta$ or $\vec{\beta}^n \not\subseteq D$. 
The following is now immediate from Lemma \ref{lemma: dimension_reduction} and 
Definition \ref{rho2}.

\begin{corollary} \label{cor: resh_dim}
  Suppose that $n$ is a positive integer, $\lambda$ is an ordinal, 
  $D$ is a club in $\lambda$, 
  and $\mc{C}$ is an $(n+1)$-$C$-sequence on $D$. Suppose also that 
  $\delta \in \acc(D)$, and let $\mc{C}^\delta$ be the $n$-$C$-sequence 
  on $C_\delta$ defined by letting $I(\mc{C}^\delta) = \{\vec{\gamma} \mid 
  \vec{\gamma}^\frown \langle \delta \rangle \in I(\mc{C})\}$ and 
  $C^\delta_{\vec{\gamma}} = C_{\vec{\gamma}\delta}$ for all 
  $\vec{\gamma} \in I(\mc{C}^\delta)$. Then, for all $\vec{\gamma} \in 
  [C_\delta]^n$ and all $\alpha \in C_\delta \cap \gamma_0$, we have 
  \[
    \resh_n^{\mc{C}^\delta}(\alpha,\vec{\gamma}) = 
    \pi^{C_\delta}_n(\resh_{n+1}^{\mc{C}}(\alpha,\vec{\gamma},\delta)).
  \]
\end{corollary}

We are now ready for the main theorem of this section, which establishes Theorem C in the introduction.

\begin{theorem}
  Suppose that $n$ is a positive integer, $D$ is a club in $\lambda$, and $\mc{C}$ 
  is a $\Sq{n}^s(D)$-sequence. Then $\Phi(\resh^{\mc{C}}_n)$ is nontrivial modulo locally semi-constant functions.
\end{theorem}

\begin{proof}
  The proof is by induction on $n$. If $n = 1$, then the result follows from 
  Corollary \ref{cor: rho_nontriv}, so suppose that 
  $n_0 \geq 1$ and $n = n_0 + 1$. Suppose for the sake of contradiction that 
  $\Phi(\resh_n^{\mc{C}})$ is trivial, and fix a trivializing family 
  \[
    \Psi = \left\langle \psi_{\vec{\gamma}}: D \cap \gamma_0 \ra 
    \bigoplus_{[D]^n} \bb{Z} \ \middle| \ \vec{\gamma} \in [D]^{n_0} 
    \right\rangle.
  \]
  Since $\mc{C}$ is a $\Sq{n}^s(D)$-sequence, we know that 
  $\kappa := \otp(D)$ is a regular uncountable cardinal. It follows that the set
  \[
    E := \left\{\delta \in D ~ \middle| ~ \forall \vec{\gamma} \in 
    [D \cap \delta]^{n-1} ~ \psi_{\vec{\gamma}}[D \cap \gamma_0] \subseteq 
    \bigoplus_{[D \cap \delta]^n} \bb{Z}\right\}
  \]
  is a club subset of $D$. We can therefore find $\delta \in E$ such that 
  $\mc{C}^\delta$ is a strongly nontrivial coherent $n_0$-$C$-sequence on 
  $C_\delta$, and hence, by the inductive hypothesis, 
  $\Phi(\resh_{n_0}^{\mc{C}^\delta})$ is nontrivial. 
  
  Let us now make a couple of observations. For each 
  $\vec{\gamma} \in [C_\delta]^{n_0}$, the function
  \[
    f_{\vec{\gamma}} := \resh_n^{\mc{C}}(\cdot, \vec{\gamma},\delta) - 
    \left( (-1)^{n_0} \psi_{\vec{\gamma}} + \sum_{i < n_0} (-1)^i 
    \psi_{\vec{\gamma}^i\delta} \right)
  \]
  is locally semi-constant, and hence $\pi^{C_\delta}_n \circ f_{\vec{\gamma}}$ 
  is also locally semi-constant. 
  Moreover, by Corollary \ref{cor: resh_dim}, 
  we know that $\resh_{n_0}^{\mc{C}_\delta}(\cdot, \vec{\gamma}) = 
  \pi^{C_\delta}_n \circ \resh_n^{\mc{C}}(\cdot, \vec{\gamma},\delta)$.
  Since $\delta \in E$, we know that 
  $\pi^{C_\delta}_n \circ \psi_{\vec{\gamma}} = 0$; putting this together, 
  it follows that 
  \begin{align} \label{eqn: triv_eqn}
    \pi^{C_\delta}_n \circ f_{\vec{\gamma}} = \resh_{n_0}^{\mc{C}^\delta}(\cdot, 
    \vec{\gamma}) - \sum_{i < n_0} (-1)^i \pi^{C_\delta}_n \circ 
    \psi_{\vec{\gamma}^i\delta}
  \end{align}
  is locally semi-constant for all $\vec{\gamma} \in [C_\delta]^{n_0}$
  
  Suppose first that $n_0 = 1$, and let $\sigma = \pi^{C_\delta}_n \circ 
  \psi_\delta$. Then equation (\ref{eqn: triv_eqn}) implies that $\sigma$ 
  trivializes $\Phi(\resh^{\mc{C}^\delta}_{n_0})$, contradicting the fact that 
  $\Phi(\resh^{\mc{C}^\delta}_{n_0})$ is nontrivial.
  
  Suppose next that $n_0 > 1$. For each $\vec{\beta} \in [C_\delta]^{n_0-1}$, 
  let $\sigma_{\vec{\beta}} = \pi^{C_\delta}_n \circ \psi_{\vec{\beta}\delta}$. 
  Then equation (\ref{eqn: triv_eqn}) implies that the family 
  $\langle \sigma_{\vec{\beta}} \mid \vec{\beta} \in [C_\delta]^{n_0-1} \rangle$ 
  trivializes $\Phi(\resh^{\mc{C}^\delta}_{n_0})$, again yielding the desired 
  contradiction.
\end{proof}

We end this section with the promised discussion of the prospects for 
deriving nontrivial coherent $n$-dimensional families of functions mapping 
into $\bb{Z}$ via the machinery of higher-dimensional walks. As mentioned 
in the introduction, it seems somewhat unlikely that this can be done in a 
canonical and uniform way for $n > 1$. The reason for this is the following 
fact, from a forthcoming work of Bergfalk, Zhang, and the first author.

\begin{fact}[\cite{coh_of_ord_book}]
  Suppose that $n < \omega$, $2^{\omega_n} = \omega_{n+1}$, and 
  \[
    \Phi = \langle \varphi_{\vec{\gamma}}:\gamma_0 \ra \bb{Z} \mid 
    \vec{\gamma} \in [\omega_{n+2}]^{n+2} \rangle
  \]
  is coherent mod finite. Then there is a cardinal-preserving forcing 
  extension in which $\Phi$ is trivial mod finite.
\end{fact}

An analogue of this fact should also hold for coherence and triviality 
modulo locally semi-constant functions.
To see the relevance of this fact, suppose that we are in a model of 
$\ZFC$ in which $n_0 < \omega$, $2^{\omega_{n_0}} = \omega_{n_0+1}$, and 
$\mc{C}$ is an order-type-minimal $(n_0+2)$-$C$-sequence on 
$\omega_{n_0+2}$. In particular, letting $n = n_0+2$, $\mc{C}$ is trivially a 
$\Sq{n}^s(\omega_n)$-sequence. Now suppose that 
\[
    \Phi_{\mc{C}} = \langle \varphi_{\vec{\gamma}}:\gamma_0 \ra \bb{Z} \mid 
    \vec{\gamma} \in [\omega_n]^{n} \rangle
\]
is a coherent $n$-family of functions derived in a sufficiently canonical 
and uniform way from performing $n$-dimensional walks using $\mc{C}$. 
By the above fact, there is a cardinal-preserving forcing extension $V'$ in 
which $\Phi_{\mc{C}}$ is trivial. However, in $V'$, $\mc{C}$ remains an 
order-type-minimal $n$-$C$-sequence and hence a $\Sq{n}^s(\omega_n)$-sequence. Moreover, as long as the passage from $\mc{C}$ to $\Phi_{\mc{C}}$ is 
sufficiently absolute, as is the case in all current families derived from 
the walks machinery, $\Phi_{\mc{C}}$ as computed in $V'$ will equal 
$\Phi_{\mc{C}}$ as computed in $V$. But $\Phi_{\mc{C}}$ is trivial in 
$V'$, suggesting that there cannot be a simple ``recipe'' for converting 
a $\Sq{n}^s(\lambda)$-sequence into a nontrivial coherent $n$-family of 
functions mapping into $\bb{Z}$. This observation partially justifies our 
moving from the function $\rho_2^n$ to the richer function $\resh_n$, which, 
for $n \geq 2$, maps not into $\bb{Z}$ but into the larger group 
$\bigoplus_{[\lambda]^{n-1}}\bb{Z}$.

\printbibliography
\end{document}